\documentclass{elsarticle}
\usepackage{german,latexsym,amsfonts,amssymb,amsbsy,amsmath}
\usepackage{epsfig,theorem}
\usepackage{color}
\usepackage[all]{xy}
\usepackage{url}
\usepackage{tikz}
\usetikzlibrary{arrows,matrix}

\def\Top{\mathcal{T}\! op}
\def\Cat{\mathcal{C}at}
\def\SSets{\mathcal{SS}ets}
\def\Sets{\mathcal{S}ets}
\def\Ass{\mathcal{A}ss}
\def\Com{\mathcal{C}om}
\def\scA{\mathcal{A}}
\def\scB{\mathcal{B}}
\def\scD{\mathcal{D}}
\def\scI{\mathcal{I}}
\def\scM{\mathcal{M}}
\def\scN{\mathcal{N}}
\def\scC{\mathcal{C}}
\def\scP{\mathcal{P}}
\def\scS{\mathcal{S}}
\def\scT{\mathcal{T}}

\def\mlf{\rm mlf}

\def\In{\mathop{\rm In}}

\def\Opr{\mathop{\mathcal{O}\rm pr}\nolimits} %
\def\Id{\mathop{\rm Id}}

\makeatletter
\@addtoreset{equation}{section}

\makeatother

\theoremstyle{change}
\theorembodyfont{\rm}
\newtheorem{prop}{Proposition:}[section]
\newtheorem{theo}[prop]{Theorem:}
\newtheorem{lem}[prop]{Lemma:}
\newtheorem{defi}[prop]{Definition:}
\newtheorem{coro}[prop]{Corollary:}
\newtheorem{conj}[prop]{Conjecture:}
\newtheorem{nconj}[prop]{Naive Conjecture:}
\newtheorem{conv}[prop]{Convention:}
\newtheorem{exam}[prop]{Example:}
\newtheorem{leer}[prop]{}
\newtheorem{rema}[prop]{Remark:}
\newtheorem{nota}[prop]{Notation:}
\newenvironment{proof}{\par\noindent\textbf{Proof}}{\hfill\ensuremath{\Box}
}

\begin{document}
\title{\textbf{An Additivity Theorem for the Interchange of $\mathbf{E_n}$ Structures}}
\author{Z. Fiedorowicz}
\address{Department of Mathematics, The Ohio State University\\ Columbus, OH 43210-1174, USA}
\ead{fiedorow@math.osu.edu}

\author{R.M.~Vogt}
\address{Universit\"at Osnabr\"uck, Fachbereich Mathematik/Informatik\\ Albrechtstr. 28a, 49069 Osnabr\"uck, Germany}
\ead{rvogt@uos.de}

\author{\textit{Dedicated to the memory of Roland Schw"anzl}}

\begin{abstract}
Let $\mathcal{A}$ and $\mathcal{B}$ be operads and let $X$ be an object with an
$\mathcal{A}$-algebra and a $\mathcal{B}$-algebra structure. These structures are said to interchange if
each operation $\alpha: X^n\to X$ of the $\mathcal{A}$-structure is a homomorphism with respect to the
$\mathcal{B}$-structure and vice versa. In this case the combined structure is codified by the tensor
product $\mathcal{A}\otimes \mathcal{B}$ of the two operads. There is not much known about $\mathcal{A}\otimes \mathcal{B}$ in general,
because the analysis of the tensor product requires
the solution of a tricky word problem. 

Intuitively one might expect that the tensor product of an $E_k$-operad with an $E_l$-operad (which encode
the multiplicative structures of $k$-fold, respectively $l$-fold loop spaces) ought to be an $E_{k+l}$-operad.
However, there are easy counterexamples to this naive conjecture.
In this paper we essentially
solve the word problem for the nullary, unary, and binary operations of the tensor product of arbitrary
topological operads and show that the
tensor product of a cofibrant $E_k$-operad with a cofibrant $E_l$-operad is an $E_{k+l}$-operad.
It follows that if $\mathcal{A}_i$ are $E_{k_i}$ operads for $i=1,2,\ldots,n$ then $\mathcal{A}_1\otimes\ldots\otimes
\mathcal{A}_n$ is at least an $E_{k_1+\ldots +k_n}$ operad, i.e. there is an $E_{k_1+\ldots +k_n}$-operad 
$\mathcal{C}$ and a map of operads
$\mathcal{C}\to \mathcal{A}_1\otimes\ldots\otimes \mathcal{A}_n$.
\end{abstract}
\maketitle

\section{Introduction}\label{sec1}

Two algebraic structures are said to interchange if the structure maps of
one structure are homomorphisms with respect to the second structure and
vice versa. A precise definition will be given below.

Interchange features are abundant in algebra, category theory, algebraic
topology and related fields. A well-known exercise in introductory algebra
is to show that two interchanging group structures coincide and are abelian
group structures. In iterated loop space theory interchanging loop
structures provide rich algebraic structures. In the theory of
$n$-categories the interchange of the various category structures is of
central interest, and one of the main problems in the search for the
``best'' notion of a weak $n$-category is the determination of the
``right'' notion of interchange.

In the present paper we address the interchange of $E_n$
structures extending a program originally suggested by J.M. Boardman
\cite{Bo} in the context of a recognition principle for $n$-fold loop
spaces. This program has experienced a revival of interest for various
reasons: In connection with the research on weak $n$-categories,
$A_\infty$-categories, and Segal categories the question of the uniqueness
of $n$-fold delooping machines for $1\le n\le \infty$ has become of
interest again. The solution offered in \cite{Dunn2} has gaps. The analysis
of the delooping problem brings up the question of
interchanging $E_1$ structures, usually called $A_\infty$ structures.

Kontsevich's generalization of Deligne's Hochschild cohomology conjecture
to algebras over the little $n$-cubes operad, and the dual problem about
the multiplicative structure of (topological) Hochschild homology lead
directly into interchange problems of $E_n$ structures and $A_\infty$
structures. 

A space $X$ has two interchanging structures encoded by operads $\scB$ and
$\scC$ iff it has a structure encoded by the tensor product $\scB\otimes
\scC$. Our main result is

\textbf{Additivity Theorem:} If $\scB$ is an $E_k$ operad and $\scC$ an $E_l$ operad
and both are cofibrant, then $\scB\otimes\scC$ is an $E_{k+l}$ operad.

Our choice of weak equivalences are maps of operads $f:\scB\to \scC$ such that each map
$f(n): \scB(n)\to \scC(n)$ is a $\Sigma_n$-equivariant homotopy equivalence, i.e. we work with the
Str{\o}m model structure on the underlying collections. It is not known
that the category $\Opr$ of operads carries a model structure with these weak equivalences;
nevertheless one can define the term ``cofibrant'' by a lifting property and 
there is a good cofibrant replacement 
functor (see \cite{V} and the discussion in \cite[Section 8.1]{BM}) .

If we replace the Str{\o}m model structure on the category of collections
by the Quillen model structure, $\Opr$ has a model structure \cite{BM0} and the Additivity Theorem
also holds in this context.

If we drop the cofibrancy condition in the theorem, then $\scB\otimes\scC$ is at least
$E_{k+l}$, which means that there is a $E_{k+l}$ operad $\scD$ and a 
map of operads $\scD \to \scB\otimes\scC$. But there might be an $E_r$
operad $\scD'$ with $r>k+l$ and a map of operads $\scD' \to \scB\otimes\scC$.

Versions of the Additivity Theorem have been known in special cases and have been addressed
in other contexts. E.g. Dunn \cite{Dunn} proved that the $E_k$ operads $\scD_k$ of decomposable little $k$-cubes satisfy
 additivity  on the nose:
$$\scD_k\otimes \scD_l\cong \scD_{k+l}.$$
From this he deduced that the $n$-fold tensor product $\scC_1\otimes\ldots\otimes \scC_1$ of the little
$1$-cubes operad is an $E_n$ operad. In his thesis \cite{Brink} Brinkmeier extended this result to $n$-fold tensor products of
little $k$-cubes operads: $\scC_{k_1}\otimes\ldots\otimes\scC_{k_n}$ is an $E_{k_1+\ldots +k_n}$ operad.
These results do not imply the general Additivity Theorem because the tensor product is \textbf{not} homotopy invariant, i.e.
it does not preserve weak equivalences of operads. 


In his book project ``Higher Algebra'' Lurie introduces a tensor product of $\infty$-operads and proves a version of Brinkmeier's result
in his setting \cite[Section 5.1.2]{Lurie}. This tensor product seems to be related to ours but the precise relationship has not been worked out yet.

Moerdijk and Weiss extended the tensor product of operads to dendroidal sets  \cite{Moer}, \cite{MW1},
\cite{MW2}, \cite{W}, and 
suggested an investigation of an additivity theorem for $n$ copies of the dendroidal set $N_d(\scA ss)$ associated 
with the operad $\scA ss$ of monoid structures. But while $\scA ss\otimes \scA ss\cong \scC om$ (see \ref{3_7}
for more details), the structure of the dendroidal tensor product $N_d(\scA ss)\otimes N_d(\scA ss)$ is not clear.

Recent work of Cisinski and Moerdijk \cite{CM1}, \cite{CM2}, \cite{CM3} and Heuts, Hinich and Moerdijk \cite{HHM}  relates the dendroidal sets world with that of Lurie.
We hope that the methods presented in our paper will be of help in understanding the precise relation of their tensor products
with the interchange of structures.

Our tensor product of operads, which encodes interchanging structures on the nose,
 is quite elusive, and during our work on this paper we
often fell into traps. In Section 3 we will list some surprising examples which will
give an indication that we have to solve a non-trivial word problem.

The strategy of the proof of the Additivity Theorem is as follows:
For a particular choice of universal $E_k$ and $E_l$ operads $\scB=W|\mathcal{NM}_k|$ and $\scC=W|\mathcal{NM}_l|$
we cover the tensor product $\scB\otimes \scC$
by suitable contractible subsets and relate the
nerve of this covering to $\scM_{k+l}$. Here $\scM_k$ is the $\Cat$ operad which was
analyzed in \cite{BFSV} and shown to be an $E_{k}$ operad, $\mathcal{N}$ is the nerve
functor, and $W$ is a cofibrant replacement functor.

The paper is organized as follows: In Sections 2 and 3 we recall the definitions of
$A_\infty$ and $E_k$ operads, of interchange and of the tensor product of operads. As
mentioned above, Section 3 also contains some surprising examples of tensor products. Our
main results and a recollection of the operad $\mathcal{M}_k$ follow in
Section 4. In Section 5 we explain the strategy of the proof of the Additivity Theorem in
greater detail. Section 6 deals with the unary and binary operations in the tensor
product of arbitrary operads. 

The forgetful functor $U$ from reduced operads to topological monoids has a right
adjoint $R$, and we call a $\Sigma$-free operad $\scB$ \textit{axial} if the unit
$\scB(n)\to RU(\scB)(n)$ is a closed cofibration for each $n$.
This property is crucial in our proof. Therefore
we study the adjoint pair $U$ and $R$ in Section 7. In Section 8 we recall the
definition of the cofibrant replacement functor $W$.
We prove the axiality of $W|\mathcal{NM}_k|$ and related properties used
in the proof of the main result. In the remaining sections we define a cover of
$(W|\mathcal{NM}_k|  \otimes W|\mathcal{NM}_l|)(m)$ by contractible cells, describe
the diagram of intersections of these cells, and construct the homotopy equivalences mentioned above.

We dedicate this paper to the memory of Roland Schw\"anzl \cite{S}, our dear
friend and collaborator, who started off with
us on this project but tragically succumbed to a fatal illness at an early stage of its development.

\textit{Acknowledgment:} We are indebted to Clemens Berger for a number of clarifying conversations and 
for pointing out some errors in an earlier version of this paper, and to the referee for helpful comments
on its organization.

\section{$E_n$ structures and $A_\infty$ structures}\label{sec2} 
$E_n$ structures are closely related to the algebraic structure of an
$n$-fold loop space. They are best described using operads.

\begin{defi}\label{2_1}
Let $\scS$ be a symmetric monoidal category with multiplication $\times$.
An \textit{operad} $\mathcal{B}$ in $\scS$ is a collection $\{\mathcal{B}(k)\}_{k\ge 0}$ of
objects equipped with symmetric group actions
$\mathcal{B}(k)\times \Sigma_k\to\mathcal{B}(k)$, composition maps
$$
\mathcal{B}(k)\times(\mathcal{B}(j_1)\times\ldots\times \mathcal{B}(j_k))
\to \mathcal{B}(j_1+\ldots+j_k),
$$
and a unit $id\in\mathcal{B}(1)$ satisfying the appropriate equivariance,
associativity and unitality conditions - see \cite{May1} for details.
\end{defi}
\begin{rema}\label{2_1a}
Throughout this paper $\scS$ will be $\Cat$, the category of small categories,
$\Sets$, the category of sets,
$\SSets$, the category of simplicial sets, or $\Top$, the category of $k$-spaces and continuous maps.\\
In all three cases the symmetric monoidal structure is given by the product and the categories are complete and cocomplete.
Moreover, they are self-enriched and the product distributes over the coproduct. In particular, we can define the 
$\scS$-endomorphism operad $\mathcal{E}_X$ of an object $X$ in $\scS$ by $\mathcal{E}_X(n)=\scS(X^n,X)$ with the obvious $\Sigma_n$-action and the
obvious composition maps and unit.
\end{rema}

\begin{defi}\label{2_1b}
A topological operad will be called \textit{well-pointed} if  $\{id\}\subset\mathcal{B}(1)$ is a
closed cofibration. 
\end{defi}

\begin{leer}\label{2_2}
While (\ref{2_1}) is the most common definition of an operad, it is often
helpful to think of it in the following equivalent way, which is the original
version from \cite{BV1}. An operad $\mathcal{B}$ in a symmetric monoidal category
$\scS$ of Remark \ref{2_1a} is an $\scS$ enriched symmetric monoidal category $(\mathcal{B},\oplus,0)$ such  that 
\begin{enumerate}
\item[(i)] $ob\:\mathcal{B}=\mathbb{N}$ and $m\oplus n=m+n$
\item[(ii)] $\oplus$ is a strictly associative $\scS$-functor with strict unit 0
\item[(iii)] $\coprod\limits_{r_1+\ldots+r_n=k}\mathcal{B}(r_1,1)
\times\ldots\times \mathcal{B}(r_n,1)\times_{\Sigma_{r_1}
\times\ldots\times\Sigma_{r_n}} \Sigma_k\to\mathcal{B}(k,n)$
$$
\xymatrix{
((f_1,\ldots,f_n),\tau) \ar@{|->}[rr] 
&& (f_1\oplus\ldots\oplus f_n)\circ\tau
}$$
is an isomorphism in $\scS$. 
\end{enumerate}
In the topological case ``well-pointed'' translates to the assumption that $\{ id \}\subset \mathcal{B}(1,1)$
is a closed cofibration.

Using (iii),
each operad in the sense of  (\ref{2_1}) determines a category as in
(\ref{2_2}). Conversely, given a category as in (\ref{2_2}) we obtain an
operad by taking the collection $\{\mathcal{B}(k,1)\}_{k\ge 0}$.
\end{leer}

For some inductive arguments we will use the following
blown-up version of (\ref{2_2}):

\begin{leer}\label{2_2a}
Each operad $\mathcal{B}$ gives rise to an
$\scS$ enriched symmetric monoidal category $(\mathcal{B},\oplus,0)$ defined
by
\begin{enumerate}
\item[(i)] $ob\:\mathcal{B}=$ \{totally ordered finite sets\} and 
$S\oplus T=S\amalg T$, the ordered disjoint union.
\item[(ii)] $\oplus$ is a strictly associative $\mathcal{S}$-functor with strict unit 
$0=\emptyset$
\item[(iii)] $\mathcal{B}(S,\{t\})=\mathcal{B}(|S|)$, where $|S|$ is the
cardinality of $S$.
\item[(iv)] $\coprod\limits_{S_1\amalg\ldots\amalg S_n=U}
\mathcal{B}(S_1,\{t_1\})
\times\ldots\times \mathcal{B}(S_n,,\{t_n\})\times_{\Sigma(S_1)
\times\ldots\times\Sigma(S_n)} \Sigma(U)\to\mathcal{B}(U,T)$
$$
\xymatrix{
((f_1,\ldots,f_n),\tau) \ar@{|->}[rr] 
&& (f_1\oplus\ldots\oplus f_n)\circ\tau
}$$
is an isomorphism in $\mathcal{S}$, where $\Sigma(U)$ is the permutation group of (the underlying
set of) U and $T=\{t_1<\ldots<t_n\}$.
\end{enumerate}
\end{leer}
We will also find it convenient to use this blown-up version in operad notation; e.g.
$\mathcal{B}(S)$ stands for $\mathcal{B}(|S|)$.

\begin{defi}\label{2_3}
Let $\mathcal{B}$ and $\mathcal{C}$ be  $\scS$-operads.
\begin{enumerate}
\item[(1)] $\mathcal{B}$ is called $\Sigma$\textit{-free} if the $\Sigma_n$-action on $\mathcal{B}(n)$
is free for each $n$ in the cases $\scS=\Cat$, $\Sets$, or $\SSets$. If $\scS=\Top$ we require that
$\mathcal{B}(n)\to\mathcal{B}(n)/\Sigma_n$ is a numerable principal
$\Sigma_n$-bundle for each $n$.
\item[(2)] An \textit{operad map} $\mathcal{B}\to\mathcal{C}$ is a collection of
equivariant maps $\mathcal{B}(n)\to\mathcal{C}(n)$ in $\scS$, compatible with the
operad structure.
\item[(3)] A $\mathcal{B}$\textit{-structure} on an object $X$ in $\scS$ is an operad map
$\mathcal{B}\to\mathcal{E}_X$ into the endomorphism operad $\mathcal{E}_X$
of $X$. We say that $\mathcal{B}$ \textit{acts on} $X$, or that $X$ is a
$\mathcal{B}$-algebra; if $\scS=\Top$ we also call $X$ a $\mathcal{B}$-space. 
\item[(4)] An operad map is called a \textit{weak equivalence} if each map
$\mathcal{B}(n)\to\mathcal{C}(n)$ is an equivariant homotopy equivalence (in $\Cat$ or $\SSets$ this means that
each map is an equivariant homotopy equivalence after applying the classifying space functor, respectively the topological
realization functor). 
\item[(5)] Two operads are called \textit{equivalent} if there is a finite chain of weak
equivalences connecting them.
\item[(6)] A topological operad is called $E_n$ if it is equivalent to the little $n$-cubes
operad~$\mathcal{C}_n$. In particular, it is $\Sigma$-free, because $\mathcal{C}_n$ is $\Sigma$-free.
\item[(7)] An $A_\infty$  operad is another term for an $E_1$ operad.
\item[(8)] An operad $\scB$ is called \textit{reduced} if $\scB(0)=\{0\}$. We denote
the categories of operads and reduced operads by $\mathcal{O}pr$ and
$\mathcal{O}pr_0$ respectively. Here we use the same notation for each of our categories $\scS$; it will be clear
from the context which $\scS$ we mean.
\end{enumerate}
\end{defi}

Recall that $\mathcal{C}_n(k)$ is the space of $k$-fold configurations
of subcubes of the unit cube $I^n$, whose axes are parallel to those of
$I^n$ and whose interiors are disjoint. Any $n$-fold loop space $\Omega^nY$
has a natural action by this operad. Conversely, each path-connected space with a
$\mathcal{C}_n$-structure is of the weak homotopy type of an $n$-fold loop
space (c.f. \cite{BV1}, \cite{BV2}, and \cite{May1} for details).

\begin{rema}\label{2_3b}
As pointed out in the introduction our notion of $E_n$ operads is not shared by all authors dealing with operads.
Our basis is Str{\o}m's model structure on $\Top$ \cite{Strom}, while most authors prefer the cofibrantly generated Quillen
model structure. We will call our set-up the \textit{Str{\o}m environment} and - if not stated otherwise -
we will work in this environment. \\
In the \textit{Quillen environment}
the weak equivalences are operad maps $f:\scB\to \scC$ such that each map $f(n):\scB(n)\to \scC(n)$ is a weak
homotopy equivalence of spaces; an operad $\scB$ is $\Sigma$-free if for each $n$ the right action of $\Sigma_n$ on
 $\scB(n)$ is free; and $\scB$ is called $E_n$ if there is a finite chain of Quillen weak equivalences of $\Sigma$-free
operads joining $\scB$ with $\scC_n$. In particular, each $E_n$ operad in the Str{\o}m environment is $E_n$ in the
Quillen environment.
\end{rema}

\begin{defi}\label{2_3a} Let $S$ and $T=\{t_1,\ldots,t_n\}$ be totally ordered finite sets
and let $\scB$ be a reduced operad.
 We define \textit{restriction maps}
$$ -\cap S:\scB(T)\to \scB(S\cap T)$$
by composing with $(\varepsilon_1\oplus\ldots\oplus \varepsilon_n)$, where 
$\varepsilon_i=id$ if $t_i\in S$ and is $0$ otherwise. We will also use
this notation in related situations like products, sums etc. of
$\scB(T)$'s.
\end{defi}

We will also make use of the operads $\Ass$ and $\Com$ which encode the
structures of a monoid and a commutative monoid respectively. By
definition, $\Ass(n)=\Sigma_n$, where $\sigma\in\Sigma_n$ stands for the
operation
$$
(x_1,\ldots,x_n) \mapsto x_{\sigma^{-1}(1)}\cdot\ldots\cdot x_{\sigma^{-1}(n)}
$$
in a monoid. From this the operad data for $\Ass$ can be deduced.\\
$\Com(n)=\{\lambda_n\}$ is a single point. Here $\lambda_n$ stands for the
operation 
$$
\xymatrix{
(x_1,\ldots,x_n) \ar@{|->}[r]
& x_1+\ldots+x_n.
}
$$

An $n$-fold loop space $\Omega^nY$ has $n$ interchanging loop space
structures. Since $\mathcal{C}_n$ acts naturally on $\Omega^nY$, these $n$
interchanging structures should somehow be encoded in
$\mathcal{C}_n$. Before we can make this precise we need to formally define
the notion of interchange.

\section{Interchange}\label{sec3}
Since this paper is about interchange of topological operad structures the
operads in this section will be topological operads. The interested reader
can easily make the necessary adjustments for our other categories $\scS$.

Let $X$ be a space with actions of operads $\mathcal{B}$ and $\mathcal{C}$.
Then the coproduct $\mathcal{B}\amalg\mathcal{C}$ of $\mathcal{B}$
and $\mathcal{C}$ in the category of operads acts on $X$.

\begin{defi}\label{3_1}
We say that the $\mathcal{B}$- and $\mathcal{C}$-actions on $X$ \textit{interchange}
if each operation $\beta:X^k\to X$, $\beta\in\mathcal{B}(k)$, in the
$\mathcal{B}$-structure is a homomorphism of $\mathcal{C}$-spaces, and vice
versa. Explicitly, this means that for each $\beta\in\mathcal{B}(k)$ and
each $\gamma\in\mathcal{C}(l)$ the square 
$$
\xymatrix@=40pt{
(X^k)^l \ar[r]^{\tau_{k,l}}_{\cong} \ar[d]^{\beta^l} 
& (X^l)^k \ar[rr]^{\gamma^k} 
&& X^k \ar[d]^\beta
\\
X^l \ar[rrr]^\gamma &&& X
}
$$ 
commutes, where $\tau_{kl}\in\Sigma_{k\cdot l}$ is the permutation which
reorders the coordinates of $X^{k\cdot l}$ from lexicographical to reserve
lexicographical order.
\end{defi}

Note that the two composites $\beta\circ\gamma^k\circ\tau_{k,l}$ and
$\gamma\circ \beta^l$ are elements in
$(\mathcal{B}\amalg\mathcal{C})(k\cdot l)$. If $k>0$ and $l>0$, we may
interpret the interchange condition as follows.  Given a $k\times l$ array 
$\{x_{ij}\}_{i=1,j=1}^{i=k,j=l}$ of elements
of $X$, we can apply $\beta$ to the columns of the array and then $\gamma$
to the resulting products.  Alternatively we can apply $\gamma$ to each row of
the array, then $\beta$ to the resulting products. The interchange condition
states that we obtain the same final result either way.

\begin{defi}\label{3_2}
The \textit{tensor product} $\mathcal{B}\otimes\mathcal{C}$ of operads $\mathcal{B}$
and $\mathcal{C}$ is obtained from the coproduct
$\mathcal{B}\amalg\mathcal{C}$ by factoring out the interchange relation
(\ref{3_1}).
\end{defi}

A more detailed description of $\mathcal{B}\amalg\mathcal{C}$ and
$\mathcal{B}\otimes\mathcal{C}$  will be given in the beginning of Section 5.

\begin{rema}\label{3_2a} In the case when $\beta\in\mathcal{B}(1)$,
the interchange relation implies
$$\beta\circ\gamma=\gamma\circ\beta^l.$$
This relation and the dual relation when $\gamma\in\mathcal{C}(1)$ are
called \textit{unary interchanges} and will
play an important role in our analysis of the tensor product of operads.
More generally we shall refer to interchanges involving $\beta\in\mathcal{B}(k)$
and $\gamma\in\mathcal{C}(l)$ as $(k,l)$-\textit{interchanges}.
\end{rema}

By definition, a space $X$ admits a $\mathcal{B}$-structure interchanging
with a $\mathcal{C}$-struc\-ture iff it admits a
$(\mathcal{B}\otimes\mathcal{C})$-structure.

The following two results about tensor products of operads are extant in the
literature:

\begin{prop}\label{3_3} (Dunn \cite{Dunn})
There is a canonical weak equivalence of operads
$$
\xymatrix{
\mathcal{C}_1\otimes \ldots\otimes \mathcal{C}_1\ar[r] &\mathcal{C}_n
& (n \textrm{  tensor factors})
}
$$
So we ``recover'' the $n$ interchanging loop space structures in 
$\mathcal{C}_n$. 
\end{prop}

\begin{prop}\label{3_3a}(\cite{BFV}) The operad $\Ass\otimes\mathcal{C}_k$
is equivalent to $\mathcal{C}_{k+1}$\end{prop}

In view of these two results, we might conjecture that:

\begin{nconj}\label{3_4}
If $\mathcal{B}$ is $E_k$ and $\mathcal{C}$ is $E_l$, then
$\mathcal{B}\otimes\mathcal{C}$ is $E_{k+l}$.
\end{nconj}

Unfortunately, the situation is not that simple: The functor
$\mathcal{B}\otimes$- does not preserve weak equivalences, and the
structure of $\mathcal{B}\otimes\mathcal{C}$ is anything but clear.
In general, one has to solve a substantial word problem. For instance
unary operations in $\mathcal{B}(1)$ and $\mathcal{C}(1)$ may be factored
in many different ways and these factors may then be redistributed in a
complicated way using the unary interchange relations mentioned in Remark \ref{3_2a}.
The following examples illustrate some rather surprising consequences
that these and other interchange relation may imply.

\begin{prop}\label{3_5}(\cite[Lemma 2.23]{BV2})
Let $\mathcal{B}$ and $\mathcal{C}$ be operads such that
$\mathcal{B}(0)\neq \emptyset\neq\mathcal{C}(0)$ and let 
$\mathcal{B}'$ and $\mathcal{C}'$ be their universal quotients with
exactly one nullary operation.
Then \\
(1)$(\mathcal{B}\otimes\mathcal{C})(0)$ contains exactly one element.\\
(2) $\mathcal{B}\otimes\mathcal{C}\cong \mathcal{B}'\otimes\mathcal{C}'$
\end{prop}

\begin{proof}:
Given $\beta \in\mathcal{B}(0)$ and $\gamma\in\mathcal{C}(0)$, the
interchange relation gives
$$
\xymatrix@=40pt{
\ast=(X^0)^0 \ar[r]^(.55){\tau_{0,0}=id} \ar[d]^{\beta^0=id}
& (X^0)^0 \ar[rr]^{\gamma^0=id} 
&& X^0=\ast \ar[d]^\beta
\\
\ast=X^0 \ar[rrr]^\gamma &&& X
}
$$
Hence $\beta=\gamma$ in $\mathcal{B}\otimes\mathcal{C}$.

The second part is an immediate consequence of the first part.
\end{proof}

\begin{prop}\label{3_6}
Let $\mathcal{B}$ and $\mathcal{C}$ be operads such that
$\mathcal{B}(1)=\{id\}$, $\mathcal{C}(1)=\{id\}$ and $\mathcal{B}(0)$,
$\mathcal{B}(2)$, $\mathcal{C}(0)$, $\mathcal{C}(2)$ are not empty. Then
$\mathcal{B}\otimes\mathcal{C}\cong \Com$.
\end{prop}

\begin{proof}
Since $\mathcal{B}(2)$ and $\mathcal{C}(2)$ are not empty, so are 
$\mathcal{B}(k)$ and $\mathcal{C}(k)$ for $k\geq 2$.
Let $\lambda \in(\mathcal{B}\otimes\mathcal{C})(0)$ be the unique nullary
operation and let $X$ be a $(\mathcal{B}\otimes\mathcal{C})$-space. For
$\beta\in\mathcal{B}(k)$ and $\gamma\in\mathcal{C}(k)$ consider the
following $(k\times k)$-array of points in $X$, where $\ast=\lambda(\ast)$
$$
\begin{array}{cccc}
x_1  & \ast & \ldots & \ast\\
\ast & x_2  & \ldots & \ast\\
     & \ldots \\
\ast & \ast & \ldots & x_k
\end{array}
$$
Let $\beta$ act horizontally on this array and $\gamma$ vertically. By the
interchange relation, the results should be the same if we compute the
action first horizontally, then vertically or vice versa. Now
$$
\beta(*,\dots,*,x_i,*,\dots,*)=\beta\circ(\lambda,\dots,\lambda,id,\lambda,\dots,\lambda)(x_i)=x_i
$$
for $i=1,..,k$, since $\beta\circ(\lambda,\dots,\lambda,id,\lambda,\dots,\lambda)\in\scB(1)=\{id\}$.
Hence computing the action first horizontally and then vertically 
gives $\gamma(x_1,\ldots,x_k)$. Similarly computing it the other way we get
$\beta(x_1,\ldots,x_k)$, so that $\beta=\gamma$ in
$\mathcal{B}\otimes\mathcal{C}$ for all $\beta$ and $\gamma$.
\end{proof}

\begin{coro}\label{3_7}
\begin{itemize}
\item[(i)] Suppose $\mu_1,\mu_2:X^2\longrightarrow X$ are two (not necessarily
associative) $H$-space structures on a topological space $X$.  Suppose that
$\mu_1$ and $\mu_2$ both have the same strict 2-sided unit $*$ and that
$\mu_1$ and $\mu_2$ satisfy the $(2,2)$-interchange relation
$$(\mbox{E-H})\qquad\qquad \mu_1\left(\mu_2(x_1,x_2),\mu_2(x_3,x_4)\right)
= \mu_2\left(\mu_1(x_1,x_3),\mu_1(x_2,x_4)\right)$$
for all $x_1,x_2,x_3,x_4\in X$.  Then $\mu_1=\mu_2$ and this multiplication is
strictly associative and commutative.
\item[(ii)] More generally, if $\mathcal{B}(2)\times X^2\longrightarrow X$,
$\mathcal{C}(2)\times X^2\longrightarrow X$, are continuous nonempty families
of (not necessarily associative) multiplications on a topological space $X$, which
have a common 2-sided unit $*$, and which satisfy the $(2,2)$-interchange
relations $(\mbox{E-H})$ for all $\mu_1\in\mathcal{B}(2)$ and  $\mu_2\in\mathcal{C}(2)$,
then $\mathcal{B}(2)=\mathcal{C}(2)=\{\mu\}$ and $\mu$ defines a commutative
monoid structure on $X$.
\item[(iii)] $\Ass\otimes\Ass\cong 
\Ass\otimes\Com \cong \Com\otimes\Com \cong \Com $
\end{itemize}
\end{coro}

Corollary \ref{3_7}(i) with the additional hypotheses that $\mu_1$ and $\mu_2$
are both associative is well known to topologists as Eckmann-Hilton interchange \cite{EH}.
Less known is the fact that the associativity hypotheses are superfluous.

\begin{proof}
We first prove (i).  Let $\mathcal{B}$ be the operad generated by $\mu_1$ and the 2-sided unit.
Specifically $\mathcal{B}(0)=\{*\}$, $\mathcal{B}(1)=\{id\}$, $\mathcal{B}(2)=\{\mu_1,\mu_1\circ(1\ 2)\}$,
and for $k>2$, $\mathcal{B}(k)$ consists of all $k$-fold iterates of $\mu_1$ and their permutations.
(For $k\ge 2$, the elements of $\mathcal{B}(k)$ are in 1-1 correspondence with planar  binary
trees with $k$ labeled inputs and no stumps.)
Similarly let $\mathcal{C}$ be the operad generated by $\mu_2$ and the 2-sided unit. Then $\mathcal{B}$
and $\mathcal{C}$ both act on $X$.  We must show that these operad actions interchange.

The unary interchanges between  $\mathcal{B}$ and $\mathcal{C}$ hold trivially since
 $\mathcal{B}(1)=\{id\}=\mathcal{C}(1)$. Similarly the nullary $(k,0)$- and $(0,l)$-interchanges
 follow from  $\mathcal{B}(0)=\{*\}=\mathcal{C}(0)$. Finally we obtain that
all $(2,2)$-interchanges hold by applying appropriate permutations to the interchange
relation $(\mbox{E-H})$.  Thus it remains to show that $(k,l)$-interchanges hold when $k>2$ or $l>2$.

Let $\beta_2\in\mathcal{B}(2)$.  Then $\beta_2$ defines an $H$-space structure on $X$ and
on all products $X^l$ (via coordinatewise multiplication). From the already established $(2,2)$-interchange
 relations, $\mu_2$ is a homomorphism $X^2\longrightarrow X$ with respect to $\beta_2$.
Thus iterates of $\mu_2$ are composites of products of $\beta_2$-homomorphisms and thus are also homomorphisms.
Clearly permutations of $X^l$ are also homomorphisms.  This establishes all $(2,l)$-interchanges
for all $l>2$.

Now fix an element $\gamma_l\in\mathcal{C}(l)$.  Then $\gamma_l$ defines an $l$-ary operation on
$X$ and coordinatewise on all products $X^k$. From the already established $(2,l)$-interchange
relations, we obtain that $\mu_1$, its iterates and permutations thereof determine homomorphisms of these
$\gamma_l$ structures.  This establishes all $(k,l)$-interchanges for $k>2$.

Thus we obtain that the $\mathcal{B}$ and $\mathcal{C}$ actions on $X$ interchange and thus
determine a $\mathcal{B}\otimes\mathcal{C}$ action.  But by \ref{3_6}
$\mathcal{B}\otimes\mathcal{C}\cong \Com$.  Thus $\mu_1=\mu_2$ is strictly
associative and commutative, establishing (i).

Part (ii) is an immediate consequence.  For we can pick $\mu_1\in\mathcal{B}(2)$
and ${\mu_2\in\mathcal{C}(2)}$ and apply (i).  We obtain that $\mu_1=\mu_2$
is associative and commutative.  Since $\mu_1$ and $\mu_2$ are arbitrary, (ii) follows.

Part (iii) is an immediate consequence of \ref{3_6}.
\end{proof}

The same argument as in the proof of \ref{3_7} part (i) establishes the following result.

\begin{prop}
Suppose that two operads $\mathcal{B}$ and $\mathcal{C}$ act on a space $X$.
Suppose that the interchange relations hold between all the generating
elements of $\mathcal{B}$ and $\mathcal{C}$.  Then the actions of 
$\mathcal{B}$ and $\mathcal{C}$ on $X$ interchange.
\end{prop}

\begin{rema}\label{3_8}
The tensor product of $\Ass$ with an operad $\mathcal{B}$ satisfying
$\mathcal{B}(0)=\ast$ has been determined in \cite{BV3}. We will recall
this result in Section 7 and deduce from it the structure of
$\Com \otimes \mathcal{B}$. In the meantime we point out another surprising
consequence of that result.
\end{rema}

\begin{prop}
If the bar construction $B_\bullet M$ on a well-pointed topological monoid $M$
is an $H$-space in the category of simplicial topological spaces, then 
the geometric realization $|B_\bullet M|$ is homotopy equivalent to a
loop space, as an $H$-space.
\end{prop}

\begin{proof}
As is shown in \cite{BV2}, loop space structures are parametrized by an
operad $W\scA ss$, a blown-up version of $\scA ss$ (which we will discuss in Section 8
below). It is also shown there that $H$-space structures are parametrized by
the suboperad $W^{(2)}\scA ss$ of $W\scA ss$ generated by the nullary, unary and binary
operations $W\scA ss(0)$, $W\scA ss(1)$, and $W\scA ss(2)$. Now a simplicial $H$-space structure
on $B_\bullet M$ amounts to the same thing as an $\scA ss\otimes W^{(2)}\scA ss$ structure on $M$.
But according to \cite{BV3}, the induced map
$$\scA ss\otimes W^{(2)}\scA ss\longrightarrow \scA ss\otimes W\scA ss$$
is an isomorphism of operads. Hence $M$ is an $\scA ss\otimes W\scA ss$ algebra and it
follows that $|B_\bullet M|$ has a $ W\scA ss$-structure. Since $|B_\bullet M|$ is a path-connected Dold space \cite[Cor. 5.2]{SV}
this structure admits a homotopy inverse (e.g. \cite[Prop. 3.16]{SV}). It is well-known that a
 $W\scA ss$-space admitting a homotopy inverse is homotopy equivalent through homotopy homomorphisms to a loop space.
\end{proof}

Since $\Ass$ is $E_1$ and $\Com$ certainly is not $E_2$, Corollary \ref{3_7} 
provides
a counterexample to the Naive Conjecture \ref{3_4}. 
In fact, each connected abelian
topological monoid is of the weak homotopy type of an infinite loop space,
indeed equivalent to a product of Eilenberg-MacLane spaces.
So we adjust the conjecture in the following way:

\begin{defi}\label{3_9}
An operad $\mathcal{B}$ is called \textit{at least} $E_n$ if there is an
$E_n$ operad $\mathcal{C}$ and a map of operads $\mathcal{C}\to \mathcal{B}$.
(So any $\mathcal{B}$ space has an $E_n$ structure.)
\end{defi}

\begin{conj}\label{3_10}
If $\mathcal{B}$ is $E_k$ and $\mathcal{C}$ is $E_l$, then
$\mathcal{B}\otimes\mathcal{C}$ is at least $E_{k+l}$.
\end{conj}

\begin{rema}\label{3_11} 
Given maps of operads $\mathcal{A}\to \mathcal{B}\to \mathcal{A}$ with 
$\mathcal{A}$ being $E_n$, then each $\mathcal{B}$-space is an 
$\mathcal{A}$-space and each $\mathcal{A}$-space is a $\mathcal{B}$-space.
$\mathcal{B}$ is at least $E_n$, and one might expect $\mathcal{B}$ 
to be in fact $E_n$. 
This need not be true even if the composite 
$\mathcal{A}\to \mathcal{B}\to \mathcal{A}$ is the identity,
as the following example shows:
$$ \mathcal{C}_n \to \mathcal{C}_n\amalg \mathcal{C}_n\to \mathcal{C}_n
$$
where the first map is the inclusion of the first summand and the second is
the folding map. $\mathcal{C}_n\amalg \mathcal{C}_n$ 
codifies two non-interchanging $\mathcal{C}_n$-structures
and is certainly  not $E_n$: The operad $\pi_0(\mathcal{C}_n)$ 
of path components of $\mathcal{C}_n$ is isomorphic to $\Com$ for $n\geq 2$.
Hence there is a surjection (in fact, a bijection) $\pi_0(\mathcal{C}_n 
\amalg \mathcal{C}_n)\to \Com\amalg \Com$. Now $(\Com\amalg \Com)(2)$ has
two elements, while $\Com (2)$ has only one. So $\mathcal{C}_n
\amalg \mathcal{C}_n$ has too many path components to be $E_n$.
For $n=1$ the same argument works if one replaces $\Com$ by $\Ass$.
\end{rema}

\section{Main results}
\begin{conv}\label{4_1}
In view of Proposition \ref{3_5}, we only work with \textit{reduced} topological operads unless
explicitly stated otherwise. So operad will mean reduced topological operad.
\end{conv}

\begin{defi}\label{4_2}
An operad $\mathcal{B}$ in $\mathcal{O}pr_0$ is called \textit{cofibrant}
 if for
any diagram 
$$
\xymatrix{
&& \mathcal{P} \ar[d]^p
\\
\mathcal{B} \ar[rr]^f && \mathcal{C}
}
$$
of operad maps with $p$ a weak equivalence and each
$p:\mathcal{P}(n)\to\mathcal{C}(n)$ an equivariant fibration, there is a
lift $g:\mathcal{B}\to\mathcal{P}$ such that $p\circ g=f$.
\end{defi}

It is shown in \cite{BV2} that for any $\Sigma$-free well-pointed operad $\mathcal{B}$,
there is a cofibrant operad $W\mathcal{B}$ and an operad equivalence
$W\mathcal{B}\longrightarrow\mathcal{B}$. We call $W\mathcal{B}\longrightarrow\mathcal{B}$
a cofibrant resolution of $\mathcal{B}$. If $\scB$ fails to be well-pointed we add a whisker
to $\scB(1)$ to obtain a well-pointed operad $\scB'$ together with a weak equivalence
$\scB'\to\scB$ (e.g. see \cite[p.167]{May1}) and take $W\scB'\to\scB$ as cofibrant 
resolution of $\mathcal{B}$.

In the Quillen environment $\Opr$ has a model structure \cite{BM0}. As cofibrant replacement
functor in this model structure  for $\Sigma$-free operads in the Quillen sense we can take $WP\scB\to \scB$, 
where $W$ is the construction of 
\cite{BM} or \cite{BV2} (the two agree in $\Top$) and $P$ is the usual $CW$-approximation functor. 

The main result of our paper is the following.

\begin{leer}\label{4_4}\hspace{-1ex}\textbf{Additivity Theorem:}
If $\scB$ is an $E_k$ operad and $\scC$ an $E_l$ operad
and both are cofibrant, then $\scB\otimes\scC$ is an $E_{k+l}$ operad.
This holds in the Str{\o}m and the Quillen environment.
\end{leer}

This verifies the Naive Conjecture \ref{3_4} in the special case of cofibrant
operads. It also verifies the modified Conjecture \ref{3_10}. Indeed it is
equally easy to prove a strengthened version of this conjecture.

\begin{coro} Suppose that $\scA_i$ are at least $E_{k_i}$ for $i=1,2,\dots, n$.
Then the tensor product $\scA_1\otimes\scA_2\otimes\dots\otimes\scA_n$ is at
least $E_{k_1+k_2+\dots+k_n}$.
\end{coro}

\begin{proof} We proceed by induction on $n$. The result is trivially true for
$n=1$. Assume it holds for $n-1$. Then there is an $E_{k_1+k_2+\dots+k_{n-1}}$
operad $\scB$ which maps into $\scA_1\otimes\scA_2\otimes\dots\otimes\scA_{n-1}$.
Let $\scC$ be a $E_{k_n}$ operad which maps into $\scA_n$. Let $W\scB\longrightarrow\scB$
and $W\scC\longrightarrow\scC$ be cofibrant resolutions. Then according to
Theorem \ref{4_4}, $W\scB\otimes W\scC$ is $E_{k_1+k_2+\dots+k_n}$ and we
have a chain of operad maps
$$W\scB\otimes W\scC\longrightarrow\scB\otimes\scC\longrightarrow 
\scA_1\otimes\scA_2\otimes\dots\otimes\scA_n.$$
This completes the induction and the proof.
\end{proof}

We will now make use of

\begin{lem}\label{4_6}(e.g. see \cite{V}):
Given a diagram of operads
$$
\xymatrix{
&& \mathcal{P} \ar[d]^p
\\
\mathcal{B} \ar[rr]^f && \mathcal{C}
}
$$
with $\mathcal{B}$ cofibrant and $p$ a weak equivalence. Then there is a
lift $g:\mathcal{B}\to\mathcal{P}$ up to homotopy, uniquely up to homotopy.
(Here homotopy means homotopy through operad maps.) 
\end{lem}

\begin{coro}\label{4_7}
\begin{enumerate}
\item[(i)] Any two cofibrant $E_k$ operads $\mathcal{B}$ and $\mathcal{C}$
are homotopy equivalent in the strong sense; i.e. there are operad maps
$f:\mathcal{B}\to\mathcal{C}$ and $g:\mathcal{C}\to\mathcal{B}$ such that
$g\circ f\simeq id$ and $f\circ g\simeq id$ through operad maps.
\item[(ii)] If $\mathcal{A}$ is any operad and $\mathcal{B}$ and
$\mathcal{C}$ are cofibrant $E_k$ operads, than
$\mathcal{A}\otimes\mathcal{B}$ and $\mathcal{A}\otimes\mathcal{C}$ are
homotopy equivalent in the strong sense.
\item[(iii)] If $\mathcal{B}$ and $\mathcal{C}$ are $E_k$ and $\mathcal{B}$
is cofibrant, there is a weak equivalence $\mathcal{B}\to\mathcal{C}$.
\end{enumerate}
\end{coro}

In view of the corollary it suffices to prove Theorems \ref{4_4} for our favorite cofibrant operads.
Moreover, the  Additivity Theorem in the Quillen environment follows from the one in the Str{\o}m
environment: If $\scB$ is a cofibrant $E_n$ operad in the Quillen environment then 
$WP\scB$ is cofibrant in both environments, and the canonical map $WP\scB\to \scB$ is a homotopy
equivalence of operads in the strong sense.

Our favorite cofibrant $E_k$ operad will be $W|\mathcal{NM}_k|$,  a cofibrant resolution
of $|\mathcal{NM}_k|$, the topological realization of the nerve $\mathcal{NM}_k$ 
of the $\Cat$-operad $\scM_k$ which parametrizes the algebraic
structure of a $k$-fold monoidal category \cite{BFSV}. We can briefly describe
$\mathcal{M}_k(m)$ as a poset whose objects are words of length $m$ in the alphabet
$\{1,2,\dots,m\}$ combined together using $k$ binary operations $\square_1,\square_2,\dots,\square_k$,
which are strictly associative and have common unit 0.
We moreover require that the objects of $\mathcal{M}_k(m)$ are precisely those words where
each generator $\{1,2,\dots,m\}$ occurs exactly once. The morphisms of $\mathcal{M}_k(m)$
are generated by interchanges
$$\eta^{ij}_{A,B,C,D}:(A\square_j B)\square_i(C\square_j D)
\longrightarrow (A\square_i C)\square_j(B\square_i D),$$
where $1\le i<j\le k$ and $A,B,C,D$ are words in $\{1,2,\dots,m\}$ such that
$(A\square_j B)\square_i(C\square_j D)$ (and hence also $(A\square_i C)\square_j(B\square_i D)$)
represent objects of $\mathcal{M}_k(m)$. These interchange morphisms are then combined using
the binary operations $\square_1,\square_2,\dots,\square_k$ as well as composition of morphisms.
The coherence theorem of \cite{BFSV} proves the nonobvious fact that $\mathcal{M}_k(m)$
is a poset, i.e. there is at most one morphism between any two objects, and gives a simple
algorithm for determining when there is a morphism between any two objects (c.f. the proof
of Lemma \ref{5_01}).

\begin{leer}\label{4_7a}
The topological operads $|\mathcal{NM}_k|$ and $|\mathcal{NM}_l|$ are far from cofibrant. Indeed
$|\mathcal{NM}_k|(1)$ and $|\mathcal{NM}_l|(1)$ are both $\{id\}$ and hence by Proposition \ref{3_6}
$|\mathcal{NM}_k|\otimes |\mathcal{NM}_l|=\Com$. On the other hand we will show that
$W|\mathcal{NM}_k|\otimes W|\mathcal{NM}_l|$ is $E_{k+l}$. In anticipation of this result, it
will be convenient to adopt the following convention. We will identify $\mathcal{M}_k$ with
the obvious suboperad of $\mathcal{M}_{k+l}$ and we will identify $\mathcal{M}_l$ with the
suboperad of $\mathcal{M}_{k+l}$ obtained by shifting the indices on the binary operations from
$\square_1,\square_2,\dots,\square_l$ to $\square_{k+1},\square_{k+2},\dots,\square_{k+l}$.
\end{leer}

In the course of proving Theorem \ref{4_4} we obtain more results about the spaces
of unary and binary operations in an arbitrary tensor product of operads. As these results
may be of separate interest, we state them below.

\begin{prop}\label{4_8} Let $\scA$ and $\scB$ be arbitrary topological operads with
$\scA(0)=\{0\}$ and $\scB(0)=\{0\}$. Then
\begin{itemize}
\item $(\scA\otimes\scB)(1)\cong \scA(1)\times\scB(1)$
\item $(\scA\otimes\scB)(2)$ is homeomorphic to the pushout of the following diagram
$$
\xymatrix{
\scA(2)\times\scB(2)\ar[r]^f\ar[d]^g &\scA(1)^2\times\scB(2)
\\
\scA(2)\times\scB(1)^2
}
$$
Here $f(\alpha,\beta)=(\alpha\circ(id\oplus0),\alpha\circ(0\oplus id),\beta)$ and
$g(\alpha,\beta)=(\alpha,\beta\circ(id\oplus0),\beta\circ(0\oplus id))$.
\end{itemize}
\end{prop}

To get a handle on the homotopy type of the space of binary operations, we need
to impose additional hypotheses (which are satisfied by $\scA=W|\mathcal{NM}_k|$ and
$\scB=W|\mathcal{NM}_l|$ in particular). We then obtain the following result.

\begin{coro}\label{4_9} Let $\scA$ and $\scB$ be arbitrary topological operads with
$\scA(0)=\{0\}$ and $\scB(0)=\{0\}$. Suppose also that the spaces $\scA(1)$ and
$\scB(1)$ are contractible and the map $\scA(2)\longrightarrow\scA(1)^2$ given by
$\alpha\mapsto(\alpha\circ(id\oplus0),\alpha\circ(0\oplus id))$ is a cofibration.
Then $(\scA\otimes\scB)(2)$ has the homotopy type of the join $\scA(2)*\scB(2)$.
\end{coro}

\begin{proof} Because of the cofibration hypothesis, the pushout diagram for
$(\scA\otimes\scB)(2)$ has the same homotopy type as the homotopy pushout.
Because $\scA(1)$ and $\scB(1)$ are contractible, this in turn has the
same homotopy type as the homotopy pushout of 
$$
\xymatrix{
\scA(2)\times\scB(2)\ar[rr]^{\mbox{pr}_2}\ar[d]^{\mbox{pr}_1} &&\scB(2)
\\
\scA(2)
},
$$
which is by definition the join $\scA(2)*\scB(2)$.
\end{proof}

If $\scA$ is an $E_k$ operad and satisfies the cofibration hypothesis of Corollary \ref{4_9},
while $\scB$ is an arbitrary $E_l$ operad, then $\scA(2)$ has the homotopy type of $S^{k-1}$,
and $\scB(2)$ has the homotopy type of $S^{l-1}$. Thus by Corollary \ref{4_9}, $(\scA\otimes\scB)(2)$
has the homotopy type of $S^{k-1}*S^{l-1}\cong S^{k+l-1}$. This is consistent with $\scA\otimes\scB$
being an $E_{k+l}$ operad, and suggests that Theorem \ref{4_4} might hold with a weaker hypothesis.
However, we will not pursue this further in this paper.

\section{Outline of the proofs}

Before we discuss the proof of Theorem \ref{4_4}, we need a more explicit description
of the tensor product of operads. We begin with a description of the coproduct
$\mathcal{B}\amalg\mathcal{C}$ of $\mathcal{B}$ and $\mathcal{C}$. The elements
of $(\scA\amalg\scB)(m)$ are equivalence classes of planar trees with $m$ labeled
inputs and one output, and whose nodes are labeled by elements of the the operads $\scA$
and $\scB$. It is required that the arity of each node label correspond to the number
of input branches coming into that node. So we have to allow nodes without inputs, which we 
call \textit{stumps} labeled by elements in $\scA(0)$ or $\scB(0)$. For a more detailed description
see Definition \ref{7_7}. The topology imposed  is the quotient topology
on the evident disjoint union of products of spaces $\scA(i)$ and $\scB(j)$.
Operad composition in $\scA\amalg\scB$ is given by splicing together trees.

The following equivalence relations are imposed
on the trees.  First of all if two trees have subtrees of the form shown
below but are otherwise identical, then they are identified:
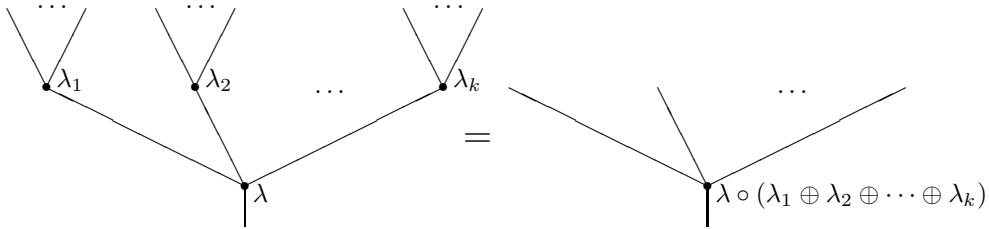
\begin{leer}\label{5_1a}
$$\mbox{
\setlength{\unitlength}{0.75\unitlength}
\begin{picture}(220,115)(-100,0)
\put(0,0){\line(0,1){20}}
\put(0,20){\circle*{4}}
\put(0,20){\line(-2,1){100}}
\put(0,20){\line(-1,2){25}}
\put(0,20){\line(2,1){100}}
\put(35,67){\dots}
\put(4,12){$\lambda$}
\put(-100,70){\circle*{4}}
\put(-100,70){\line(-1,2){20}}
\put(-100,70){\line(1,2){20}}
\put(-95,70){$\lambda_1$}
\put(-105,110){\dots}
\put(-25,70){\circle*{4}}
\put(-25,70){\line(-1,2){20}}
\put(-25,70){\line(1,2){20}}
\put(-20,70){$\lambda_2$}
\put(-30,110){\dots}
\put(100,70){\circle*{4}}
\put(100,70){\line(-1,2){20}}
\put(100,70){\line(1,2){20}}
\put(105,70){$\lambda_k$}
\put(95,110){\dots}
\put(110,40){\Large=}
\end{picture}
}\mbox{
\setlength{\unitlength}{0.75\unitlength}
\begin{picture}(220,115)(-100,0)
\put(0,0){\line(0,1){20}}
\put(0,20){\circle*{4}}
\put(0,20){\line(-2,1){100}}
\put(0,20){\line(-1,2){25}}
\put(0,20){\line(2,1){100}}
\put(35,67){\dots}
\put(4,12){$\lambda\circ(\lambda_1\oplus\lambda_2\oplus\dots\oplus\lambda_k)$}
\end{picture}
}$$
Here we are assuming that the node labels $\lambda$, $\lambda_1$,
$\lambda_2$, \dots, $\lambda_k$ all belong to $\scA$ or all belong
to $\scB$, and $\lambda\circ(\lambda_1\oplus\lambda_2\oplus\dots\oplus\lambda_k)$
denotes operad composition.

Additionally we also allow
changing unary nodes which are labeled by the unit of $\scA$ to unary nodes
labeled by the unit of $\scB$ and vice versa. 
\end{leer}
\begin{leer}\label{5_1b}
 We also need to impose an equivariance relation
$$\mbox{
\setlength{\unitlength}{0.75\unitlength}
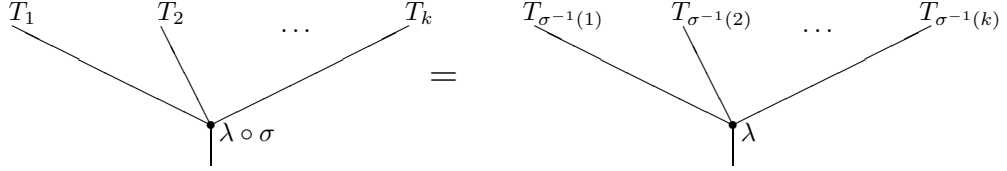
\begin{picture}(220,80)(-100,0)
\put(0,0){\line(0,1){20}}
\put(0,20){\circle*{4}}
\put(0,20){\line(-2,1){100}}
\put(0,20){\line(-1,2){25}}
\put(0,20){\line(2,1){100}}
\put(35,67){\dots}
\put(4,12){$\lambda\circ\sigma$}
\put(-102,73){$T_1$}
\put(-27,73){$T_2$}
\put(98,73){$T_k$}
\put(110,40){\Large=}
\end{picture}
}\mbox{
\setlength{\unitlength}{0.75\unitlength}
\begin{picture}(250,80)(-130,0)
\put(0,0){\line(0,1){20}}
\put(0,20){\circle*{4}}
\put(0,20){\line(-2,1){100}}
\put(0,20){\line(-1,2){25}}
\put(0,20){\line(2,1){100}}
\put(35,67){\dots}
\put(4,12){$\lambda$}
\put(-106,73){$T_{\sigma^{-1}(1)}$}
\put(-31,73){$T_{\sigma^{-1}(2)}$}
\put(94,73){$T_{\sigma^{-1}(k)}$}
\end{picture}
}$$
\end{leer}
Here we are assuming that the node label $\lambda$ is either in $\scA(k)$ or
in $\scB(k)$ and $\sigma\in\Sigma_k$.

The resulting spaces of equivalence
classes of trees form the coproduct operad $\scA\amalg\scB$.
The two equivalence relations insure that the
images of $\scA$ and $\scB$ are suboperads of $\scA\amalg\scB$.

\begin{leer}\label{5_1c}
To pass from the coproduct $\scA\amalg\scB$ to the tensor product we need
to impose another relation corresponding to the interchange. We identify two
trees if they have subtrees of the form shown below, but are otherwise identical:
$$\mbox{
\setlength{\unitlength}{0.75\unitlength}
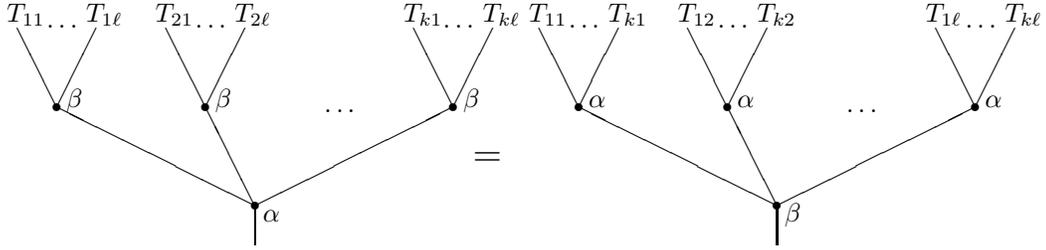
\begin{picture}(220,120)(-100,0)
\put(0,0){\line(0,1){20}}
\put(0,20){\circle*{4}}
\put(0,20){\line(-2,1){100}}
\put(0,20){\line(-1,2){25}}
\put(0,20){\line(2,1){100}}
\put(35,67){\dots}
\put(4,12){$\alpha$}
\put(-100,70){\circle*{4}}
\put(-100,70){\line(-1,2){20}}
\put(-100,70){\line(1,2){20}}
\put(-95,70){$\beta$}
\put(-105,110){\dots}
\put(-125,113){$T_{11}$}
\put(-85,113){$T_{1\ell}$}
\put(-25,70){\circle*{4}}
\put(-25,70){\line(-1,2){20}}
\put(-25,70){\line(1,2){20}}
\put(-20,70){$\beta$}
\put(-30,110){\dots}
\put(-50,113){$T_{21}$}
\put(-10,113){$T_{2\ell}$}
\put(100,70){\circle*{4}}
\put(100,70){\line(-1,2){20}}
\put(100,70){\line(1,2){20}}
\put(105,70){$\beta$}
\put(95,110){\dots}
\put(75,113){$T_{k1}$}
\put(115,113){$T_{k\ell}$}
\put(110,40){\Large=}
\end{picture}
}\mbox{
\setlength{\unitlength}{0.75\unitlength}
\begin{picture}(250,120)(-130,0)
\put(0,0){\line(0,1){20}}
\put(0,20){\circle*{4}}
\put(0,20){\line(-2,1){100}}
\put(0,20){\line(-1,2){25}}
\put(0,20){\line(2,1){100}}
\put(35,67){\dots}
\put(-125,113){$T_{11}$}
\put(-85,113){$T_{k1}$}
\put(4,12){$\beta$}
\put(-100,70){\circle*{4}}
\put(-100,70){\line(-1,2){20}}
\put(-100,70){\line(1,2){20}}
\put(-95,70){$\alpha$}
\put(-105,110){\dots}
\put(-50,113){$T_{12}$}
\put(-10,113){$T_{k2}$}
\put(-25,70){\circle*{4}}
\put(-25,70){\line(-1,2){20}}
\put(-25,70){\line(1,2){20}}
\put(-20,70){$\alpha$}
\put(-30,110){\dots}
\put(100,70){\circle*{4}}
\put(100,70){\line(-1,2){20}}
\put(100,70){\line(1,2){20}}
\put(105,70){$\alpha$}
\put(95,110){\dots}
\put(75,113){$T_{1\ell}$}
\put(115,113){$T_{k\ell}$}
\end{picture}
}$$
\end{leer}
Here $\alpha\in\scA(k)$, $\beta\in\scB(\ell)$, and $T_{ij}$ denote the
branches of the trees lying above the nodes shown. Note that these branches
are permuted on the two sides of the relation. The resulting spaces of
equivalence classes of trees form the tensor product $\scA\otimes\scB$.
For future reference, we will define the equivalence relation shown in the picture
above as a \underline{$(k,l)$-interchange} on trees. Trees related by a sequence
of $(k,1)$- or $(1,l)$-interchanges will be said to be related by \underline{unary
interchanges} (cf. Remark 3.3).

\begin{rema}\label{5_001}
According to Proposition \ref{3_5}, the nullary operations in $\scA$ and $\scB$
are identified to a single nullary operation in $\scA\otimes\scB$.  This together
with relation \ref{5_1a} allows us to remove any stumps from trees representing
elements in $\scA\otimes\scB$. We shall refer to these relations as \underline{nullary
interchanges}.
\end{rema}

In order to determine the homotopy type of a tensor product of operads, we need
to describe its spaces of operations as colimits of ``nice'' diagrams. Since the
operad $\mathcal{M}_2$ encodes the algebraic structure of a 2-fold monoidal category,
i.e. a category with two multiplications which interchange with each other up to
coherent natural transformations, it should not be surprising that $\scM_2$
should serve as the basis for constructing such diagrams. However we need to make
some adjustments.

First of all, to avoid confusion between the multiplications in $\scM_2$
and the internal multiplications in $W|\mathcal{NM}_k|\otimes W|\mathcal{NM}_l|$ arising
from those in $\scM_k$ and $\scM_l$, we will denote the multiplications in $\scM_2$
by $\boxtimes_1$ and $\boxtimes_2$ instead of $\square_1$ and $\square_2$.

\begin{lem}\label{5_01} The poset operad $\scM_2$ has a quotient poset operad
$\mathcal{M}_2^{ab}$, obtained from $\scM_2$ by making the operations
$\boxtimes_1$ and $\boxtimes_2$ commutative.\\
\end{lem}

\begin{proof} In order to verify that making $\boxtimes_1$ and $\boxtimes_2$ commutative
is compatible with the poset structure of $\scM_2(m)$, we recall the criterion
for the existence of a morphism $\alpha\longrightarrow\beta$ in $\scM_2(m)$. If $\{a,b\}$
is any two element subset of $\{1,2,\dots,m\}$ and $\alpha\cap\{a,b\}=a\boxtimes_i b$, then
either $\beta\cap\{a,b\}=a\boxtimes_j b$ for $j\ge i$ or $\beta\cap\{a,b\}=b\boxtimes_j a$ for
$j>i$. [Here we use the notation for restriction maps introduced in Definition \ref{2_3a}.]
Using this criterion we see that the number of pairs $\{a,b\}$ for which
$\alpha\cap\{a,b\}=a\boxtimes_2 b$ is less than or equal to the number of pairs $\{c,d\}$
for which $\beta\cap\{c,d\}=c\boxtimes_2 d$, and that the counts are equal iff $\alpha=\beta$.
Making $\boxtimes_1$ and $\boxtimes_2$ commutative does not affect these counts, so the
poset structure on $\scM_2(m)$ induces a poset structure on $\scM_2^{ab}(m)$.
It is straight forward to verify that the operad structure on $\scM_2$ passes to
an operad structure on $\scM_2^{ab}$.
\end{proof}

\begin{defi}\label{5_05} We will not need to use the poset structure of $\scM_2^{ab}$
until Section \ref{sec11}. However we will need to use the underlying sets of objects of
$\scM_2^{ab}$ in various constructions before that. Hence we will use the separate
notation $\mathfrak{M}_2^{ab}(m)$ to denote the underlying set of objects of $\scM_2^{ab}(m)$.
Obviously the operad structure on $\scM_2^{ab}$ restricts to an operad structure on
$\mathfrak{M}_2^{ab}(m)$.
\end{defi}

\begin{defi}\label{5_4} We define a collection of simplicial complexes (in the classical sense of the
term, rather than simplicial sets) $\{\mathfrak{K}_\bullet(m)\}_{m\ge0}$
as follows. The vertex set of $\{\mathfrak{K}_\bullet(m)\}_{m\ge0}$ is $\mathfrak{M}_2^{ab}(m)$.
The complexes $\mathfrak{K}_\bullet(0)$ and $\mathfrak{K}_\bullet(1)$ are 0-dimensional consisting of a single
vertex while $\mathfrak{K}_\bullet(2)$ is the simplicial complex consisting of the single 1-simplex
$\{1\boxtimes_1 2,1\boxtimes_2 2\}$ and its subsimplices. The complex
$\mathfrak{K}_\bullet(3)$ has the vertex set
$$
\xymatrix{
&(1\boxtimes_22)\boxtimes_13 & (2\boxtimes_13)
\boxtimes_21 \\
(1\boxtimes_12\boxtimes_13) &
(1\boxtimes_23)\boxtimes_12 &
(1\boxtimes_13)\boxtimes_22& (1\boxtimes_22\boxtimes_23)\\
& (2\boxtimes_23)\boxtimes_11 &
(1\boxtimes_12)\boxtimes_23
}
$$
The simplices are those collections of vertices which contain at most one vertex from each
column. For $m\ge 3$ the simplicial complex $\mathfrak{K}_\bullet(m)$ has as $r$-simplices
those $(r+1)$-tuples $\{\alpha_0,\alpha_1,\dots \alpha_r\}$ such that for each subset $\{a<b<c\}\subset\{1,2,\dots,m\}$
the restrictions $\{\alpha_0\cap\{a,b,c\},\alpha_1\cap\{a,b,c\},\dots,\alpha_r\cap\{a,b,c\}$ forms a simplex in
 $\mathfrak{K}_\bullet(\{a,b,c\})$, which is identified with 
$\mathfrak{K}_\bullet(3)$ via the isomorphism induced by $1\mapsto a$, $2\mapsto b$ and $3\mapsto c$.
\end{defi}

\begin{lem}\label{5_1}For any topological operads $\scA$ and $\scB$, there is a natural map of operads
$$\varepsilon:\scA\amalg\scB\longrightarrow\mathfrak{M}_2^{ab}$$
\end{lem}

\begin{proof}
To define $\varepsilon$, take a tree representative $T$ of an element in $(\scA\amalg\scB)(m)$. Replace each stump by $0$
and delete each node with exactly one input by combining its input and output.
Regard each node in $T$ with more than $1$ input and label coming from $\scA$ as an iterated multiplication $\boxtimes_1$ and label
coming from $\scB$ as an iterated multiplication $\boxtimes_2$. Interpret the edges of the resulting tree as compositions
in the operad $\mathfrak{M}_2^{ab}$. This evaluated tree gives an element of
$\mathfrak{M}_2^{ab}(m)$. It is obvious that this construction is compatible with the
equivalence relations on trees which define the elements of $(\scA\amalg\scB)(m)$. [Note that
this construction does not give a well defined map $\scA\amalg\scB$ to the set of objects of
$\scM_2$,
since that would not be compatible with the equivariance relation on trees.]
\end{proof}

\begin{leer}\label{5_2}
We observe that Lemma \ref{5_1} gives a trivial colimit decomposition of $(\scA\amalg\scB)(m)$,
namely as a disjoint union of $\varepsilon^{-1}(\alpha)$ indexed over all the elements of
$\mathfrak{M}_2^{ab}(m)$. For future reference, we will use the notation $\widetilde{G}_m(\alpha)$
to refer to $\varepsilon^{-1}(\alpha)$. The first step in obtaining a colimit decomposition for the tensor
product is to observe that trees in $(\scA\amalg\scB)(m)$ which are related by unary interchanges
have the same image under $\varepsilon$, since $\mathfrak{M}_2^{ab}(1)=\{1=id\}$. Thus for any object
$\alpha$ in $\mathfrak{M}_2^{ab}(m)$ we define
$$G_m(\alpha)=\widetilde{G}_m(\alpha)/\mbox{nullary and unary interchanges}$$
\end{leer}

 From now on we will take $\scA=W|\mathcal{NM}_k|$ and $\scB=W|\mathcal{NM}_l|$.

The proofs of the remaining statements are postponed to Sections 10 and 11. 

\begin{prop}\label{5_3} For each element $\alpha$ in $\mathfrak{M}_2^{ab}(m)$, the natural map
$$G_m(\alpha)\longrightarrow(W|\mathcal{NM}_k|\otimes W|\mathcal{NM}_l|)(m)$$
is a cofibration. 
\end{prop}

Thus $(W|\mathcal{NM}_k|\otimes W|\mathcal{NM}_l|)(m)$ is a union of subspaces $G_m(\alpha)$
over all objects $\alpha$ in $\mathfrak{M}^{ab}_2(m)$. To determine the homotopy type
$(W|\mathcal{NM}_k|\otimes W|\mathcal{NM}_l|)(m)$, we need to analyze the intersections of these
closed subspaces. We would expect that these intersections should correspond to nonunary
interchanges, which are encoded by the simplices of $\mathfrak{K}_\bullet(m)$.

\begin{defi}\label{5_41} The \textit{interchange diagram} for $(W|\mathcal{NM}_k|\otimes W|\mathcal{NM}_l|)(m)$ is the
following diagram $G_m$ indexed by $\scI(m)=\mbox{Sd}\mathfrak{K}_\bullet(m)$,
the barycentric subdivision of $\mathfrak{K}_\bullet(m)$ with poset structure opposite to the inclusions 
of the faces of $\mathfrak{K}_\bullet(m)$:
\begin{itemize}
\item To each vertex $\alpha\in\mathfrak{K}_0(m)=\mathfrak{M}_2^{ab}(m)$ assign $G_m(\alpha)$ to $\alpha$.
\item To each barycenter of a simplex in $\mathfrak{K}_\bullet(m)$ assign the intersection
$\cap G_m(\alpha_i)$, where $\{\alpha_i\}$ are the vertices of that simplex. 
\item The maps in the diagram are inclusions.
\end{itemize}
\end{defi}

\begin{prop}\label{5_5} The interchange diagram for
$(W|\mathcal{NM}_k|\otimes W|\mathcal{NM}_l|)(m)$ is a diagram of cofibrations, and
$(W|\mathcal{NM}_k|\otimes W|\mathcal{NM}_l|)(m)$ is the colimit of that diagram. Moreover
each space in that diagram has the general form
$$\left(\prod W|\mathcal{NM}_k|(r_i)\times\prod W|\mathcal{NM}_l|(s_j)\right)/\mbox{unary interchanges},$$
and the maps in the diagram are given by operad compositions (including insertion
of the unique constant).
\end{prop}

The operad spaces $|\mathcal{NM}_k|(r)$ and $|\mathcal{NM}_l|(s)$ can be described
as colimits of diagrams of contractible spaces indexed by the posets
$\mathcal{M}_k(r)$ and $\mathcal{M}_l(s)$, namely by assigning to each object
in $\mathcal{M}_k(r)$ or $\mathcal{M}_l(s)$ the realization of the nerve of the subposet for which
that object is terminal. Pulling back these diagrams along the augmentations
$W|\mathcal{NM}_k|(r)\longrightarrow |\mathcal{NM}_k|(r)$ and
$W|\mathcal{NM}_l|(s)\longrightarrow |\mathcal{NM}_l|(s)$, we obtain similar colimit
diagrams for $W|\mathcal{NM}_k|(r)$ and $W|\mathcal{NM}_l|(s)$. Moreover the maps in
these diagrams are cofibrations.

Combining these colimits with the colimit diagram of Proposition \ref{5_5}, we obtain
a refined iterated colimit decomposition of $(W|\mathcal{NM}_k|\otimes W|\mathcal{NM}_l|)(m)$.
By Proposition 5.2 of \cite{BFV}, such an iterated colimit diagram can be reexpressed
as a single colimit diagram over an appropriate Grothendieck construction, which we
denote $\scI(k,l)(m)$. We thus obtain the following result.

\begin{prop}\label{5_6} $(W|\mathcal{NM}_k|\otimes W|\mathcal{NM}_l|)(m)$ is the 
colimit of a diagram $G'_m$ of cofibrations
of contractible spaces indexed by the poset $\scI(k,l)(m)$,
obtained by the Grothendieck construction from the interchange diagram of Proposition \ref{5_5} by replacing the space
at each node by the product of posets $\prod\mathcal{M}_k(r_i)\times\prod\mathcal{M}_l(s_j)$
parametrizing the colimit decomposition at that node. The functors in that diagram are given
by operad compositions (including insertions of the unique constant).
\end{prop}

Next we recall a relevant definition from \cite{BFSV}.

\begin{defi}\label{5_6a}
By a \textit{cellular decomposition} of a topological space $X$
over a finite poset $\scP$, we will mean a diagram $\{C_\alpha\}_{\alpha\in\scP}$ of closed subsets of $X$ indexed by $\scP$ satisfying the following conditions:
\begin{enumerate}
\item $X=\cup_{\alpha\in\scP} C_\alpha$.
\item If $\alpha\le\beta$ in $\scP$, then the inclusion $C_\alpha\to C_\beta$ is a cofibration.
\item Each $C_\alpha$ is contractible.
\item $C_\alpha\cap C_\beta=\cup_{\begin{array}{c}\gamma\le\alpha\\ \gamma\le\beta\end{array}}C_\gamma$
\end{enumerate}
Under these circumstances, as shown in  \cite{BFSV}, we have a sequence of equivalences and
homeomorphisms:
$$|\scN\scP|\stackrel{\simeq}{\longleftarrow}\mbox{hocolim}_{\alpha\in\scP}C_\alpha
\stackrel{\simeq}{\longrightarrow}\mbox{colim}_{\alpha\in\scP}C_\alpha
\stackrel{\cong}{\longrightarrow}X,$$
We will say that a topological operad $\scA=\{\scA(m)\}_{m\ge0}$ has a cellular decomposition
over a poset operad $\scP=\{\scP(m)\}_{m\ge0}$, if $\scA(m)$ has a cellular decomposition over
$\scP(m)$ for each $m$, and the operad structures on $\scA$ and $\scP$ satisfy the following
compatibility conditions.
\begin{enumerate}
\item If $a\in C_\alpha\subset\scA(m)$ and $b_i\in C_{\beta_i}\subset\scA(n_i)$, $i=1,2,\dots m$,
then $a\circ(\oplus_{i=1}^m b_i)\in C_{\alpha\circ(\oplus_{i=1}^m \beta_i)}\subset\scA(n_1+n_2+\dots+n_m)$.
\item If $a\in C_\alpha\subset\scA(m)$ and $\sigma\in\Sigma_m$, then $a\sigma\in C_{\alpha\sigma}$.
\end{enumerate}
Under these circumstances the chain of equivalences
$$|\scN\scP(m)|\stackrel{\simeq}{\longleftarrow}\mbox{hocolim}_{\alpha\in\scP(m)}C_\alpha
\stackrel{\simeq}{\longrightarrow}\mbox{colim}_{\alpha\in\scP(m)}C_\alpha
\stackrel{\cong}{\longrightarrow}\scA(m),$$
gives an equivalence of operads.
(This definition is inspired by, but differs slightly from the one of Berger \cite{Berger}).
\end{defi}

Thus Proposition \ref{5_6} provides a cellular decomposition of 
$(W|\mathcal{NM}_k|  \otimes W|\mathcal{NM}_l|)(m)$
over the poset $\scI(k,l)(m)$ for each $m$.

\begin{prop}\label{5_7} There is a functor
$$L:\scI(k,l)(m)\longrightarrow\mathcal{M}_{k+l}(m)$$
which satisfies Quillen's Theorem A, and thus induces an equivalence upon passage to nerves.
\end{prop}

Combining these results, we obtain  a chain of equivalences
$$
\begin{array}{rl}
(W|\mathcal{NM}_k|  \otimes W|\mathcal{NM}_l|)(m) &
\cong\mbox{colim}_{\scI(k,l)(m)}G'_m
 \stackrel{\simeq}{\longleftarrow}\mbox{hocolim}_{\scI(k,l)(m)}G'_m\\
& \stackrel{\simeq}{\longrightarrow} 
\mbox{hocolim}_{\scI(k,l)(m)}*
=|\mathcal{N}\scI(k,l)(m)|\\
& \stackrel{\simeq}{\longrightarrow}|\mathcal{NM}_{k+l}|(m)
\end{array}
$$
Finally we construct an operad structure on $\scI(k,l)=\{\scI(k,l)(m)\}_{m\ge0}$, 
verify the relevant compatibility conditions on the cellular decomposition of
$W|\mathcal{NM}_k|  \otimes W|\mathcal{NM}_l$, and thus
 show that this chain of equivalences defines a chain of operad equivalences.
This completes the proof of Theorem \ref{4_4}.

\section{Unary and binary operations}

In this section we analyze the spaces of unary and binary operations in a
tensor product $\scA\otimes\scB$ of two arbitrary reduced operads $\scA$
and $\scB$, and thus prove Proposition \ref{4_8}.

First of all observe that in the coproduct operad $\scA\amalg\scB$, the
unary operations are arbitrary compositions of unary operations in $\scA$
and $\scB$, with the units of $\scA(1)$ and $\scB(1)$ identified and with
the single relation that adjacent factors both in $\scA(1)$ or both in $\scB(1)$
can be combined into a single factor using the multiplications in $\scA(1)$ or
$\scB(1)$. In other words, $(\scA\amalg\scB)(1)$ is the coproduct (or free product)
$\scA(1)*\scB(1)$ in the category of topological monoids. The interchange
relations in $\scA\otimes\scB$ restrict on the space of unary operations to
the relation that factors from $\scA(1)$ commute with factors from $\scB(1)$.
Thus the factors from $\scA(1)$ can be commuted to the front of a word, leaving
the factors from $\scB(1)$ at the back of the word. Then the factors from $\scA(1)$ and
from $\scB(1)$ can be combined into single factors, using the multiplications
in $\scA(1)$ and $\scB(1)$, leaving a pair of factors, the first from $\scA(1)$
and the second from $\scB(1)$. This establishes that
$$(\scA\otimes\scB)(1)\cong\scA(1)\times\scB(1)$$
and thus proves the first part of Proposition \ref{4_8}.

To analyze the binary operations in $\scA\otimes\scB$, we begin by noting
that according to Proposition \ref{5_1} and the follow up discussion in paragraph \ref{5_2} we have
$$(\scA\amalg\scB)(2)=\widetilde{G}_2(1\boxtimes_1 2)\amalg \widetilde{G}_2(1\boxtimes_2 2).$$
Now by definition the elements of $\widetilde{G}_2(1\boxtimes_1 2)$ are represented by trees with
one binary node labeled by an element $a_2$ of $\scA(2)$ and arbitrary chains of unary nodes labeled
by arbitrary elements of $\scA(1)$ and $\scB(1)$ on both input branches coming into the binary node
and on the output branch from that node. Upon dividing out by the unary interchange relations, i.e.
passing to $G_2(1\boxtimes_1 2)$, we may commute the factors coming from $\scB(1)$ past factors
coming from $\scA(1)$ towards the top of each branch and combine them into single factors, using the
multiplication in $\scB(1)$. Let $b\in\scB(1)$ be the factor remaining on the top of the output branch
of the binary node labeled by $a_2$. We can then use the unary interchange relation
$$b\circ a_2 = a_2\circ(b\oplus b)$$
to move $b$ from below the binary node to both branches above that node. We can then commute these
copies of $b$ past factors of $\scA(1)$ on these branches and then combine them with the factors
from $\scB(1)$ at the top of these branches. This leaves a tree with one factor of $\scB(1)$ at the
top of each input branch and factors of $\scA(1)$ below them on the input branches, $a_2\in\scA(2)$
on the binary node and additional factors of $\scA(1)$ on the output branch below that node. These
$\scA$ nodes can all be composed together to produce a single binary node with a label in $\scA(2)$.
Finally we can use the equivariance relation to insure that input 1 is on the left branch of the
binary node and input 2 is on the right. Thus each element of $G_2(1\boxtimes_1 2)$ has a unique
tree representative of the form
$$
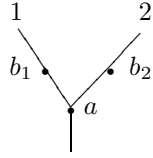
\xymatrix@=5pt@M=-1pt@W=-1pt{
1\ && && &&\ 2\\
\ar@{-}[ddddrrr] && && &&\quad\ \ar@{-}[ddddlll]\\
\\
b_1&{\scriptscriptstyle\bullet}&& &&{\scriptscriptstyle\bullet}&b_2\\
\\
&&&{\scriptscriptstyle\bullet}\ar@{-}[ddd]&a\\
\\
\\
&&&&\ 
}$$

which we will refer to as the reduced form of that element. This shows that
$$G_2(1\boxtimes_1 2)\cong \scA(2)\times\scB(1)^2.$$
A similar argument, with the roles of $\scA$ and $\scB$ interchanged, shows that
each element of $G_2(1\boxtimes_2 2)$ has a unique reduced form representative
$$
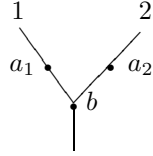
\xymatrix@=5pt@M=-1pt@W=-1pt{
1\ && && &&\ 2\\
\ar@{-}[ddddrrr] && && &&\quad\ \ar@{-}[ddddlll]\\
\\
a_1&{\scriptscriptstyle\bullet}&& &&{\scriptscriptstyle\bullet}&a_2\\
\\
&&&{\scriptscriptstyle\bullet}\ar@{-}[ddd]&b\\
\\
\\
&&&&\ 
}$$
and thus
$$G_2(1\boxtimes_2 2)\cong \scA(1)^2\times\scB(2).$$

It now remains to analyze the nonunary interchanges in $(\scA\otimes\scB)(2)$.
It is easy to see that any such interchange reduces to a $(2,2)$-interchange.
A reduced representative in $G_2(1\boxtimes_1 2)$ can only interchange this way
if both $b_1$ and $b_2$ factor through a common element in $b\in\scB(2)$ by
composing with the constant 0. The resulting interchanges are shown below\newline
\centerline{$
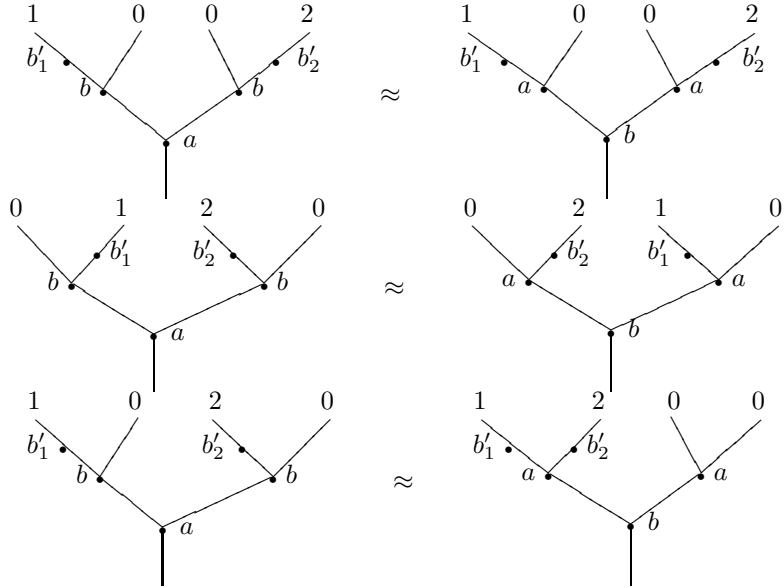
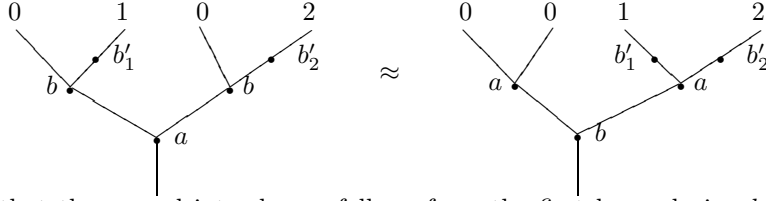
\xymatrix@=5pt@M=-1pt@W=-1pt{
1\ && &&&&0 &&&&0 && &&&2\\
\ar@{-}[ddrrr] && &&&&\ar@{-}[ddlll] &&&&\ar@{-}[ddrr] && &&&\ar@{-}[ddlll]\\
b_1'&\,{\scriptscriptstyle\bullet}&&&& &&& &&&& &&{\scriptscriptstyle\bullet}\,&\,b_2'\\
&&b &{\scriptscriptstyle\bullet}\ar@{-}[dddrrrrr]&&&&&&&&&{\scriptscriptstyle\bullet}\ar@{-}[dddllll]&b\\
\\
\\
&&&& &&&&{\scriptscriptstyle\bullet}\ar@{-}[dddd]&\,a\\
\\
\\
\\
&&&& &&&&
}$
\raisebox{-30pt}{$\qquad\approx\qquad$}
$\xymatrix@=5pt@M=-1pt@W=-1pt{
1\ && &&&&0 &&&&0 && &&&2\\
\ar@{-}[ddrrr] && &&&&\ar@{-}[ddlll] &&&&\ar@{-}[ddrr] && &&&\ar@{-}[ddlll]\\
b_1'&\,\,{\scriptscriptstyle\bullet}&&&& &&& &&&& &&{\scriptscriptstyle\bullet}\,\,&\,b_2'\\
&&a &{\scriptscriptstyle\bullet}\ar@{-}[dddrrrrr]&&&&&&&&&{\scriptscriptstyle\bullet}\ar@{-}[dddllll]&a\\
\\
\\
&&&& &&&&{\scriptscriptstyle\bullet}\ar@{-}[dddd]&\,b\\
\\
\\
\\
&&&& &&&&
}$}\newline
\centerline{$\xymatrix@=5pt@M=-1pt@W=-1pt{
0\ && &&&&1 &&&&2 && &&&&0\\
\ar@{-}[ddrrr] && &&&&\ar@{-}[ddlll] &&&&\ar@{-}[ddrr] && &&&&\ar@{-}[ddllll]\\
&&&&&{\scriptscriptstyle\bullet} &b_1'&& &&b_2'\,&{\scriptscriptstyle\bullet}\ \,& &&&\\
&&b &{\scriptscriptstyle\bullet}\ar@{-}[dddrrrrr]&&&&&&&&&{\scriptscriptstyle\bullet}\ar@{-}[dddllll]&b\\
\\
\\
&&&& &&&&{\scriptscriptstyle\bullet}\ar@{-}[dddd]&\,a\\
\\
\\
\\
&&&& &&&&
}$
\raisebox{-30pt}{$\qquad\approx\qquad$}
$\xymatrix@=5pt@M=-1pt@W=-1pt{
0\ && &&&&2 &&&&1 && &&&&0\\
\ar@{-}[ddrrr] && &&&&\ar@{-}[ddlll] &&&&\ar@{-}[ddrr] && &&&&\ar@{-}[ddllll]\\
&&&&&{\scriptscriptstyle\bullet} &b_2'&& &&b_1'\,&{\scriptscriptstyle\bullet}\ \,& &&&\\
&&a &{\scriptscriptstyle\bullet}\ar@{-}[dddrrrrr]&&&&&&&&&{\scriptscriptstyle\bullet}\ar@{-}[dddllll]&a\\
\\
\\
&&&& &&&&{\scriptscriptstyle\bullet}\ar@{-}[dddd]&\,b\\
\\
\\
\\
&&&& &&&&
}$}\newline
\centerline{$\xymatrix@=5pt@M=-1pt@W=-1pt{
1\ && &&&&0 &&&&2 && &&&&0\\
\ar@{-}[ddrrr] && &&&&\ar@{-}[ddlll] &&&&\ar@{-}[ddrr] && &&&&\ar@{-}[ddllll]\\
b_1'&{\scriptscriptstyle\bullet}&&&& &&& &&b_2'\,&{\scriptscriptstyle\bullet}\ \,& &&&\\
&&b &{\scriptscriptstyle\bullet}\ar@{-}[dddrrrrr]&&&&&&&&&{\scriptscriptstyle\bullet}\ar@{-}[dddllll]&b\\
\\
\\
&&&& &&&&{\scriptscriptstyle\bullet}\ar@{-}[dddd]&\,a\\
\\
\\
\\
&&&& &&&&
}$
\raisebox{-30pt}{$\qquad\approx\qquad$}
$\xymatrix@=5pt@M=-1pt@W=-1pt{
1\ && &&&&2 &&&&0 && &&&&0\\
\ar@{-}[ddrrr] && &&&&\ar@{-}[ddlll] &&&&\ar@{-}[ddrr] && &&&&\ar@{-}[ddllll]\\
b_1'&{\scriptscriptstyle\bullet}&&&&{\scriptscriptstyle\bullet} &b_2'&& &&&& &&&\\
&&a &{\scriptscriptstyle\bullet}\ar@{-}[dddrrrrr]&&&&&&&&&{\scriptscriptstyle\bullet}\ar@{-}[dddllll]&a\\
\\
\\
&&&& &&&&{\scriptscriptstyle\bullet}\ar@{-}[dddd]&\,b\\
\\
\\
\\
&&&& &&&&
}$}\newline
\centerline{$\xymatrix@=5pt@M=-1pt@W=-1pt{
\\
0\ && &&&&1 &&&&0 && &&&2\\
\ar@{-}[ddrrr] && &&&&\ar@{-}[ddlll] &&&&\ar@{-}[ddrr] && &&&\ar@{-}[ddlll]\\
&&&&&{\scriptscriptstyle\bullet} &\,b_1'&& &&&& &&\,{\scriptscriptstyle\bullet}\ \,&b_2'\\
&&b &{\scriptscriptstyle\bullet}\ar@{-}[dddrrrrr]&&&&&&&&&{\scriptscriptstyle\bullet}\ar@{-}[dddllll]&b\\
\\
\\
&&&& &&&&{\scriptscriptstyle\bullet}\ar@{-}[dddd]&\,a\\
\\
\\
\\
&&&& &&&&
}$
\raisebox{-30pt}{$\qquad\approx\qquad$}
$\xymatrix@=5pt@M=-1pt@W=-1pt{
\\
0&& &&&&0 &&&&1\ && &&&2\\
\ar@{-}[ddrrr] && &&&&\ar@{-}[ddlll] &&&&\ar@{-}[ddrr] && &&&\ar@{-}[ddlll]\\
&&&&& &&& &&b_1'\,&{\scriptscriptstyle\bullet}\ & &&{\scriptscriptstyle\bullet}\ \,&b_2'\\
&&a &{\scriptscriptstyle\bullet}\ar@{-}[dddrrrrr]&&&&&&&&&{\scriptscriptstyle\bullet}\ar@{-}[dddllll]&a\\
\\
\\
&&&& &&&&{\scriptscriptstyle\bullet}\ar@{-}[dddd]&\,b\\
\\
\\
\\
&&&& &&&&
}$}\newline
Note that the second interchange follows from the first by replacing $b$ by
$b\tau$, where $\tau$ is the transposition $(1,2)$, using the equivariance
relation. Similarly the fourth interchange follows from the third. Also note
that in the third interchange we can apply the simplification $a\circ(0\oplus0)=0$
(since $\scA(0)=\{0\}$). If we then reduce both sides of that interchange 
 (using unary interchanges and composition), we see that both sides become
the same. Thus the third (and hence the fourth) interchanges are superfluous.
This leaves only the first interchange. Moreover even this interchange can be
simplified by noting that $b_1'\circ 0=0$ and $b_2'\circ0=0$ (since $\scB(0)=\{0\}$).
This converts the first interchange to the following form\newline
\centerline{$
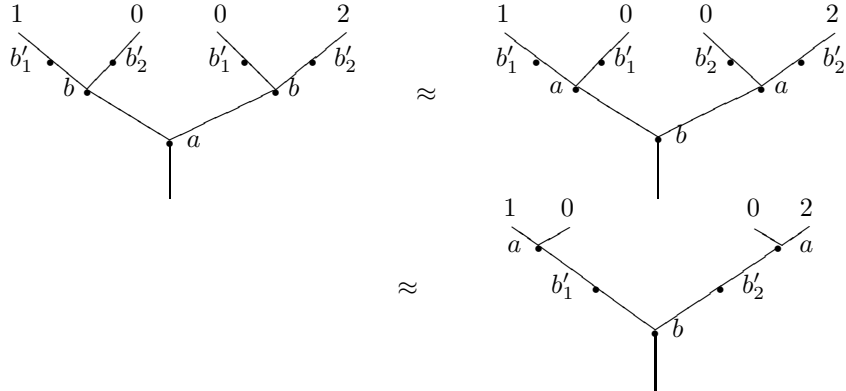
\xymatrix@=5pt@M=-1pt@W=-1pt{
1\ && &&&&0 &&&&0 && &&&2\\
\ar@{-}[ddrrr] && &&&&\ar@{-}[ddlll] &&&&\ar@{-}[ddrr] && &&&\ar@{-}[ddlll]\\
b_1'&\,{\scriptscriptstyle\bullet}&&&&{\scriptscriptstyle\bullet} &b_2'&& &&b_1'&{\scriptscriptstyle\bullet}\ \,&
&&{\scriptscriptstyle\bullet}\,&\,b_2'\\
&&b &{\scriptscriptstyle\bullet}\ar@{-}[dddrrrrr]&&&&&&&&&{\scriptscriptstyle\bullet}\ar@{-}[dddllll]&b\\
\\
\\
&&&& &&&&{\scriptscriptstyle\bullet}\ar@{-}[dddd]&\,a\\
\\
\\
\\
&&&& &&&&
}$
\raisebox{-30pt}{$\qquad\approx\qquad$}
$\xymatrix@=5pt@M=-1pt@W=-1pt{
1\ && &&&&0 &&&&0 && &&&2\\
\ar@{-}[ddrrr] && &&&&\ar@{-}[ddlll] &&&&\ar@{-}[ddrr] && &&&\ar@{-}[ddlll]\\
b_1'&\,{\scriptscriptstyle\bullet}&&&&{\scriptscriptstyle\bullet} &b_1'&& &&b_2'&{\scriptscriptstyle\bullet}\ \,&
&&{\scriptscriptstyle\bullet}\,&\,b_2'\\
&&a &{\scriptscriptstyle\bullet}\ar@{-}[dddrrrrr]&&&&&&&&&{\scriptscriptstyle\bullet}\ar@{-}[dddllll]&a\\
\\
\\
&&&& &&&&{\scriptscriptstyle\bullet}\ar@{-}[dddd]&\,b\\
\\
\\
\\
&&&& &&&&
}$}\newline
\centerline{\raisebox{-30pt}{$\qquad\qquad\qquad\qquad\qquad\qquad\quad\ \approx\qquad\quad$}
$\xymatrix@=5pt@M=-1pt@W=-1pt{
1\ &&0 &&&& &&&& && &0&&2\\
\ar@{-}[dr] &&\ar@{-}[dl] &&&& &&&& && &\ \ar@{-}[dr]&&\ar@{-}[dl]\\
a&{\scriptscriptstyle\bullet}\ar@{-}[ddddrrrrrr]&&&& &&& &&&&&&{\scriptscriptstyle\bullet}\ \ar@{-}[ddddlllllll]&a\\
&& &&&&&&&&& &\\
&&b_1'\ &{\scriptscriptstyle\bullet}\ \ &&&&&&&&&{\scriptscriptstyle\bullet}\ &b_2'\\
\\
&&&& &&&{\scriptscriptstyle\bullet}\ar@{-}[dddd]&\,b\\
\\
\\
\\
&&& &&&&
}$}\newline
Taking $b'=b\circ(b_1'\oplus b_2')$, this in turn simplifies to the interchange\newline
\centerline{$
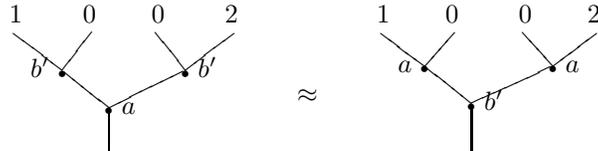
\xymatrix@=5pt@M=-1pt@W=-1pt{
1 && &&0 & && &0&& &&2\\
\ar@{-}[ddrr] &&&&\ar@{-}[ddll] && &&\ar@{-}[ddrr] && &&\ar@{-}[ddll]\\
\\
&b'&{\scriptscriptstyle\bullet}\ar@{-}[ddrrr]&&&&&&&&{\scriptscriptstyle\bullet}\ar@{-}[ddlllll]&b'\\
\\
&&&&&{\scriptscriptstyle\bullet}\ar@{-}[ddd]&a\\
\\
\\
&&&&&
}$
\raisebox{-30pt}{$\qquad\approx\qquad$}
$\xymatrix@=5pt@M=-1pt@W=-1pt{
1 && &&0 & && &0&& &&2\\
\ar@{-}[ddrr] &&&&\ar@{-}[ddll] && &&\ar@{-}[ddrr] && &&\ar@{-}[ddll]\\
\\
&a&{\scriptscriptstyle\bullet}\ar@{-}[ddrrr]&&&&&&&&{\scriptscriptstyle\bullet}\ar@{-}[ddlllll]&a\\
\\
&&&&&{\scriptscriptstyle\bullet}\ar@{-}[ddd]&b'\\
\\
\\
&&&&&
}$}\newline
This can be briefly summarized as follows. If an element 
$(a,b_1,b_2)\in G_2(1\boxtimes_1 2)\cong\scA(2)\times\scB(1)^2$
is the image of an element $(a,b)\in\scA(2)\times\scB(2)$ under the
map $(a,b)\mapsto\left(a,b\circ(id\oplus 0),b\circ(0\oplus id)\right)$,
then it is identified with the element
$\left(a\circ(id\oplus 0),a\circ(0\oplus id),b\right)$ in
$G_2(1\boxtimes_2 2)\cong \scA(1)^2\times\scB(2).$
A similar analysis of interchanges starting from an element
$G_2(1\boxtimes_2 2)$, shows that the only relation obtained
is the reverse of the above relation. It follows that
$(\scA\times\scB)(2)$ is a pushout as in the
second part of Proposition \ref{4_8}.

\section{Axial operads}\label{sec7a}

It is technically convenient to prove some of our results in the category $\SSets$ of 
simplicial sets. So throughout this section we work in the categories $\Top$ and $\SSets$ of $k$-spaces 
and simplicial sets respectively. The formulas given in this section make sense in $\Top$. If we
work in $\SSets$, they are meant to be applied degreewise.

\begin{leer}\label{7_1}
\textbf{Observation:} If $\scA$ and $\scB$ are simplicial operads, then
$$
|\scA\otimes \scB|\cong |\scA|\otimes |\scB|,
$$
because $\scA\otimes \scB$ is a colimit of a diagram involving finite products of
simplicial sets $\scA (k)$ and $\scB (l)$ and the realization functor preserves colimits
and finite limits.
\end{leer}

\begin{lem}\label{7_2}
(Igusa \cite{Igusa}) For both categories, the forgetful functor 
$$
U:\Opr_0\longrightarrow \mathcal{M}\rm{onoids},\qquad \mathcal{B}
\mapsto\mathcal{B}(1)
$$
has a right adjoint $R:\mathcal{M}\rm{onoids}\longrightarrow \Opr_0$, where
$\mathcal{M}\rm{onoids}$ is the category of monoids in $\Top$ respectively $\SSets$.
\end{lem}

\begin{proof}
The right adjoint $R:\mathcal{M}\rm{onoids}\longrightarrow \Opr_0$ is defined as follows:
$RM(k)=M^k$ for a monoid $M$. The symmetric group permutes the factors,
and composition is defined by

$
 (x_1,\ldots,x_n)\circ ((y_{11},\ldots,y_{1k_1})\oplus\ldots\oplus (y_{n1},
\ldots,y_{nk_n}))\\
{} \qquad\qquad = (x_1\cdot y_{11},\ldots,x_1\cdot y_{1k_1},\ldots,x_n\cdot y_{n1},\ldots,
x_n\cdot y_{nk_n}).
$\\
The unit of the adjunction
$$
\xi:\mathcal{B}\longrightarrow RU\mathcal{B}
$$
is given by the \textit{axial maps}
$$
\mathcal{B}(n)\longrightarrow\mathcal{B}(1)^n
$$
whose $i$-th coordinates are the compositions
$$
\xi_i :\mathcal{B}(n)\cong \mathcal{B}(n)\times
(\mathcal{B}(0)^{i-1}\times \{id\}\times \mathcal{B}(0)^{n-i})
\longrightarrow \mathcal{B}(1)
$$
[In other words, the $i$-th coordinate is the restriction map 
$$-\cap\{i\}:\mathcal{B}(n)\longrightarrow\mathcal{B}(\{i\})\cong\mathcal{B}(1)$$
of Definition \ref{2_3a}.]
The counit $URM\longrightarrow M$ is the identity map.
\end{proof}

$RU(\mathcal{B})$ is closely related to $\Ass \otimes \mathcal{B}$ and 
$\Com \otimes \mathcal{B}$. 
The following result is a reformulation of \cite[Thm. 5.5]{BV3}; the topological proof of \cite{BV3} also
works in $\SSets$.

\begin{prop}\label{7_3}
In $\Top$ or $\SSets$, if $\mathcal{B}\in \mathcal{O}pr_0$,
 then 
$$\Ass\otimes \mathcal{B}\cong (\Ass \times RU(\mathcal{B}))/\sim,$$
where the relation $\sim$ is defined on 
$$(\Ass \times RU(\mathcal{B}))(n)= \Sigma_n\times \mathcal{B}(1)^n$$
by $(\pi,(b_1,\ldots,b_n))\sim (\rho,(b_1,\ldots,b_n))$ iff $\pi^{-1}(i)<
\pi^{-1}(j)$ and $\rho^{-1}(i)>\rho^{-1}(j)$ imply that there is a $c\in 
\mathcal{B}(2)$ such that $(b_i, b_j)=\xi (c)$.
\hfill\ensuremath{\Box}
\end{prop}

Note that $(\Ass\otimes \mathcal{B})(1) = \mathcal{B}(1)$, so that
$ RU(\Ass\otimes \mathcal{B})= RU(\mathcal{B}).$

As a simple consequence we get

\begin{prop}\label{7_4}
In $\Top$ or $\SSets$, if $\mathcal{B}\in \mathcal{O}pr_0$,
 then $$\Com\otimes \mathcal{B}\cong RU(\mathcal{B}))$$
and the adjunction
map $\xi : \mathcal{B}\to RU(\mathcal{B})$ corresponds to the canonical
map $\mathcal{B}\to \Com\otimes \mathcal{B}$.
\end{prop}

\begin{proof}
Using Corollary \ref{3_7} and Proposition \ref{7_3} we obtain 
$$
\Com \otimes \mathcal{B}\cong \Ass\otimes \Ass\otimes \mathcal{B}
\cong (\Ass\times RU(\Ass\otimes \mathcal{B}))/\sim \ \cong
(\Ass\times RU(\mathcal{B}))/\sim
$$ 
Now $(\Ass\otimes \mathcal{B})(2)=\Sigma_2\times \mathcal{B}(1)^2/\sim$,
and a pair $(a,b)\in \mathcal{B}(1)^2$ is the image of $(id,(a,b))\in
(\Ass\otimes \mathcal{B})(2)$ under
$\xi$. Hence 
$$(\Ass\times RU(\mathcal{B})/\sim)(n)=(\Sigma_n\times \mathcal{B}(1)^n)/
\Sigma_n = \mathcal{B}(1)^n
$$
Let $\beta\in \mathcal{B}(n)$ and let $\lambda_n$ denote the unique element in
$\Com (n)$. Then by the same argument as in the proof of Proposition \ref{3_6},
we have $\beta=\lambda_n\circ (\xi_1(\beta)\oplus\ldots\oplus
\xi_n(\beta))$. Hence
the adjunction map $\xi$ corresponds to the canonical map
$\mathcal{B}\to \Com\otimes \mathcal{B}$.
\end{proof}

\begin{rema}\label{7_5} Proposition \ref{7_4} is a special case of the
following general result, which is a simple consequence of 
Proposition \ref{7_4} and Proposition \ref{4_8}: 
Let $\mathcal{A}$ and $\mathcal{B}$ be reduced operads and let $M$ and $N$ be
monoids in $\Top$ or $\SSets$. Then there are natural isomorphisms  
\begin{enumerate}
\item $R(M)\otimes \mathcal{A}\cong R(M\times U(\mathcal{A}))$
\item $U(\mathcal{A}\otimes\mathcal{B})\cong  
U(\mathcal{A}) \times U(\mathcal{B})$
\item $R(M\times N)\cong R(M) \times R(N)\cong R(M) \otimes R(N)$
\item $RU(\mathcal{A}\otimes\mathcal{B})\cong  RU(\mathcal{A})
\otimes RU(\mathcal{B})\cong  RU(\mathcal{A}) \times RU(\mathcal{B})$
\end{enumerate}
\end{rema}

\begin{defi}\label{7_6}
A $\Sigma$-free topological operad $\mathcal{B}$ is called \textit{axial} if 
$\xi:\mathcal{B}\longrightarrow RU\mathcal{B}$ is 
a closed cofibration. A $\Sigma$-free simplicial operad $\mathcal{B}$ is 
called \textit{axial} if
$\xi:\mathcal{B}\longrightarrow RU(\mathcal{B})$ is injective. 
(In these cases we will consider $\mathcal{B}(n)$ as a subspace of
$\mathcal{B}(1)^n$ in the sequel). 
\end{defi}

\section{W-constructions}\label{sec7}

For the reader's convenience and for notational reasons we briefly recall the 
$W$-constructions we are going to use. They were originally defined for
well-pointed topological operads by Boardman and Vogt \cite{BV2} and extended to 
other categories including  $\SSets$ by Berger and Moerdijk \cite{BM}. 

\begin{leer}\label{7_7} \textbf{The non-reduced W-construction:}
In this subsection we allow non-reduced operads.

We define a functor
$$W^u:\Opr\to \Opr$$
where $\Opr$ is the category of all topological operads, together with a natural
transformation $\varepsilon :W^u\to \Id$.
\end{leer}

Let $\scB$ be a topological operad. An element in $W^u\scB(n)$ 
 is an equivalence class of quadruples $(\psi ,f,g,h)$ 
consisting of 
\begin{itemize}
\item  a finite directed rooted planar tree $\psi$.  
Each vertex $v$ has a 
finite set $\In(v)$ 
of incoming edges and exactly one outgoing edge. 
$\In(v)=\emptyset$ is allowed. Thus
 $\psi$ has a finite set $\In(\psi)$ of \textit{inputs} and exactly
 one \textit{output}, the \textit{root}. The inputs and the root are called
\textit{external} edges, all other edges are called \textit{internal}; they have a vertex
at both ends. We require that $|\In(\psi)| =n$, 
where $|M|$ denotes the cardinality of $M$.
\item  a function $f$ assigning to each vertex $v$ a label 
$f(v)\in\scB(|\In(v)|)$.
\item a bijection $g:\In(\psi )\to\underline{n}=\{1,\ldots,n\}$, where 
$n=|\In(\psi)|$. We call $g(i)$ an \textit{input-label} and think of it as label of input $i$.
\item a function $h$ assigning to each internal edge a \textit{length} $t\in [0,1]$. By convention,
the outer edges have length $1$.
\end{itemize}
We usually suppress $f,g,h$ from the notation and think of $(\psi ,f,g,h)$ as a 
tree $\psi$ with labeled inputs, edges, and vertices. 
We allow the \textit{trival tree}, i.e. a single edge with label $1$
(direction is from top to bottom), and call vertices without an input a
 \textit{stump}. 

\begin{leer}\label{7_8}
The equivalence relation between such trees is generated by 
\vspace{-1ex}
 \begin{enumerate}
  \item \textit{The equivariance relation:} The relation \ref{5_1b}
applied to labeled trees,
\item \textit{The identity relation:} 

\centerline{
$
\xymatrix@=5pt@M=-1pt@W=-1pt{
\ar@{-}[dddd]^{\displaystyle{\epsilon_1}}\\
\\
\\
\\
\bullet\save[]+<10pt,0pt>*{id}="a"\restore\ar@{-}[dddd]^{\displaystyle
{\epsilon_2}}\\
\\
\\
\\
\\
}$
\raisebox{-22pt}{$\qquad\sim\qquad$}
\raisebox{-2pt}{$\xymatrix@=5pt@M=-1pt@W=-1pt{
\ar@{-}[dddddddd]^{\mbox{max}(\displaystyle{\epsilon_1},\epsilon_2)}\\
\\
\\
\\
\\
\\
\\
\\
\\
}$}
}
where $\epsilon_1$ and $\epsilon_2$ are the lengths of the edges.
\item \textit{The shrinking relation:} Edges of length $0$ may be shrunk composing the labels
of the vertices at their ends using the composition in the operad $\scB$.
\end{enumerate}
\end{leer}
The operad structure is given by tree composition: 
The composite tree $\varphi\circ(\psi_1\oplus\ldots\oplus\psi_n)$ is 
obtained by grafting $\psi_i$ on the input of $\varphi$ with input-label $i$. 
The newly created inner edges obtain the labels 1.

The natural transformation 
$$\varepsilon : W^u\scB \to \scB$$
is defined by shrinking all internal edges to $0$. The map $\varepsilon(n): W^u\scB(n) \to \scB(n)$ has
a section $\eta(n)$ defined by sending $\alpha \in \scB(n)$ to the element represented by the tree with
one vertex labeled $\alpha$ and $n$ inputs labeled from left to right in increasing order. The homotopy
$h_t$ shrinking the lengths of internal edges by the factor $t$ deforms $W^u\scB(n)$ fiberwise into the section.
In particular, $\varepsilon$ is a weak equivalence.

If the inclusion $\{id\}\subset \scB(1)$ is a closed cofibration (i.e. the operad is \textit{well-pointed}) then
$\varepsilon :W^u\scB\to \scB$ is a cofibrant replacement in the sense of \ref{4_2} applied in the category $\Opr$.

The same construction can be carried out in the category $\SSets$ for simplicial operads with the unit
interval $[0,1]$ being replaced by the standard simplicial $1$-simplex. 
For a detailed account see \cite{BM}
where it is also shown that the topological realization of the simplicial $W^u$-construction
applied to a simplicial operad $\scB_\bullet$ is naturally homeomorphic to the topological $W^u$-construction
 applied to its realization $|\scB_\bullet|$:
$$|W^u\scB_\bullet|\cong W^u|\scB_\bullet|$$
Since every simplicial operad is well-pointed, $\varepsilon : W^u\scB_\bullet\to \scB_\bullet$ is a cofibrant
replacement in the category of simplicial operads. 

\begin{leer}\label{7_11} \textbf{The reduced W-construction:}
 The inclusion $\Opr_0\subset \Opr$ has a left adjoint $L$: For an operad $\scB$ let $Null(\scB)\subset \scB$ be
the suboperad generated by $id\in\scB(1)$ and by $\scB(0)$, and let $Null$ be the unique operad with $Null(0)$ and $Null(1)$
consisting of a point. There is a unique map of operads $Null(\scB)\to Null$, and $L(\scB)$ is the pushout of
$$Null\leftarrow Null(\scB)\xrightarrow{\subset} \scB$$
in the category of operads.  

The reduced $W$-construction is the composite functor 
$$W: \Opr_0\subset \Opr\xrightarrow{W^u} \Opr\xrightarrow{L} \Opr_0.$$
Hence $W\scB$ is obtained from $W^u\scB$ by freely adjoining an unlabeled stump and imposing the
\begin{leer}\label{7_12}
\textit{Stump relation:} If a labeled tree $\psi$ (unlabeled stumps on edges of length 1 are allowed) has an edge labeled $1$ such that
the subtree above that edge has no inputs (i.e. it is a subtree topped by stumps) then
this subtree may be replaced by the unlabeled stump.
\end{leer}
By adjointness $W$ is a cofibrant replacement functor for well-pointed operads in $\Opr_0$. 

There is also  a simplicial version of $W$ obtained from $W^u$ in the same way as in the topological case. 
Since realization preserves quotients, the 
natural homeomorphism of the $W^u$-construction induces a natural homeomorphism
$$|W\scB_\bullet|\cong W|\scB_\bullet|$$ for a reduced simplicial operad $\scB_\bullet$.
\end{leer}

Since our favorite
cofibrant $E_k$ operad is $W|\scN\scM_k|\cong |W\scN\scM_k|$ we will work in the topological and the simplicial category. We 
often find it easier to argue topologically and to transfer the results to the simplicial case using the natural homeomorphisms
between the two $W$-constructions. On the other hand, cofibration conditions in our later constructions reduce to checking injectivity
in the the simplicial category while they require proofs, which can be quite involved, in the topological category.

\begin{rema}\label{7_10}
 In \cite{BM} and \cite {BV2} the symbol $W$ is used 
instead of $W^u$. We reserve $W$ for the reduced version which corresponds to
the $W'$-construction of \cite[p. 159]{BV2} and is not treated in \cite{BM}. In \cite{BM} the cofibrancy of the
topological as well as the simplicial version of $W^u\scB$ requires that $\scB$ is $\Sigma$-free. The reason is that
their notion of a weak equivalence is weaker than our one, but for realizations of $\Sigma$-free simplicial operads the two agree.
\end{rema}

As pointed out in \ref{2_2} we often think of $W\scB$ as a topologically enriched
symmetric monoidal category. Then composition is defined by the operad composition.
So composites are represented by trees having inner edges of lengths $1$.

\begin{nota}\label{7_14a}
We call a representing tree of an element in $W\scB_\bullet$ or $W\scC$ \textit{reduced}
if it cannot be further reduced by applying the identity, stump, 
or, in case of a 
topological operad $\scC$, the shrinking relation.\\
An \textit{input path} of a tree is a directed edge path from an input to the root.
\end{nota}

\begin{prop}\label{7_15} (1) For a topological operad $\scC$ the axial map
$\xi :W\scC(n)\to (W\scC(1))^n$ is injective.\\
(2) $W\scB_\bullet$ and $W|\scB_\bullet|$ are axial operads for any simplicial operad $\scB_\bullet$..
\end{prop}
\begin{proof} 
(1) Since $\scC$ is reduced the space
$W^u\scC(0)$ has a canonical base point given by the stump. Hence there is an ``axial map'' 
$$\xi^u:W^u\scC(n)\to (W^u\scC(1))^n$$
defined by composing with stumps. This map is clearly injective. In the reduced case, we
consider the image of a reduced representing tree. Here
the stump relation may reduce the components of the
axial image of this tree further. Nevertheless one can recover the original tree from these
reduced components, because if $v$ is the vertex of the $i$-th input path of the original tree where it 
meets another input path the first time, then the subtree above $v$
containing the $i$-th input path is contained in the $i$-th coordinate
of the axial image.

(2) Since a simplicial map is injective iff its realization is injective (e.g. see \cite[II.1.7]{Lamotke})
the axial map $\xi :W\scB_\bullet(n)\to (W\scB_\bullet(1))^n$ is injective by Part (1) so that
$W\scB_\bullet$ is axial. Since the realization of an injective simplicial map is a closed cofibration
$W|\scB_\bullet|$ is axial, too.
\end{proof}

\begin{defi}\label{7_16}
A monoid $M$ is called \textit{left-factorial} 
if any $(x_1,\ldots,x_n)\in M^n$
has a maximal left factorization $x_i=y\cdot x'_i$, 
$i=1,\ldots,n$: if
$x_i=z\cdot x''_i$ is any other factorization, 
$i=1,\ldots,n$, it can be factored further as
$$
x''_i=w\cdot x'_i, \quad i=1,\ldots,n \textrm{ with } y=z\circ w
$$
We call $y=\mlf(x_1,\ldots,x_n)$ a \textit{maximal left factor}
of $(x_1,\ldots,x_n)$.\\
An operad $\scB$ is called \textit{left-factorial} if the monoid $\scB(1)$ is left-factorial.
\end{defi}

\begin{rema}
By definition, a maximal left factor is unique up to multiplication by invertible elements.
For most of our operads $\scA$, among them $W\scB$ for any operad $\scB$, the group of
invertible elements in $\scA(1)$ is trivial, so that maximal left factors are unique.
This justifies the notation $y=\mlf(x_1,\ldots,x_n)$.
\end{rema}

\begin{lem}\label{LemmaNull} Let $\scB$ be a topological operad.
 Then $W\scB$ is left factorial.
\end{lem}
\begin{proof}
 Since $\mlf(\mlf(\alpha_1,\alpha_2),\alpha_3)=\mlf(\alpha_1,\alpha_2,\alpha_3)$
it suffices to show the existence of a maximal left factorization for a pair $(\alpha_1,\alpha_2)
\in W\scB(1)^2$. 
Let $A$ be a reduced tree representing $\alpha_1$. The left factors of $\alpha_1$
are represented by subtrees $T$ obtained from $A$ by deleting the subtree on top
of an edge of length $1$ of the input path. The height of $T$ is defined to be the 
length of its input path. Then $\mlf(\alpha_1,\alpha_2)$ is represented by the subtree
$T$ of $A$ of maximal height which also represents a left factor of $\alpha_2$.
\end{proof}

\begin{lem}\label{LemmaZwei} Let $\scB$ be a topological operad.
If $\alpha_1,\ldots,\alpha_n\in W\scB(1),\ n\geq 2$, then there is
a pair $i<j$ such that
$$\mlf(\alpha_1,\ldots,\alpha_n)=\mlf(\alpha_i,\alpha_j)$$
\end{lem}
\begin{proof}
This follows from $\mlf(\mlf(\alpha_1,\alpha_2),\alpha_3)=\mlf(\alpha_1,\alpha_2,\alpha_3)$.
\end{proof}

\begin{lem}\label{LemmaEins}
Let $\scB$ be a topological operad such that $\scB(1)=\{id\}$. Given $\alpha_1,\ldots,\alpha_n\in
 W\scB(1)$ such that
$(\alpha_i,\alpha_j)\in  W\scB(2)$ for each pair $i<j$, then $(\alpha_1,\ldots,\alpha_n)\in W\scB(n)$.
\end{lem}

 Let $J$ denote the category of subsets of $\{1,\ldots,n\}$ of cardinality $1$ or $2$ and inclusions
as morphisms. Let $P$ denote the limit of the $J^{op}$ diagram in $\Top$ sending $\{i\}$ to $W\scB(1)$, and
$\{i,j\}$ to $W\scB(2)$. The two maps $W\scB(2) \to W\scB(1)$ corresponding to the
inclusions of $\{i\}$and $\{j\}$ into $\{i,j\}$ are given by the components of the axial map. The lemma states that there is a
surjection $W\scB(n)\to P$. 
\begin{proof}
We proceed by induction.
For $n=2$ the statement is trivial. Now suppose that $n\geq 3$. 

Throughout  our proof
we will argue with reduced trees representing our elements. Let $T_i$ 
represent $\alpha_i$. By assumption there are reduced trees $S_{ij}$ with two
inputs, mapped by the axial map to $(T_i,T_j)$ (we may have to reduce the image). Let $s_{ij}$ denote the node
where the  input paths of $S_{ij}$ meet. The height
of $s_{ij}$ is the number of edges between $s_{ij}$ and the root along an input path. By renumbering the $\alpha_i$ we
may assume that $s_{1n}$ has the minimal height of all the $s_{ij}$. 

By induction there is a reduced tree $U$ with $n-1$ inputs whose axial image is $(T_1,\ldots,T_{n-1})$. Let $u$ be its node of lowest height where
two input paths meet. Then $U$ has the form

$$
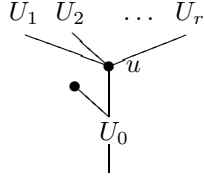
\xymatrix@=5pt@M=-1pt@W=-1pt{
U_1\ &U_2\ & &\dots &\ \,U_r\\
\ar@{-}[ddrr] &\ar@{-}[ddr] &&&\ar@{-}[ddll]\\
\\
&&\bullet\ar@{-}[ddd]&\save[]+<-2pt,0pt>*{u}\restore\\
&\bullet\ar@{-}[ddr]\\
\\
&&\ \\
&&\save[]+<2pt,-1pt>*{U_0}\restore\\
&&\ar@{-}[dd]\\
&&\\
&&
}$$
where the stump stands for a subtree topped by stumps if its outgoing edge has length $t,\ 0<t<1$. Of course, there may
be more such stumps. 
We may assume that the input path labeled $1$ passes through $U_1$. Then $T_1$ is obtained from $U$ by
replacing $U_1$ by $U_1'$, obtained from $U_1$ by putting stumps on all inputs of $U_1$ except of input $1$ and reducing,
and by putting stumps on $U_2,\ldots,U_r$ and reducing. In particular, $T_1$ contains the node $u$ and the subtree of $U_0$
below it.

By assumption, there is a tree $V\in \scT\scB(2)$ whose axial image is $(T_1,T_n)$. 
So $V$ contains the edge path of $U$ from $u$ to the root.
By minimality of $s_{1n}$ the two input
paths of $V$, which we denote by $1$ and $n$, meet in this edge path.
If they meet in a vertex
$v$ below $u$ consider the incoming edge of $v$ which lies on the $n$-th edge path of $V$ and the subtrees $U_0'$ of $U_0$
and $V'$ of $V$ on top of it. Then you obtain $U_0'$ from $V'$ by putting a stump on its single input and reduction. So
if we replace $U_0'$ in $U$ by $V'$ we obtain a tree whose axial image is $(T_1,\ldots,T_n)$.
We can proceed in the same way if the two input paths meet in $u$ and input path $n$
of $V$ meets $u$ at an input corresponding to a stump in $U$. This procedure only fails if input path $n$ meets
$u$ in an input which is topped by a tree $U_i$ with $1<i\leq r$. 
Suppose that the input paths $i_1,\ldots, i_q$ of $U$ pass through $U_i$. Let $(T_{i_1}',\ldots,T_{i_q}')$ be the axial image
of $U_i$ and let $V'$ be the subtree of $V$ sitting on top of the input to $u$ below $U_i$.
Then $T_j$ has the form
$$
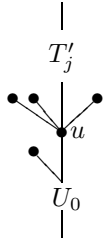
\xymatrix@=5pt@M=-1pt@W=-1pt{
&&\ar@{-}[dd]\\
\\
&&\\
&&T_j'\\
&&\ar@{-}[ddd]\\
\bullet\ar@{-}[ddrr] &\bullet\ar@{-}[ddr]&&&\bullet\ar@{-}[ddll]\\
\\
&&\bullet\ar@{-}[ddd]&\save[]+<-2pt,0pt>*{u}\restore\\
&\bullet\ar@{-}[ddr]\\
\\
&&\ \\
&&\save[]+<2pt,-1pt>*{U_0}\restore\\
&&\ar@{-}[dd]\\
&&\\
&&
}$$ 
for all $j\in \{i_1,\ldots ,i_q,n\}$ where the stumps again stand for subtrees topped by stumps. 
The condition $(\alpha_k,\alpha_l)\in W\scB(2)$ for all $k\neq l$ therefore implies that
$(T_k',T_l')$ represents an element in $W\scB(2)$ for all $k,\ l\in \{i_1,\ldots ,i_q,n\},\ k\neq l$. 
By induction there is a tree $X$ with $q+1$ inputs
whose axial image is $(T_{i_1}',\ldots,T_{i_q}',T_n')$. If we replace $U_i$ in $U$ by $X$, we obtain 
a tree representing an element in $W\scB(n)$ whose axial image is $(T_1,\ldots,T_n)$.
\end{proof}
\begin{lem}\label{LemmaDrei}
Let $\scB$ be a $\Sigma$-free topological operad. If $\alpha,\beta,\gamma,\delta \in W\scB(1)$ satisfy that
$(\alpha\circ \beta, \gamma)$ and $(\alpha,\gamma \circ \delta)$ are in $W\scB(2)$, then $(\alpha,\gamma)$ is
in $ W\scB(2)$ (for $(\alpha,\gamma)\in W\scB(2)$ see our convention in Definition \ref{7_6}).
\end{lem}
\begin{proof}
 Let $A,B,G,D$ be reduced trees representing $\alpha,\beta,\gamma,\delta$. By assumption there are trees $S$ and $T$
whose axial image is $(A\circ B,G)$ respectively $(A,G\circ D)$. In $S$ the input paths corresponding to $A\circ B$ and
$G$ can either meet in the $A$-part or the $B$-part. If they meet in the $A$-part we delete the $B$-part from $S$ and obtain
a tree whose axial image is $(A,G)$. If they meet in the $B$-part, then $G$ is of the form $G=A\circ G'$. But since 
$(\alpha\circ\rho,\alpha)\notin W\scB(2)$ for all $\rho\in W\scB(1)$, the condition $(\alpha,\gamma \circ \delta)\in
W\scB(2)$ then cannot hold.
\end{proof}

\section{Binodal trees}\label{sec8}

In this section we will discuss a tree notation which is very convenient in analyzing
the colimit decomposition of the tensor product $W\mathcal{NM}_k\otimes W\mathcal{NM}_l$.

\begin{defi} A \textit{binodal tree} is a \textbf{nonplanar} rooted tree with labeled inputs,
no stumps and no unary nodes. 
Also all
the nodes are colored either black or white and there are no edges connecting two
white nodes. [If some construction below gives rise to an edge connecting two white
nodes, it is to be understood that such an edge is to be shrunk to a point and the
two white nodes merged together.]
\end{defi}

When we draw pictures of binodal trees, we will represent black nodes by solid dots $\bullet$
and white nodes by hollow dots $\circ$. We will also use the following convention: if $T$
is a binodal tree, then $|T|$ will denote the set of inputs for $T$.

\begin{defi}\label{8_1} Let $\scA$ be an operad in one of our categories $\mathcal{S}$ 
and let $T$ be a binodal tree. We define the object $\scA(T)$
with a left $\scA(1)$ action recursively as follows.
\begin{itemize}
\item If $|T|$ has cardinality 1, then $\scA(T)=\scA(1)$ with left action
given by multiplication.
\item If the bottom node of $T$ is colored white and thus $T$ looks like
$$
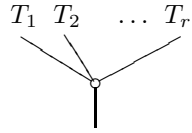
\xymatrix@=5pt@M=-1pt@W=-1pt{
T_1 &&T_2 &&&\dots &&T_r\\
\ar@{-}[dddrrr] &&\ar@{-}[dddr] &&& &&\ar@{-}[dddllll]\\
\\
\\
&&&\circ\ar@{-}[ddd]&&&\\
\\
\\
&&&
}$$
then 
$$\scA(T)=\scA(T_1)\times\scA(T_2)\times\dots\times\scA(T_r)$$
with the diagonal left action of $\scA(1)$ on the right hand side.
\item If the bottom node of $T$ is colored black and thus $T$ looks like
$$
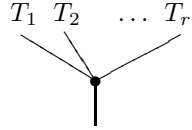
\xymatrix@=5pt@M=-1pt@W=-1pt{
T_1 &&T_2 &&&\dots &&T_r\\
\ar@{-}[dddrrr] &&\ar@{-}[dddr] &&& &&\ar@{-}[dddllll]\\
\\
\\
&&&\bullet\ar@{-}[ddd]&&&\\
\\
\\
&&&
}$$
then
$$\scA(T)=\scA(r)\times_{\scA(1)^r}\scA(T_1)\times\scA(T_2)\times\dots\times\scA(T_r)$$
with left $\scA(1)$ action coming from its left action on $\scA(r).$
\end{itemize}
\end{defi}
\begin{rema}\label{8_3a}
 If we work in $\Top$ or $\Sets$ it is helpful to think of an object in $\scA(T)$
in the following way. We extend this point of view to $\SSets$ by applying it
degreewise.

An element of $\scA(T)$ is a \textbf{planar} binodal tree $T$ in which each black node $v$ is decorated
with an element of $\scA(\In (v))$ and each edge with an element of $\scA(1)$ subject
to the following relations:
\begin{enumerate}
\item[(1)] The equivariance relation \ref{5_1b} (here we give each white node a dummy decoration
which is invariant under the action of the symmetric group).
\item[(2)] If an edge below a black node has decoration $\alpha\cdot\beta$ then $\beta$
can be moved into the black node by composing with the node's decoration.\newline
\centerline{
$
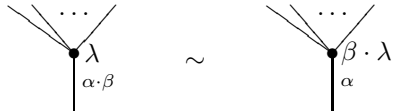
\xymatrix@=5pt@M=-1pt@W=-1pt{
\ar@{-}[dddrrrrrr] &&&\ar@{-}[dddrrr] &&&\dots& &&\ar@{-}[dddlll]\\
\\
\\
&& && &&\bullet\ar@{-}[dddd]^{\alpha\cdot\beta}&\save[]+<-2pt,0pt>*{\lambda}\restore\\
\\
\\
\\
&&&&&&\\
}
$
\raisebox{-20pt}{\qquad$\sim$\qquad}
$
\xymatrix@=5pt@M=-1pt@W=-1pt{
\ar@{-}[dddrrrrrr] &&&\ar@{-}[dddrrr] &&&\dots& &&\ar@{-}[dddlll]\\
\\
\\
&& && &&\bullet\ar@{-}[dddd]^{\alpha}&\save[]+<4pt,0pt>*{\beta\cdot\lambda}\restore\\
\\
\\
\\
&&&&&&\\
}
$
}
\item[(3)] If an edge above a black node has decoration $\alpha\cdot\beta$ then $\alpha$
can be moved into the black node by composing in the canonical way with the node's decoration.
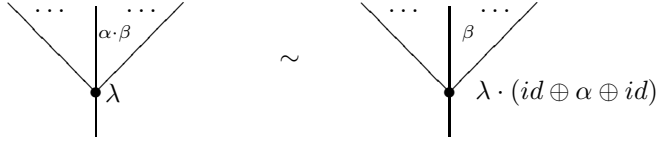
\newline
\centerline{
$
\xymatrix@=5pt@M=-1pt@W=-1pt{
\ar@{-}[ddddddrrrrrr] &&&\dots &&&\ar@{-}[dddddd]^{\save[]+<-25pt,6pt>*{\scriptstyle\alpha\cdot\beta}\restore} &&&\dots& &&\ar@{-}[ddddddllllll]\\
\\
\\
\\
\\
\\
&&& && &\bullet\ar@{-}[ddd]&\save[]+<1pt,0pt>*{\lambda}\restore&&&&&\\
\\
\\
&&&&&&&&&&&&&&&&&\\
}
$
\raisebox{-20pt}{\qquad$\sim$\qquad}
$
\xymatrix@=5pt@M=-1pt@W=-1pt{
\ar@{-}[ddddddrrrrrr] &&&\dots &&&\ar@{-}[dddddd]^{\save[]+<-25pt,6pt>*{\scriptstyle\beta}\restore} &&&\dots& &&\ar@{-}[ddddddllllll]\\
\\
\\
\\
\\
\\
&&& && &\bullet\ar@{-}[ddd]&\save[]+<39pt,0pt>*{\lambda\cdot(id\oplus\alpha\oplus id)}\restore&&&&&\\
\\
\\
&&&&&&&&&&&&&&&&&\\
}
$
}
\item[(4)] If an edge below a white node has decoration $\alpha\cdot\beta$ then $\beta$
can be moved into all edges above that node simultaneously by composing with the decorations
of these edges.\newline
\centerline{
$
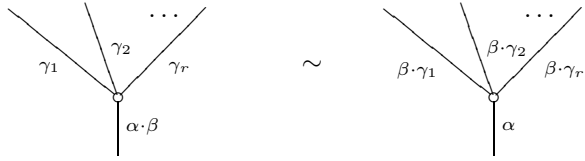
\xymatrix@=5pt@M=-1pt@W=-1pt{
\ar@{-}[ddddddrrrrrr]_{\gamma_1}  &\qquad&&\ar@{-}[ddddddrrr]^{\gamma_2} &&&&&&\dots& &&\ar@{-}[ddddddllllll]^{\gamma_r}\\
\\
\\
\\
\\
\\
&&& && &\circ\ar@{-}[dddd]^{\alpha\cdot\beta}&&&&&&\\
\\
\\
\\
&&&&&&&&&&&&&&&&&\\
}
$
\raisebox{-20pt}{\qquad$\sim$\qquad}
$
\xymatrix@=5pt@M=-1pt@W=-1pt{
\ar@{-}[ddddddrrrrrr]_{\beta\cdot\gamma_1}  &\qquad&&\ar@{-}[ddddddrrr]^{\beta\cdot\gamma_2} &&&&&&\dots& &&\ar@{-}[ddddddllllll]^{\beta\cdot\gamma_r}\\
\\
\\
\\
\\
\\
&&& && &\circ\ar@{-}[dddd]^{\alpha}&&&&&&\\
\\
\\
\\
&&&&&&&&&&&&&&&&&\\
}
$
}
\end{enumerate}
If $\scS=\Cat$, we can similarly think of the objects and morphisms of $\scA(T)$ as binodal trees decorated
with appropriate objects or morphisms of $\scA$, subject to the above relations.
\end{rema}

\begin{rema}
While we are primarily concerned with the above construction for topological and simplicial
operads, we will also need to consider this construction for operads in $\Cat$. In
particular, we will use the posets $\mathcal{M}_k(T)$ and $\mathcal{M}_l(T)$ as indexing
categories in our colimit decomposition of $W|\mathcal{NM}_k|\otimes W|\mathcal{NM}_l|$.
In this case the description in Remark \ref{8_3a} simplifies considerably, since $\mathcal{M}_k$
and $\mathcal{M}_l$ have no nontrivial unary operations. Thus there is no need for edge labels
or the relations (2)-(4). 
\end{rema}

\begin{prop}\label{8_2} Let $\scA$ be an axial operad and $T$ a binodal tree.
Then there is a natural imbedding $\scA(T)\subseteq RU\scA(|T|)=\scA(1)^{|T|}$.
\end{prop}

\begin{proof} We proceed inductively on the cardinality of $|T|$.
If $|T|$ has cardinality 1, then $\scA(T)=\scA(1)=\scA(1)^{|T|}$ by definition. Assuming we have
established the result for binodal trees whose sets of inputs have cardinality less than $|T|$,
we consider two cases.

If the bottom node of $T$ is colored white and thus $T$ looks like
$$
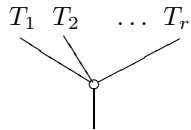
\xymatrix@=5pt@M=-1pt@W=-1pt{
T_1 &&T_2 &&&\dots &&T_r\\
\ar@{-}[dddrrr] &&\ar@{-}[dddr] &&& &&\ar@{-}[dddllll]\\
\\
\\
&&&\circ\ar@{-}[ddd]&&&\\
\\
\\
&&&
}$$
then by induction we have
$$\scA(T)=\scA(T_1)\times\scA(T_2)\times\dots\times\scA(T_r)\subseteq\prod_{i=1}^r\scA(1)^{|T_i|}=\scA(1)^{|T|}$$

If the bottom node of $T$ is colored black and thus $T$ looks like
$$
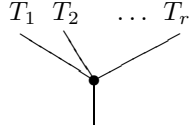
\xymatrix@=5pt@M=-1pt@W=-1pt{
T_1 &&T_2 &&&\dots &&T_r\\
\ar@{-}[dddrrr] &&\ar@{-}[dddr] &&& &&\ar@{-}[dddllll]\\
\\
\\
&&&\bullet\ar@{-}[ddd]&&&\\
\\
\\
&&&
}$$
then by induction hypothesis and the axiality of $\scA$
$$\scA(T)=\scA(r)\times_{\scA(1)^r}\scA(T_1)\times\scA(T_2)\times\dots\times\scA(T_r)\subseteq
\scA(1)^r\times_{\scA(1)^r}\prod_{i=1}^r\scA(1)^{|T_i|}\cong\scA(1)^{|T|}$$
where the last map is given by composition in the operad $RU\scA$.
\end{proof}

\begin{coro}\label{8_3}Suppose $\scA$ is an axial and left factorial topological or simplicial operad. 
Then for any binodal tree $T$, $\scA(T)$ can be identified with the subspace $\scA^*(T)$ of 
$RU\scA(|T|)=\scA(1)^{|T|}$ defined recursively as follows:
\begin{itemize}
\item If $|T|$ has cardinality 1, we define $\scA^*(T)=RU\scA(|T|)\cong\scA(1)$.
\item If the bottom node of $T$ is colored white and thus $T$ looks like
$$
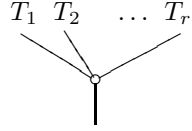
\xymatrix@=5pt@M=-1pt@W=-1pt{
T_1 &&T_2 &&&\dots &&T_r\\
\ar@{-}[dddrrr] &&\ar@{-}[dddr] &&& &&\ar@{-}[dddllll]\\
\\
\\
&&&\circ\ar@{-}[ddd]&&&\\
\\
\\
&&&
}$$
we define
$$\scA^*(T)=\scA(T_1)\times\scA(T_2)\times\dots\times\scA(T_r)
\subseteq\prod_{i=1}^r\scA(1)^{|T_i|}\cong\scA(1)^{|T|}$$
\item If the bottom node of $T$ is colored black and thus $T$ looks like
$$
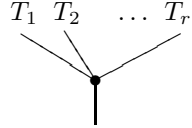
\xymatrix@=5pt@M=-1pt@W=-1pt{
T_1 &&T_2 &&&\dots &&T_r\\
\ar@{-}[dddrrr] &&\ar@{-}[dddr] &&& &&\ar@{-}[dddllll]\\
\\
\\
&&&\bullet\ar@{-}[ddd]&&&\\
\\
\\
&&&
}$$
we define $\scA^*(T)$ to be the subspace of
$\scA(T_1)\times\scA(T_2)\times\dots\times\scA(T_r)
\subseteq\prod_{i=1}^r\scA(1)^{|T_i|}\cong\scA(1)^{|T|}$
consisting of those $r$-tuples $(\overline{a}_1,\overline{a}_2,\dots,\overline{a}_r)$
such that the corresponding $r$-tuple of maximal left factors $(a_1,a_2,\dots,a_r)$
is in the image of the axial map $\scA(r)\longrightarrow\scA(1)^r$. Here $a_i$ is the
maximal left factor of $\overline{a}_i$.
\end{itemize}
\end{coro}

This follows immediately from the proof of Proposition \ref{8_2}.\hfill\ensuremath{\Box}

In view of Proposition \ref{7_15} and Lemma \ref{LemmaNull}, Corollary \ref{8_3} applies
when  ${\scA=W|\scB_\bullet|}$, where $\scB_\bullet$ is a simplicial operad with $\scB_\bullet(1)=\{id\}$.
Thus in what follows below, $(W|\scB_\bullet|)(T)$ will always mean the subspace $(W|\scB_\bullet|)^*(T)$
of $(W|\scB_\bullet|)(1)^{|T|}$ as described in  Corollary \ref{8_3} for any binodal tree $T$.

\begin{prop}\label{8_4} 
Let $\scB_\bullet$ be a simplicial operad with $\scB_\bullet(1)=\{id\}$. Let $T$ be a binodal tree. Then
$(W|\scB_\bullet|)(T)$ can be identified with the subspace of $(W|\scB_\bullet|)(1)^{|T|}$
consisting of all tuples $\left(\alpha_s\right)_{s\in |T|}$ such that, for any subset $\{a,b,c\}\subseteq |T|$ (of not necessarily distinct elements),
$${\left(\alpha_s\right)_{s\in |T|}\cap \{a,b,c\}\in (W|\scB_\bullet|)(T\cap\{a,b,c\})}.$$
\end{prop}

\begin{proof} Let $(W|\scB_\bullet|)'(T)$ denote the subspace of $(W|\scB_\bullet|)(1)^{|T|}$
consisting of all tuples $\left(\alpha_s\right)_{s\in |T|}$ such that, for any subset $\{a,b,c\}\subseteq |T|$ (of not necessarily distinct elements),
$${\left(\alpha_s\right)_{s\in |T|}\cap \{a,b,c\}\in (W|\scB_\bullet|)(T\cap\{a,b,c\})}.$$
The inclusion of $(W|\scB_\bullet|)(T)\subseteq (W|\scB_\bullet|)'(T)$ follows by naturality.
We prove the reverse inclusion by induction on the cardinality of $|T|$. If the cardinality of $|T|$ is $\le 3$,
this holds trivially. So assume the assertion is true for binodal trees $T'$ with $|T'|$ having cardinality less
than $|T|$.

If the bottom node of $T$ is colored white and thus $T$ looks like
$$
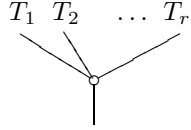
\xymatrix@=5pt@M=-1pt@W=-1pt{
T_1 &&T_2 &&&\dots &&T_r\\
\ar@{-}[dddrrr] &&\ar@{-}[dddr] &&& &&\ar@{-}[dddllll]\\
\\
\\
&&&\circ\ar@{-}[ddd]&&&\\
\\
\\
&&&
}$$
then it follows by induction that
\begin{eqnarray*}
(W|\scB_\bullet|)'(T)&= &(W|\scB_\bullet|)'(T_1)\times (W|\scB_\bullet|)'(T_2)\times\dots\times (W|\scB_\bullet|)'(T_r)\\
&= &(W|\scB_\bullet|)(T_1)\times (W|\scB_\bullet|)(T_2)\times\dots\times (W|\scB_\bullet|)(T_r)\\
&= &(W|\scB_\bullet|)(T)
\end{eqnarray*}

If the bottom node of $T$ is colored black and thus $T$ looks like
$$
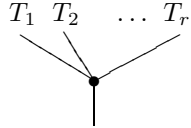
\xymatrix@=5pt@M=-1pt@W=-1pt{
T_1 &&T_2 &&&\dots &&T_r\\
\ar@{-}[dddrrr] &&\ar@{-}[dddr] &&& &&\ar@{-}[dddllll]\\
\\
\\
&&&\bullet\ar@{-}[ddd]&&&\\
\\
\\
&&&
}$$
then it follows by induction that
\begin{eqnarray*}
(W|\scB_\bullet|)'(T) &\subseteq &(W|\scB_\bullet|)'(T_1)\times (W|\scB_\bullet|)'(T_2)\times\dots\times (W|\scB_\bullet|)'(T_r)\\
&= &(W|\scB_\bullet|)(T_1)\times (W|\scB_\bullet|)(T_2)\times\dots\times (W|\scB_\bullet|)(T_r)
\end{eqnarray*}
Thus it remains to show that if 
$$\left(\alpha_s\right)_{s\in S}=(\overline{\alpha}_1,\overline{\alpha}_2,\dots,\overline{\alpha}_r)\in (W|\scB_\bullet|)'(T)$$
then $(\beta_1,\beta_2,\dots,\beta_r)$ is in the image of $W|\scB_\bullet|(r)\longrightarrow W|\scB_\bullet|(1)^r$, where $\beta_i$
is the maximal left factor of $\overline{\alpha}_i$ for $i=1,2,\dots r$.

By \ref{LemmaEins} and the induction hypothesis, we can reduce to the case $r=2$. If both $|T_1|$ and $|T_2|$ have cardinality
1 then $|T|$ has cardinality 2 and we are done by the induction hypothesis. We can assume wolog that $|T_1|$ has cardinality
$\ge 2$. By \ref{LemmaZwei} we can choose $\{s_1,s_2\}\subseteq|T_1|$ so that the maximal left factor of 
$(\alpha_{s_1},\alpha_{s_2})$ is $\beta_1$. If $|T_2|=\{s_3\}$ has cardinality 1, then $\beta_2=\alpha_{s_3}$ and we
have $(\alpha_{s_1},\alpha_{s_2},\alpha_{s_3})\in (W|\scB_\bullet|)(T\cap\{s_1,s_2,s_3\})$ which implies that $(\beta_1,\beta_2)$
is in the image of $W|\scB_\bullet|(2)\longrightarrow W|\scB_\bullet|(1)^2$. If $|T_2|$ also has cardinality $\ge 2$, pick
$\{s_3,s_4\}\subseteq|T_2|$ so that the maximal left factor of $(\alpha_{s_3},\alpha_{s_4})$ is $\beta_2$. Then we have
factorizations in $W|\scB_\bullet|(1)$ of the following form:
$$\alpha_{s_1}=\beta_1\cdot\gamma_1,\ \alpha_{s_2}=\beta_1\cdot\gamma_2,\ \alpha_{s_3}=\beta_2\cdot\gamma_3,\ 
\alpha_{s_4}=\beta_2\cdot\gamma_4$$
We then have 
$$(\alpha_{s_1},\alpha_{s_2},\alpha_{s_3})=(\beta_1\cdot\gamma_1,\beta_1\cdot\gamma_2,\beta_2\cdot\gamma_3)
\in (W|\scB_\bullet|)(T\cap\{s_1,s_2,s_3\}),$$
which implies that $(\beta_1,\beta_2\cdot\gamma_3)$ is in the image of $W|\scB_\bullet|(2)\longrightarrow W|\scB_\bullet|(1)^2$.
Similarly we have
$$(\alpha_{s_1},\alpha_{s_3},\alpha_{s_4})=(\beta_1\cdot\gamma_1,\beta_2\cdot\gamma_3,\beta_2\cdot\gamma_4)
\in (W|\scB_\bullet|)(T\cap\{s_1,s_3,s_4\}),$$
which implies that $(\beta_1\cdot\gamma_1,\beta_2)$ is in the image of $W|\scB_\bullet|(2)\longrightarrow W|\scB_\bullet|(1)^2$.
By \ref{LemmaDrei} this implies that $(\beta_1,\beta_2)$ is in the image of $W|\scB_\bullet|(2)\longrightarrow W|\scB_\bullet|(1)^2$.
This concludes the induction and proof.
\end{proof}

Hence we can reduce questions about the spaces $(W|\scB_\bullet|)(T)$ to the case of binodal trees with three inputs. For
future reference we will need to work out the intersections $(W|\scB_\bullet|)(T_1)\cap (W|\scB_\bullet|)(T_2)$ in this case.
First we note that the following is the complete list of all binodal trees with three inputs.

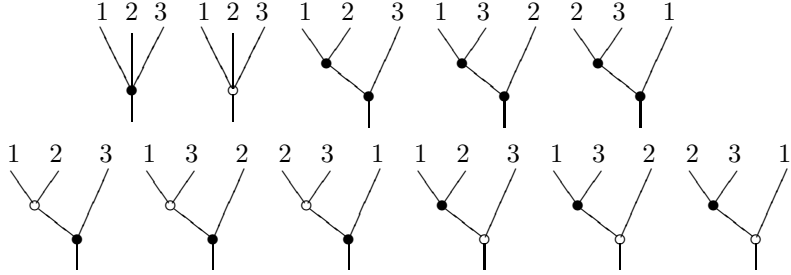
\centerline{
$\xymatrix@=5pt@M=-1pt@W=-1pt{
1 &&2 &&3\\
\ar@{-}[ddddrr] &&\ar@{-}[dddd] &&\ar@{-}[ddddll]\\
\\
\\
\\
&&\bullet\ar@{-}[dd]\\
\\
&&
}$
\quad
$\xymatrix@=5pt@M=-1pt@W=-1pt{
1 &&2 &&3\\
\ar@{-}[ddddrr] &&\ar@{-}[dddd] &&\ar@{-}[ddddll]\\
\\
\\
\\
&&\circ\ar@{-}[dd]\\
\\
&&
}$
\quad
$\xymatrix@=5pt@M=-1pt@W=-1pt{
1 &&2 &&&3\\
\ar@{-}[ddr] &&\ar@{-}[ddl]&&&\ar@{-}[ddddll]\\
\\
&\bullet\ar@{-}[ddrr]\\
\\
&&&\bullet\ar@{-}[dd]\\
\\
&&&
}$
\quad
$\xymatrix@=5pt@M=-1pt@W=-1pt{
1 &&3 &&&2\\
\ar@{-}[ddr] &&\ar@{-}[ddl]&&&\ar@{-}[ddddll]\\
\\
&\bullet\ar@{-}[ddrr]\\
\\
&&&\bullet\ar@{-}[dd]\\
\\
&&&
}$
\quad
$\xymatrix@=5pt@M=-1pt@W=-1pt{
2 &&3 &&&1\\
\ar@{-}[ddr] &&\ar@{-}[ddl]&&&\ar@{-}[ddddll]\\
\\
&\bullet\ar@{-}[ddrr]\\
\\
&&&\bullet\ar@{-}[dd]\\
\\
&&&
}$
}
\centerline{}
\centerline{
\quad
$\xymatrix@=5pt@M=-1pt@W=-1pt{
1 &&2 &&&3\\
\ar@{-}[ddr] &&\ar@{-}[ddl]&&&\ar@{-}[ddddll]\\
\\
&\circ\ar@{-}[ddrr]\\
\\
&&&\bullet\ar@{-}[dd]\\
\\
&&&
}$
\quad
$\xymatrix@=5pt@M=-1pt@W=-1pt{
1 &&3 &&&2\\
\ar@{-}[ddr] &&\ar@{-}[ddl]&&&\ar@{-}[ddddll]\\
\\
&\circ\ar@{-}[ddrr]\\
\\
&&&\bullet\ar@{-}[dd]\\
\\
&&&
}$
\quad
$\xymatrix@=5pt@M=-1pt@W=-1pt{
2 &&3 &&&1\\
\ar@{-}[ddr] &&\ar@{-}[ddl]&&&\ar@{-}[ddddll]\\
\\
&\circ\ar@{-}[ddrr]\\
\\
&&&\bullet\ar@{-}[dd]\\
\\
&&&
}$
\quad
$\xymatrix@=5pt@M=-1pt@W=-1pt{
1 &&2 &&&3\\
\ar@{-}[ddr] &&\ar@{-}[ddl]&&&\ar@{-}[ddddll]\\
\\
&\bullet\ar@{-}[ddrr]\\
\\
&&&\circ\ar@{-}[dd]\\
\\
&&&
}$
\quad
$\xymatrix@=5pt@M=-1pt@W=-1pt{
1 &&3 &&&2\\
\ar@{-}[ddr] &&\ar@{-}[ddl]&&&\ar@{-}[ddddll]\\
\\
&\bullet\ar@{-}[ddrr]\\
\\
&&&\circ\ar@{-}[dd]\\
\\
&&&
}$
\quad
$\xymatrix@=5pt@M=-1pt@W=-1pt{
2 &&3 &&&1\\
\ar@{-}[ddr] &&\ar@{-}[ddl]&&&\ar@{-}[ddddll]\\
\\
&\bullet\ar@{-}[ddrr]\\
\\
&&&\circ\ar@{-}[dd]\\
\\
&&&
}$
}

\begin{prop}\label{8_5}Let $\scB_\bullet$ be a simplicial operad with $\scB_\bullet(1)=\{id\}$. Let $T_1$ and $T_2$ be binodal
trees with three inputs. Then we have $(W|\scB_\bullet|)(T_1)\cap(W|\scB_\bullet|)(T_2)=(W|\scB_\bullet|)(T_3)$ or $\emptyset$ as shown
in the table below and in the table in the Appendix.

\begin{center}
\begin{tabular}{|c|c|c|c|}
\hline
&$T_1$ &$T_2$ &$(W|\scB_\bullet|)(T_1)\cap(W|\scB_\bullet|)(T_2)$\\
\hline
\raisebox{-10pt}{1}
&$\xymatrix@=5pt@M=-1pt@W=-1pt{
i &&j &&&k\\
\ar@{-}[ddr] &&\ar@{-}[ddl]&&&\ar@{-}[ddddll]\\
\\
&\bullet\ar@{-}[ddrr]\\
\\
&&&\bullet\ar@{-}[dd]\\
\\
&&&
}$
&$\xymatrix@=5pt@M=-1pt@W=-1pt{
i &&k &&&j\\
\ar@{-}[ddr] &&\ar@{-}[ddl]&&&\ar@{-}[ddddll]\\
\\
&\bullet\ar@{-}[ddrr]\\
\\
&&&\bullet\ar@{-}[dd]\\
\\
&&&
}$
&\raisebox{-10pt}{$\emptyset$}
\\
\hline
\raisebox{-10pt}{2}
&$\xymatrix@=5pt@M=-1pt@W=-1pt{
i &&j &&&k\\
\ar@{-}[ddr] &&\ar@{-}[ddl]&&&\ar@{-}[ddddll]\\
\\
&\bullet\ar@{-}[ddrr]\\
\\
&&&\bullet\ar@{-}[dd]\\
\\
&&&
}$
&$\xymatrix@=5pt@M=-1pt@W=-1pt{
i &&k &&&j\\
\ar@{-}[ddr] &&\ar@{-}[ddl]&&&\ar@{-}[ddddll]\\
\\
&\circ\ar@{-}[ddrr]\\
\\
&&&\bullet\ar@{-}[dd]\\
\\
&&&
}$
&\raisebox{-10pt}{$\emptyset$}
\\
\hline
\raisebox{-10pt}{3}
&$\xymatrix@=5pt@M=-1pt@W=-1pt{
i &&j &&&k\\
\ar@{-}[ddr] &&\ar@{-}[ddl]&&&\ar@{-}[ddddll]\\
\\
&\circ\ar@{-}[ddrr]\\
\\
&&&\bullet\ar@{-}[dd]\\
\\
&&&
}$
&$\xymatrix@=5pt@M=-1pt@W=-1pt{
i &&k &&&j\\
\ar@{-}[ddr] &&\ar@{-}[ddl]&&&\ar@{-}[ddddll]\\
\\
&\circ\ar@{-}[ddrr]\\
\\
&&&\bullet\ar@{-}[dd]\\
\\
&&&
}$
&\raisebox{-10pt}{$\emptyset$}
\\
\hline
\end{tabular}
\end{center}

\end{prop}
\begin{proof}
Let $T_{ij}$ be the tree in row $i$ and column $j$ of our lists. We start with the list in the Appendix.

Row 1 is trivially true because $(W|\scB_\bullet|)(T_{11})=W|\scB_\bullet|(1)^3$.

Rows 2 to 4: Since $T_{21}$ is symmetric in $1,2,3$ we may assume that $(i,j,k)=(1,2,3)$.
Then\\
$(W|\scB_\bullet|)(T_{21})=W|\scB_\bullet|(3)$\\
$(W|\scB_\bullet)(T_{22})=W|\scB_\bullet|(2)\times_{W|\scB_\bullet|(1)^2}(W|\scB_\bullet|(2)\times W|\scB_\bullet|(1))
\subset W|\scB_\bullet|(3)$\\
$(W|\scB_\bullet|)(T_{32})=W|\scB_\bullet|(2)\times W|\scB_\bullet|(1)\supset W|\scB_\bullet|(3)$\\
$(W|\scB_\bullet|)(T_{42})=W|\scB_\bullet|(2)\times_{W|\scB_\bullet|(1)^2}(W|\scB_\bullet|(1)^2\times W|\scB_\bullet|(1))$,\\
where $W|\scB_\bullet|(1)^2$ acts on $W|\scB_\bullet|(1)^2\times W|\scB_\bullet|(1)$ by
$$(x,y)\cdot(a_1,a_2,b)=(x\circ a_1,x\circ a_2,y\circ b)$$
Clearly $(W|\scB_\bullet|)(T_{43})\subset (W|\scB_\bullet|)(T_{41})\cap (W|\scB_\bullet|)(T_{42})$. Conversely,
if you stump $k$ (i.e. you graft a stump on input $k$) of an element in the intersection,
you obtain an element in $W|\scB_\bullet|(2)$, i.e. the intersection lies in $(W|\scB_\bullet|)(T_{43})$.

Row 5: Since $T_{51}$ is symmetric in $i$ and $j$ we have to consider the cases 
$(p,q,r)=(i,j,k)$ and $(p,q,r)=(i,k,j)$. 
In the first case $(W|\scB_\bullet|)(T_{51})\subset (W|\scB_\bullet|)(T_{52})$. In the second case
we use the equivariance relation to give the inputs of both trees the order  $k,i,j$. We find
$$
\begin{array}{rcl}
 (W|\scB_\bullet|)(T_{51})& = & W|\scB_\bullet|(2)\times_{W|\scB_\bullet|(1)^2}(W|\scB_\bullet|(1)\times W|\scB_\bullet|(2))\\
& \subset & W|\scB_\bullet|(2)\times W|\scB_\bullet|(2)= (W|\scB_\bullet|)(T_{52})
\end{array}
$$
The first component of the inclusion is given by stumping $j$, the second by stumping $k$ and $i$.

Row 6: $(W|\scB_\bullet|)(T_{61})\subset (W|\scB_\bullet|)(T_{62})$.

Row 7: Clearly $(W|\scB_\bullet|)(T_{73})\subset (W|\scB_\bullet|)(T_{71})\cap (W|\scB_\bullet|)(T_{72})$. Conversely,
by stumping $k$ we see that the intersection has to lie in $(W|\scB_\bullet|)(T_{73})$.

Row 8: We use the equivariance relation to give the inputs of both trees the order $k,i,j$. Then
$$
(W|\scB_\bullet|)(T_{81})\subset  W|\scB_\bullet|(2)\times W|\scB_\bullet|(1)=(W|\scB_\bullet|)(T_{82})
$$
The first component of the inclusion is given by stumping $j$, the second by stumping $k$ and $i$.

Row 9: Again we give the inputs of both trees the order $k,i,j$ and obtain the intersection
$$
W|\scB_\bullet|(1)\times W|\scB_\bullet|(2)\ \cap \ W|\scB_\bullet|(2)\times W|\scB_\bullet|(1)$$
which contains $W|\scB_\bullet|(3)$, but is not of the form $(W|\scB_\bullet|)(T)$.

Now consider the second list: In all three cases we arrange the trees such that the order of
inputs is $k,i,j$.

Row 1: If $x$ is in $(W|\scB_\bullet|)(T_{11})$ then the $i$-th and $j$-th input paths meet above the $i$-th and
$k$-th ones, while if $x$ is in $(W|\scB_\bullet|)(T_{12})$ it is the other way around. So
$(W|\scB_\bullet|)(T_{11})\cap (W|\scB_\bullet|)(T_{12})=\emptyset$.

Rows 2 and 3:
$$
(W|\scB_\bullet|)(T_{21})\cap (W|\scB_\bullet|)(T_{22})\subset (W|\scB_\bullet|)(T_{31})\cap (W|\scB_\bullet|)(T_{32})
$$
Stumping $j$ or $k$ of an element $y\in (W|\scB_\bullet|)(T_{31})\cap (W|\scB_\bullet|)(T_{32})$ makes it an
element in $W|\scB_\bullet|(2)$. Hence
$$
(W|\scB_\bullet|)(T_{31})\cap (W|\scB_\bullet|)(T_{32})\subset 
(W|\scB_\bullet|)(T_{11})\cap (W|\scB_\bullet|)(T_{12})=\emptyset
$$
\end{proof}

\begin{prop}\label{8_6} Let $\scB_\bullet$ be  a simplicial operad. Then 
\begin{itemize}
\item The imbedding $(W|\scB_\bullet|)(T)\subseteq W|\scB_\bullet|(1)^{|T|}$ of Proposition \ref{8_2} is a cofibration.
\item If $S$ and $T$ are binodal trees
such that $(W|\scB_\bullet|)(S)\subset (W|\scB_\bullet|)(T)$, then the inclusion map is a cofibration.
\end{itemize}
\end{prop}

\begin{proof} It is obvious that the imbedding $(W|\scB_\bullet|)(T)\subseteq W|\scB_\bullet|(1)^{|T|}$ is the realization
of an injective simplicial map. It follows that whenever we have $(W|\scB_\bullet|)(S)\subset (W|\scB_\bullet|)(T)$, the
inclusion is also the realization of an injective simplicial map. Since realizations of injective simplicial maps
are cofibrations, the result follows.
\end{proof}

\begin{prop}\label{8_7} Let $\scB_\bullet$ be  a simplicial operad such that $\scB_\bullet(1)=\{id\}$.
and let $\varepsilon: W|\scB_\bullet|\longrightarrow|\scB_\bullet|$ be the augmentation map. Then
\begin{itemize}
\item  For any binodal tree $T$, there is an induced map $\varepsilon: (W|\scB_\bullet|)(T)\longrightarrow|\scB_\bullet(T)|$
\item  For any subcomplex $K_\bullet\subset\scB_\bullet(T)$, the restriction 
$\varepsilon: \varepsilon^{-1}(|K_\bullet|)\longrightarrow |K_\bullet|$ is an equivalence.
\end{itemize}
\end{prop}

\begin{proof}
Since $\scB_\bullet(1)=\{id\}$, $\scB_\bullet(T)$ is simply a product of the form $\prod_i\scB_\bullet(k_i)$.
Thus $\varepsilon$ induces a map of products 
$\prod_i W|\scB_\bullet|(k_i)\longrightarrow\prod_i|\scB_\bullet(k_i)|=|\scB_\bullet(T)|$
which factors through a map
$$\varepsilon: (W|\scB_\bullet|)(T)\longrightarrow|\scB_\bullet(T)|.$$
Let $W'|\scB_\bullet|$, denote the suboperad of $W|\scB_\bullet|$ consisting of trees without any stumps.
Then $\scB_\bullet(1)=\{id\}$ and the identity relation on trees implies that
$W'|\scB_\bullet|(1)=\{id\}$. It follows that $(W'|\scB_\bullet|)(T)\subset(W|\scB_\bullet|)(T)$,
is also a product $\prod_i W'|\scB_\bullet|(k_i)$.
There is a strong deformation retraction $(W|\scB_\bullet|)(T)\longrightarrow(W'|\scB_\bullet|)(T)$
given by shrinking stumps. The product
of the canonical sections
$\prod_i\eta(k_i):\prod_i|\scB_\bullet|(k_i)\longrightarrow\prod_i W|\scB_\bullet|(k_i)$
followed by the quotient map into $(W\scB_\bullet)(T)$ is contained in 
$(W'\scB_\bullet)(T)=\prod_i W'|\scB_\bullet|(k_i)$.  Shrinking the internal edges of
$(W'\scB_\bullet)(T)$ provides a strong deformation retraction onto this section. Composing
these two deformation retractions, we obtain that
$$\varepsilon: (W|\scB_\bullet|)(T)\longrightarrow|\scB_\bullet(T)|$$
is an equivalence, which restricts to an equivalence
$$\varepsilon: \varepsilon^{-1}(|K_\bullet|)\longrightarrow |K_\bullet|$$
for any subcomplex $K_\bullet\subset\scB_\bullet(T)$.
\end{proof}

\begin{coro}\label{8_10} For any binodal tree $T$, $|\scN\scM_k(T)|$ has a cellular decomposition 
(recall Definition \ref{5_6a})
over the poset
$\scM_k(T)$ which lifts to a cellular decomposition of $(W|\mathcal{NM}_k|)(T)$ over $\scM_k(T)$.
\end{coro}

\begin{proof} Since $N\scM(1)=\{id\}$, $|(\scN\scM_k)(T)|$ is simply
a product of the form $\prod_i N\scM(k_i)$, which has a cellular decomposition over $\prod_i \scM(k_i)=\scM_k(T)$.
Since this cellular decomposition is the realization of a simplicial cellular decomposition, by the second part
of Proposition \ref{8_7}, the inverse images of the cells under the augmentation
$\varepsilon: (W|\scN\scM_k|)(T)\longrightarrow|N\scM(T)|$ provide the required cellular decomposition.
\end{proof}

In order to analyze the behavior of the colimit decomposition of the tensor product
$W\mathcal{NM}_k\otimes W\mathcal{NM}_l$ with respect to its operad structure, we
will need the following constructions of operads associated to binodal trees.

\begin{defi}\label{scbt}The operad $\scB\scT$ is the operad in $\Sets$ with $\scB\scT(m)$
being the set of all binodal trees $T$ with $|T|=\{1,2,\dots m\}$ . Composition in $\scB\scT$
is defined in the usual way by grafting trees, followed by shrinking any edges connecting two
nodes labeled by $\circ$.  The symmetric group action is the one induced by its action on $|T|$.
\end{defi}

\begin{rema}The operad $\scB\scT$ is a slight variant of the operad of grown trees $TX$ of \cite{V},
where $X(0)=X(1)=\emptyset$ and $X(i)=\{\bullet,\circ\}$ for all $i\ge2$, with the symmetric group
acting trivially on $X(i)$. Indeed $\scB\scT$ is the quotient of $TX$ obtained by shrinking all edges
joining nodes labeled by $\circ$.
\end{rema}

\begin{defi}\label{scabt}Let $\scA$ be an operad in one of our categories $\scS$.  Then
we define $\scA(\scB\scT)$ to be the $\mathcal{S}$-operad with
$$\scA(\scB\scT)(m)=\amalg_{|T|=\{1,2,\dots,m\}}\scA(T),$$
where we think of elements of $\scA(T)$ as binodal trees $T$ decorated with appropriate
elements of $\scA$ as in Remark \ref{8_3a}.  Composition in $\scA(\scB\scT)$
is defined by grafting trees, followed by shrinking any edges connecting two nodes labeled by $\circ$.
The map $\varepsilon:\scA(\scB\scT)\longrightarrow\scB\scT$ given by forgetting the decorations from
$\scA$ is evidently an operad map. [Note that $\Sets\subset\scS$.]  If $\scA$ is an axial operad, then
the imbeddings $\scA(T)\subset RU\scA(|T|)$ of Proposition \ref{8_2} define an  imbedding of $\scA(\scB\scT)$
as a suboperad of $RU\scA$.
\end{defi}

\begin{rema}\label{scPscA}If $\scA$ is a topological operad which has a cellular decomposition over a poset
operad $\scP$, then it is easy to see that $\scA(\scB\scT)$ has a cellular decomposition over the poset
operad $\scP(\scB\scT)$.
\end{rema}

\section{The interchange diagram for a tensor product}\label{sec9}

We begin by showing how to associate a pair of binodal trees $(S_\alpha,T_\alpha)$
to any element $\alpha$ in $\mathfrak{M}^{ab}_2(m)$. First of all we can represent
$\alpha$ by a nonplanar rooted tree with inputs labeled by $\{1,2,\dots,m\}$
with nodes marked by $\boxtimes_1$ and $\boxtimes_2$.
Given such a tree, the corresponding object $\alpha$ is obtained by interpreting the
edges as compositions in the operad $\mathfrak{M}^{ab}_2$. There is more than
one tree representation of $\alpha$, but there is a unique reduced form obtained
by removing all univalent edges (which correspond to composing with the identity
of $\mathfrak{M}^{ab}_2$) and by shrinking all edges connecting nodes which are both
marked by $\boxtimes_1$ and or both marked by $\boxtimes_2$ (which corresponds to
the strict associativity of $\boxtimes_1$ and $\boxtimes_2$). Now given the reduced
form representation of $\alpha$ we obtain the binodal tree $S_\alpha$ by replacing
each node decorated with $\boxtimes_1$ by a black node and each node decorated with
$\boxtimes_2$ by a white node. The binodal tree $T_\alpha$ is obtained similarly
by reversing the roles of $\boxtimes_1$ and $\boxtimes_2$. Thus we may choose planar
representatives of the binodal trees $S_\alpha$ and $T_\alpha$ having exactly the same
shape and input labels, but each white node
in $S_\alpha$ corresponds to a black node in $T_\alpha$ and vice versa. 
We call such a pair of representatives compatible. Clearly
$\alpha$ can be uniquely reconstructed from a compatible pair $(S_\alpha,T_\alpha)$.

As a first step in our analysis of $(W|\mathcal{NM}_k|\otimes W|\mathcal{NM}_l|)(m)$ we prove Proposition
\ref{5_3} by proving a stronger result.

\begin{prop}\label{9_1} Let $\alpha$ be an element of $\mathfrak{M}^{ab}_2(m)$.
Then the composite map
$$
\begin{array}{rl}
G_m(\alpha)\longrightarrow(W|\mathcal{NM}_k|\otimes W|\mathcal{NM}_l|)(m) &
\stackrel{\xi}{\longrightarrow}
RU(W|\mathcal{NM}_k|\otimes W|\mathcal{NM}_l|)(m)\\
 &=(W|\mathcal{NM}_k|(1))^m\times (W|\mathcal{NM}_l|(1))^m
\end{array}
$$
is an imbedding and its image is $(W|\mathcal{NM}_k|)(S_\alpha)\times (W|\mathcal{NM}_l|)(T_\alpha)$.
\end{prop}

\begin{proof}
We define a map
$$\zeta_{\alpha}:(W|\mathcal{NM}_k|(1))^m\times (W|\mathcal{NM}_l|(1))^m\longrightarrow
G_m(\alpha)$$
using the descriptions of $(W|\mathcal{NM}_k|)(S_\alpha)$ and $(W|\mathcal{NM}_l|)(T_\alpha)$
given in Remark \ref{8_3a}. Choose compatible planar representatives of the binodal trees
$S_\alpha$ and $T_\alpha$, and then pick corresponding labeled planar tree representatives 
$S\in(W|\mathcal{NM}_k|)(S_\alpha)$
and $T\in(W|\mathcal{NM}_l|)(T_\alpha)$. Now for any unlabeled white node in $S$ there is a unique matching black node in $T$,
labeled by an element of $W|\mathcal{NM}_l|$. Replace the white node in $S$ by the
corresponding labeled black node in $T$. Also combine the edge labels from $S$ and $T$ by putting a univalent node marked with
the edge label (in $W|\mathcal{NM}_l|(1)$) from $T$ above another univalent node marked with the edge label 
(in $W|\mathcal{NM}_k|(1)$) from $S$. The resulting tree $\zeta_{\alpha}(S,T)$ represents an element of $G_m(\alpha)$.
The map $\zeta_{\alpha}$ is well defined, since the relations (1)-(4) in Remark \ref{8_3a}
on the domain of $\zeta_{\alpha}$ correspond to the equivariance and unary interchange relations
on $G_m(\alpha)$. By Proposition \ref{8_2} and Corollary \ref{8_3} it follows that the
composite
$$\begin{array}{rl}
(W|\mathcal{NM}_k|(1))^m\times (W|\mathcal{NM}_l|(1))^m\stackrel{\zeta_{\alpha}}{\longrightarrow}\\
G_m(\alpha)\longrightarrow(W|\mathcal{NM}_k|\otimes W|\mathcal{NM}_l|)(m) &
\stackrel{\xi}{\longrightarrow}
RU(W|\mathcal{NM}_k|\otimes W|\mathcal{NM}_l|)(m)\\
 &=(W|\mathcal{NM}_k|(1))^m\times (W|\mathcal{NM}_l|(1))^m
\end{array}$$
is the natural imbedding. It follows that $\zeta_{\alpha}$ is injective.

It remains to show that $\zeta_{\alpha}$ is surjective.
Let $B$ be a tree representative of an element in $G_m(\alpha)$. We may reduce this tree using
quotient relations in $G_m(\alpha)$ as follows. First of all using the nullary interchange
relations (cf. Remark \ref{5_001}), we remove any stumps in $B$. We then eliminate any
unary nodes which are below a multivalent node.  For given a subtree of $B$ having the form
$$
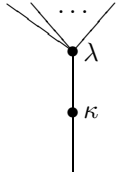
\xymatrix@=5pt@M=-1pt@W=-1pt{
\ar@{-}[dddrrrrrr] &&&\ar@{-}[dddrrr] &&&\dots& &&\ar@{-}[dddlll]\\
\\
\\
&& && &&\bullet\ar@{-}[dddd]&\save[]+<-2pt,0pt>*{\lambda}\restore\\
\\
\\
\\
&& && &&\bullet\ar@{-}[dddd]&\save[]+<-2pt,0pt>*{\kappa}\restore\\
\\
\\
\\
&& && && &&
}
$$
we may replace it by the equivalent tree
$$
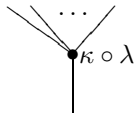
\xymatrix@=5pt@M=-1pt@W=-1pt{
\ar@{-}[dddrrrrrr] &&&\ar@{-}[dddrrr] &&&\dots& &&\ar@{-}[dddlll]\\
\\
\\
&& && &&\bullet\ar@{-}[dddd]&\save[]+<+4pt,0pt>*{\kappa\circ\lambda}\restore\\
\\
\\
\\
&&&&&&\\
}
$$
if $\kappa$ and $\lambda$ are both in $W|\mathcal{NM}_k|$ or both in $W|\mathcal{NM}_l|$, and
by
$$
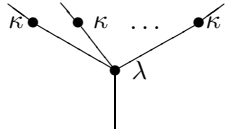
\xymatrix@=5pt@M=-1pt@W=-1pt{
\ar@{-}[drr]&&&&&\ar@{-}[dr]&&&&&&&&&&&\ar@{-}[dll]
\\
&&\bullet\ar@{-}[dddrrrrrrr]&\save[]+<-12pt,0pt>*{\kappa}\restore
&&&\bullet\ar@{-}[dddrrr]&\save[]+<+3pt,0pt>*{\kappa}\restore &&&\dots& 
&&&\bullet\ar@{-}[dddlllll]&\save[]+<0pt,0pt>*{\kappa}\restore\\
\\
\\
&&& && && &&\bullet\ar@{-}[dddd]&\save[]+<-2pt,0pt>*{\lambda}\restore\\
\\
\\
\\
& && &&&&&&\\
}
$$
if one of $\kappa$, $\lambda$ is in $W|\mathcal{NM}_k|$ and the other is in
$W|\mathcal{NM}_l|$. Thus $B$ is equivalent to a tree where all the unary nodes are at the
top of the tree, just below the input edges. We further reduce the tree by
interpreting any edge connecting multivalent nodes both labeled by elements of $W|\mathcal{NM}_k|$ or both labeled by elements of
$W|\mathcal{NM}_l|$ as compositions within that operad. After these reductions, we obtain a tree
$S\in(W|\mathcal{NM}_k|)(S_\alpha)$ by replacing all multivalent nodes labeled with elements of $W|\mathcal{NM}_l|$ by unlabeled white nodes, dropping all univalent nodes, and transferring the univalent node labels coming from $W|\mathcal{NM}_k|$ to the corresponding edges. Any
edge which remains unlabeled is considered to be labeled by the unit.  Similarly we obtain
a tree $T\in(W|\mathcal{NM}_l|)(T_\alpha)$ by reversing the roles of $W|\mathcal{NM}_k|$ and
$W|\mathcal{NM}_l|$. It is clear that $\zeta_{\alpha}(S,T)=B$. Thus $\zeta_{\alpha}$ is
bijective and hence $G_m(\alpha)$ imbeds as claimed.
\end{proof}

We illustrate the proof of Proposition \ref{9_1} with a simple example.

\begin{exam}
Let $\alpha=(1\boxtimes_2 2\boxtimes_2 4)\boxtimes_1 3$. Then after completing the reduction
process described in the proof, any element $B$ of $G_4(\alpha)$ has a representative of the form
$$
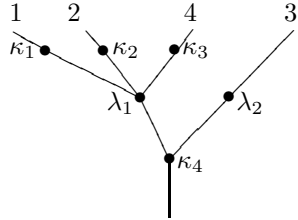
\xymatrix@=5pt@M=-1pt@W=-1pt{
1&&&&2 &&&&&&&&4 &&&&&&&&&&&3\\
\ar@{-}[drr]&&&&&\ar@{-}[dr]&&&&&&&\ar@{-}[ddddlll]&&&&&&&&&&&\ar@{-}[ddddddddllllllllllll]
\\
&&\bullet\ar@{-}[dddrrrrrrr]&\save[]+<-14pt,0pt>*{\kappa_1}\restore
&&&\bullet\ar@{-}[dddrrr]&\save[]+<+3pt,0pt>*{\kappa_2}\restore 
&&&&\save[]+<+2pt,0pt>*{\bullet}\restore
&\save[]+<+2pt,0pt>*{\kappa_3}\restore\\
\\
\\
&&& && && &&\bullet\ar@{-}[ddddrr]&\save[]+<-13pt,-2pt>*{\lambda_1}\restore&&&
&&&&\save[]+<-3pt,0pt>*{\bullet}\restore&\save[]+<+2pt,-2pt>*{\lambda_2}\restore\\
\\
\\
\\
& && &&&&&&&&\bullet\ar@{-}[dddd]&\save[]+<0pt,-2pt>*{\kappa_4}\restore\\
\\
\\
\\
& && &&&&&&&&
}
$$
Here $\kappa_i\in W|\mathcal{NM}_k|(1)$ for $i=1,2,3$, $\kappa_4\in W|\mathcal{NM}_k|(2)$,
$\lambda_1\in W|\mathcal{NM}_l|(3)$, and $\lambda_2\in W|\mathcal{NM}_l|(1)$.
Then $\zeta_{\alpha}(S,T)=B$, where\\
\centerline{
\raisebox{-40pt}{$(S,T) =$}\qquad
\raisebox{-35pt}{$\left(\begin{array}{l}\\ \\ \\ \\ \\ \\ \end{array}\right.$}
$\xymatrix@=5pt@M=-1pt@W=-1pt{
1&&&&2 &&&&&&&&4 &&&&&&&&&&&3\\
\ar@{-}[ddddrrrrrrrrr]&&&&&\ar@{-}[ddddrrrr]&&&&&&&\ar@{-}[ddddlll]
&&&&&&&&&&&\ar@{-}[ddddddddllllllllllll]
\\
&&&\save[]+<-14pt,-2pt>*{\kappa_1}\restore
&&&&\save[]+<+3pt,0pt>*{\kappa_2}\restore 
&&&&&\save[]+<+2pt,0pt>*{\kappa_3}\restore\\
\\
\\
&&& && && &&\circ\ar@{-}[ddddrr]&&&&
&&&&&\\
\\
\\
\\
& && &&&&&&&&\bullet\ar@{-}[dddd]&\save[]+<0pt,-2pt>*{\kappa_4}\restore\\
\\
\\
\\
& && &&&&&&&&
}$
\raisebox{-65pt}{\quad{\Huge,}\quad}
$
\xymatrix@=5pt@M=-1pt@W=-1pt{
1&&&&2 &&&&&&&&&4 &&&&&&&&&&&3\\
\ar@{-}[ddddrrrrrrrrr]&&&&&\ar@{-}[ddddrrrr]&&&&&&&&\ar@{-}[ddddllll]
&&&&&&&&&&&\ar@{-}[ddddddddlllllllllllll]
\\
&&&
&&&& 
&&&&&&\\
\\
\\
&&& && && &&\bullet\ar@{-}[ddddrr]&\save[]+<-13pt,-2pt>*{\lambda_1}\restore&&&
&&&&&\save[]+<+2pt,-4pt>*{\lambda_2}\restore\\
\\
\\
\\
& && &&&&&&&&\circ\ar@{-}[dddd]&\\
\\
\\
\\
& && &&&&&&&&
}$
\raisebox{-35pt}{$\left.\begin{array}{l}\\ \\ \\ \\ \\ \\ \end{array}\right)$}
}\\
\raisebox{-8pt}{Here the unlabeled edges are considered to be labeled by the unit of the operad.}
\end{exam}

Thus we can think of $G_m(\alpha)$ as a subspace, either of $(W|\mathcal{NM}_k|\otimes W|\mathcal{NM}_l|)(m)$
or of $W|\mathcal{NM}_k|(1)^m\times W|\mathcal{NM}_l|(1)^m$, 
and it makes sense to talk about intersections
of the form $\cap_i G_m(\alpha_i)$. However at this point we cannot determine if these intersections
are the same in $(W|\mathcal{NM}_k|\otimes W|\mathcal{NM}_l|)(m)$ as they are in $W|\mathcal{NM}_k|(1)^m\times W|\mathcal{NM}_l|(1)^m$
because we so far have not proved that $(W|\mathcal{NM}_k|\otimes W|\mathcal{NM}_l|)$ is axial.
To distinguish between the two types of intersections, we will write $G_m^{ax}(\alpha)$ (i.e. axial image) when we think
of $G_m(\alpha)$ as a subset of the latter instead of the former.

Moreover since the map $G_m(\alpha)\to(W|\mathcal{NM}_k|\otimes W|\mathcal{NM}_l|)(m)$ is the realization of a simplicial map,
it follows that this map is cofibration, thus establishing Proposition \ref{5_3}.
 
Our next goal is to show that there are no nonempty
intersections in either $(W|\mathcal{NM}_k|\otimes W|\mathcal{NM}_l|)(m)$
or $W|\mathcal{NM}_k|(1)^m\times W|\mathcal{NM}_l|(1)^m$
except those encoded by the simplices of $\mathfrak{K}_\bullet(m)$.

\begin{prop}\label{empty}
 If $\alpha_0,\ldots,\alpha_r$ are not vertices of a simplex of 
$\mathfrak{K}_\bullet(m)$, then $\cap_{i=0}^r G_m(\alpha_i)=\cap_{i=0}^r G_m^{ax}(\alpha_i)=\emptyset$.
\end{prop}
\begin{proof}
 By assumption there is a triple $\{a,b,c\}\subset \{1,\ldots,m\}$ such that
the set of vertices of
$\{\alpha_0\cap \{a,b,c\},\ldots,\alpha_n \cap \{a,b,c\}\}$ is not a simplex
in $\mathfrak{K}_\bullet(\{a,b,c\})$. Wolog $\{a,b,c\}=\{1,2,3\}$. Hence this collection of vertices
must contain two different objects $\beta_1$ and $\beta_2$ in the second or third column of
the table of \ref{5_4}. By \ref{8_5}, if $\beta_1$ and $\beta_2$ are in the second
column, then $(W|\scB_\bullet|)(S_{\beta_1})\cap(W|\scB_\bullet|)(S_{\beta_2})=\emptyset$, and
if $\beta_1$ and $\beta_2$ are in the third column, then 
$(W|\scB_\bullet|)(T_{\beta_1})\cap(W|\scB_\bullet|)(T_{\beta_2})=\emptyset$.
Since the restriction into $\mathfrak{K}_\bullet(3)$ corresponds to looking at the 
first three components of the axial image, we conclude that 
$\cap_{i=0}^n G_m^{ax}(\alpha_i)=\emptyset$. Since the image of a nonempty set
must be nonempty, it follows that $\cap_{i=0}^n G_m(\alpha_i)=\emptyset$
\end{proof}

Our next aim is to show that the intersections in $W|\mathcal{NM}_k|(1)^m\times W|\mathcal{NM}_l|(1)^m$
along a simplex are not empty.

\begin{lem}\label{hope}
If $\sigma=\{\alpha_0,\ldots ,\alpha_r\}$ is an $r$-simplex, $r\geq 1$, in $\mathfrak{K}_\bullet(m)$
and all vertices $\alpha_i$ have outermost operation $\boxtimes_1$, then for all $i$
$$ \alpha_i=\alpha_{i1}\boxtimes_1 \alpha_{i2} 
$$
with $|\alpha_{i1}|=|\alpha_{j1}|$ and $|\alpha_{i2}|=|\alpha_{j2}|$ for all $i, j$, where $|\alpha|$ is the underlying set of
generators of $\alpha$. \\
The same holds if we interchange $\boxtimes_1$ and $\boxtimes_2$.
\end{lem}
\begin{proof}
\textit{Special case:} There is a vertex $\alpha=1\boxtimes_1 \alpha'$ in $\sigma$ where $\alpha'$ has 
outermost operation $\boxtimes_2$. We show that any other vertex $\beta$ of $\sigma$ has the form $\beta=1\boxtimes_1\beta'$.

 We  proceed by induction.
For $m=3$ the statement is true (see table in \ref{5_4}).\\
Assume that $m>3$. We can find a generator (wolog $2$) such that
replacing $2$ by $0$ leaves an outermost $\boxtimes_2$ in $\alpha'\cap \{1,3,\ldots,m\}$.
Then by induction, $\beta\cap \{1,3,\ldots,m\}=1\boxtimes_1 \beta''$. If $\beta\ne 1\boxtimes_1 \beta'$
it must have the form $\beta=(1\boxtimes_2 2)\boxtimes_1 \beta''$. There is a generator
(wolog 3) such that $\alpha'\cap \{2,3\}=2\boxtimes_2 3$. Then
$$
\alpha\cap \{1,2,3\}= 1\boxtimes_1(2\boxtimes_23) \quad\textrm{and}\quad 
\beta\cap \{1,2,3\}=(1\boxtimes_2 2)\boxtimes_13=3\boxtimes_1 (1\boxtimes_2 2).
$$
But then $\{\alpha,\beta\}$ is not a $1$-simplex in $\mathfrak{K}_\bullet(m)$.

\textit{General case:} Let
$$
\alpha_i=\alpha_{i1}\boxtimes_1\alpha_{i2}\boxtimes_1\ldots\boxtimes_1\alpha_{iq_i}
$$
with $q_i$ maximal for such a decomposition. Among the $\alpha_{is}, 0\leq i\leq n,\ 1\leq s\leq q_i$ there is a not
necessarily unique block with a maximal number of generators. Wolog we may assume that this block is 
 $\alpha_{01}$ and that  $|\alpha_{01}|=\{1,\ldots ,p \}$. Then $\alpha_0= \alpha_{01}\boxtimes_1\alpha_0'$ and
$|\alpha_0'|=\{p+1,\ldots,m\}$.
 For each $j\in \{p+1,\ldots,m\}$
define $\alpha_i(j)=\alpha_i\cap \{1,2,\ldots,p,j\}$. Then $\alpha_i(j)$ has outermost operation
$\boxtimes_1$ by maximality of $\alpha_{01}$. Since $\alpha_0(j)=j\boxtimes_1 \alpha_{01}$ and $\alpha_{01}$ has outermost 
operation $\boxtimes_2$ we obtain that each $\alpha_i(j)$ has the form $\alpha_i(j)=j\boxtimes_1 \alpha_i(j)'$ 
by the special case above. 
Since this holds for each $j\in \{p+1,\ldots,m\}$ we obtain a decomposition
$$
\alpha_i=(\alpha_i\cap \{1,\ldots,p\})\ \boxtimes_1\ (\alpha_i \cap \{p+1,\ldots,m\})
$$
which proves the lemma.
\end{proof}

\begin{rema}
Note that the proof of Lemma \ref{hope} is constructive.  That is, we give an explicit
algorithm for constructing the decomposition in question.
\end{rema}

By iterating Lemma \ref{hope} we obtain

\begin{coro}\label{label_3_5}
If $\sigma=\{\alpha_0,\ldots ,\alpha_r\}$ is an $r$-simplex, $r\geq 1$, in $\mathfrak{K}_\bullet(m)$
and all vertices $\alpha_i$ have outermost operation $\boxtimes_1$, then there is a maximal
decomposition such that for all $i$
$$\alpha_i=\alpha_{i1}\boxtimes_1\alpha_{i2}\boxtimes_1\dots\boxtimes_1\alpha_{ip}$$
and for all $1\le j\le p$
$$|\alpha_{0j}|=|\alpha_{1j}|=\dots=|\alpha_{rj}|.$$
Maximal means that for all $1\le j\le p$, either $|\alpha_{0j}|=|\alpha_{1j}|=\dots=|\alpha_{rj}|$
consists of a single generator, or $|\alpha_{ij}|$ has outermost operation $\boxtimes_2$
for at least one $i$.\newline
The same holds if we interchange $\boxtimes_1$ and $\boxtimes_2$.
\end{coro}

We shall refer to the decomposition specified in Corollary \ref{label_3_5} as the {\it maximal
common $\boxtimes_1$ decomposition}\/  of the vertices of a simplex of the first kind or the 
{\it maximal common $\boxtimes_2$ decomposition} of the vertices of a simplex of the second
kind.

\begin{prop}\label{label_4} Suppose that 
$$\{\alpha_0,\alpha_1,\dots,\alpha_r\}$$
is an $r$-simplex in $\mathfrak{K}_\bullet(m)$ (which we denote $\overline{\alpha}$). Then
$\cap_{i=0}^r G_m^{ax}(\alpha_i)$ is nonempty and we may find a pair of binodal trees
$(S_{\overline{\alpha}},T_{\overline{\alpha}})$ such that
$$\cap_{i=0}^r G_m^{ax}(\alpha_i)
=W|\mathcal{NM}_k|(S_{\overline{\alpha}})\times W|\mathcal{NM}_l|(T_{\overline{\alpha}})$$
\end{prop}

\begin{proof}
We will first describe a recursive algorithm for constructing the trees $S_{\overline{\alpha}}$ and
$T_{\overline{\alpha}}$. We defer proving that 
$$\cap_{i=0}^r G_m^{ax}(\alpha_i)
=(W|\mathcal{NM}_k|)(S_{\overline{\alpha}})\times(W|\mathcal{NM}_l|)(T_{\overline{\alpha}})$$
until later.

The first case of a simplex of positive dimension occurs when $m=2$, where we have the 1-simplex
$\overline{\alpha}=\{1\boxtimes_1 2, 1\boxtimes_2 2\}$. In this case we take\newline
\centerline{
\raisebox{-10pt}{$S_{\overline{\alpha}}=T_{\overline{\alpha}}=$\quad}
$\xymatrix@=5pt@M=-1pt@W=-1pt{
1 &&2\\
\ar@{-}[ddr] &&\ar@{-}[ddl]\\
\\
&\bullet\ar@{-}[dd]\\
\\
&
}$
}

Now assume we have constructed the trees $S_{\overline{\beta}}$, $T_{\overline{\beta}}$ for all simplices
$\overline{\beta}$ in $\mathfrak{K}_\bullet(a)$, for $a<m$. For a simplex $\overline{\alpha}$
in $\mathfrak{K}_\bullet(m)$, we consider three possible cases.

{\it Case 1:\/} Some of the objects $\alpha_i$ have outermost operation $\boxtimes_1$ while others have
outermost operation $\boxtimes_2$. 

In this case reindex the vertices so that $\alpha_0, \alpha_1,\dots,\alpha_s$ have outermost
operation $\boxtimes_1$, while $\alpha_{s+1},\dots,\alpha_r$ have outermost operation $\boxtimes_2$.
Let
\begin{eqnarray*}
\alpha_i&=&\alpha_{i1}\boxtimes_1\alpha_{i2}\boxtimes_1\dots\boxtimes_1\alpha_{ip}\quad 0\le i\le s\\
\alpha_i&=&\alpha_{i1}\boxtimes_2\alpha_{i2}\boxtimes_2\dots\boxtimes_2\alpha_{iq}\quad s+1\le i\le r
\end{eqnarray*}
be the maximal common decompositions as in Corollary \ref{label_3_5}. Let
\begin{eqnarray*}
U_j &= &|\alpha_{0j}|=|\alpha_{1j}|=\dots=|\alpha_{sj}|\quad 1\le j\le p\\
V_j &= &|\alpha_{s+1,j}|=|\alpha_{s+2,j}|=\dots=|\alpha_{rj}|\quad 1\le j\le q
\end{eqnarray*}

By recursion, we have already defined pairs of binodal trees $(S_{\overline{\alpha}\cap U_j},T_{\overline{\alpha}\cap U_j})$ for each $j=1,2,\dots, p$. Similarly
for each $j=1,2,\dots,q$ we have already defined
pairs of binodal trees $(S_{\overline{\alpha}\cap V_j},T_{\overline{\alpha}\cap V_j})$. We define
$S_{\overline{\alpha}}$ to be the binodal tree
$$
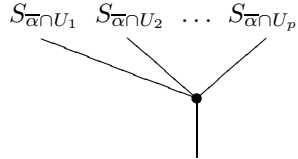
\xymatrix@=5pt@M=-1pt@W=-1pt{
S_{\overline{\alpha}\cap U_1} &&S_{\overline{\alpha}\cap U_2} &&\dots &&S_{\overline{\alpha}\cap U_p}\\
\ar@{-}[ddddrrrr] &&\ar@{-}[ddddrr] && &&\ar@{-}[ddddll]\\
\\
\\
\\
&& &&\bullet\ar@{-}[dddd]\\
\\
\\
\\
&& &&
}$$
and we define $T_{\overline{\alpha}}$ to be the binodal tree
$$
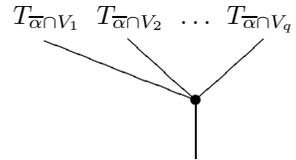
\xymatrix@=5pt@M=-1pt@W=-1pt{
T_{\overline{\alpha}\cap V_1 } &&T_{\overline{\alpha}\cap V_2} &&\dots &&T_{\overline{\alpha}\cap V_q}\\
\ar@{-}[ddddrrrr] &&\ar@{-}[ddddrr] && &&\ar@{-}[ddddll]\\
\\
\\
\\
&& &&\bullet\ar@{-}[dddd]\\
\\
\\
\\
&& &&
}$$

{\it Case 2:\/} All of the objects $\alpha_i$ have outermost operation $\boxtimes_1$.

Let
$$\alpha_i=\alpha_{i1}\boxtimes_1\alpha_{i2}\boxtimes_1\dots\boxtimes_1\alpha_{ip}\quad 0\le i\le r$$
and
$$U_j = |\alpha_{0j}|=|\alpha_{1j}|=\dots=|\alpha_{rj}|\quad 1\le j\le p$$
be as in Case 1.

By recursion, for each $j=1,2,\dots,p$ we have already defined
pairs of binodal trees $(S_{\overline{\alpha}\cap U_j},T_{\overline{\alpha}\cap U_j})$. We define
$S_{\overline{\alpha}}$ to be the binodal tree
$$
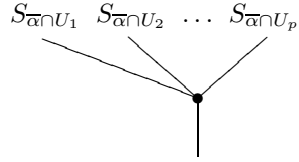
\xymatrix@=5pt@M=-1pt@W=-1pt{
S_{\overline{\alpha}\cap U_1} &&S_{\overline{\alpha}\cap U_2} &&\dots &&S_{\overline{\alpha}\cap U_p}\\
\ar@{-}[ddddrrrr] &&\ar@{-}[ddddrr] && &&\ar@{-}[ddddll]\\
\\
\\
\\
&& &&\bullet\ar@{-}[dddd]\\
\\
\\
\\
&& &&
}$$
and $T_{\overline{\alpha}}$ to be the binodal tree
$$
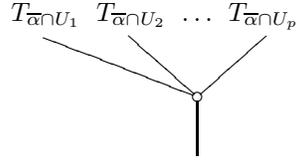
\xymatrix@=5pt@M=-1pt@W=-1pt{
T_{\overline{\alpha}\cap U_1} &&T_{\overline{\alpha}\cap U_2} &&\dots &&T_{\overline{\alpha}\cap U_p}\\
\ar@{-}[ddddrrrr] &&\ar@{-}[ddddrr] && &&\ar@{-}[ddddll]\\
\\
\\
\\
&& &&\circ\ar@{-}[dddd]\\
\\
\\
\\
&& &&
}$$

{\it Case 3:\/} All of the objects $\alpha_i$ have outermost operation $\boxtimes_2$.

Let
$$\alpha_i=\alpha_{i1}\boxtimes_2\alpha_{i2}\boxtimes_2\dots\boxtimes_2\alpha_{iq}\quad 0\le i\le r$$
and
$$V_j=|\alpha_{0j}|=|\alpha_{1j}|=\dots=|\alpha_{rj}|\quad 1\le j\le q$$
be as in Case 1.

By recursion, for each $j=1,2,\dots,q$ we have already defined
pairs of binodal trees $(S_{\overline{\alpha}\cap V_j},T_{\overline{\alpha}\cap V_j})$. We define
$S_{\overline{\alpha}}$ to be the binodal tree
$$
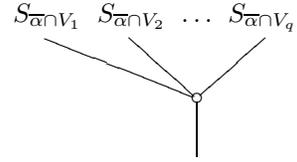
\xymatrix@=5pt@M=-1pt@W=-1pt{
S_{\overline{\alpha}\cap V_1} &&S_{\overline{\alpha}\cap V_2} &&\dots &&S_{\overline{\alpha}\cap V_q}\\
\ar@{-}[ddddrrrr] &&\ar@{-}[ddddrr] && &&\ar@{-}[ddddll]\\
\\
\\
\\
&& &&\circ\ar@{-}[dddd]\\
\\
\\
\\
&& &&
}$$
and $T_{\overline{\alpha}}$ to be the binodal tree
$$
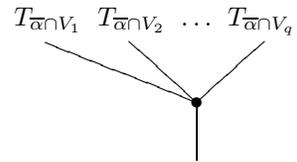
\xymatrix@=5pt@M=-1pt@W=-1pt{
T_{\overline{\alpha}\cap V_1} &&T_{\overline{\alpha}\cap V_2} &&\dots &&T_{\overline{\alpha}\cap V_q}\\
\ar@{-}[ddddrrrr] &&\ar@{-}[ddddrr] && &&\ar@{-}[ddddll]\\
\\
\\
\\
&& &&\bullet\ar@{-}[dddd]\\
\\
\\
\\
&& &&
}$$

To check that
$$\cap_{i=1}^r G_m^{ax}(\alpha_i)
=(W|\mathcal{NM}_k|)(S_{\overline{\alpha}})\times(W|\mathcal{NM}_l|)(T_{\overline{\alpha}})$$
we observe that it is true for the unique 1-simplex in the case $m=2$ (compare the discussion of binary 
operations in Section 6). If $m>2$, then by \ref{8_4}
we can reduce checking this to the case $m=3$. Up to permutations, there are two maximal simplices in 
$\mathfrak{M}^{ab}_2(3)$:
$$\{1\boxtimes_1 2\boxtimes_1 3,(1\boxtimes_2 2)\boxtimes_1 3,(2\boxtimes_1 3)\boxtimes_2 1, 1\boxtimes_2 2\boxtimes_2 3\}$$
$$\{1\boxtimes_1 2\boxtimes_1 3,(1\boxtimes_2 2)\boxtimes_1 3,(1\boxtimes_1 2)\boxtimes_2 3, 1\boxtimes_2 2\boxtimes_2 3\}$$

We index the vertices of the first simplex by the indices 0, 1, 2, 3 and we use the notations $\overline{\alpha}_{[i]}$,
$\overline{\alpha}_{[i,j]}$, $\overline{\alpha}_{[i,j,k]}$, $\overline{\alpha}_{[0,1,2,3]}$ to indicate the subsimplices
spanned by the vertices with those indices. We list below the pairs $(S_{\overline{\alpha}},T_{\overline{\alpha}})$
assigned to these simplices by our recursive algorithm.
\begin{eqnarray*}
\raisebox{-20pt}{$(S_{\overline{\alpha}_{[0]}},T_{\overline{\alpha}_{[0]}})$}
&\raisebox{-20pt}{=\ }
&\raisebox{-20pt}{$\left(\xymatrix{\\ \\}\right.$}
\xymatrix@=5pt@M=-1pt@W=-1pt{
1 &&2 &&3\\
\ar@{-}[ddddrr] &&\ar@{-}[dddd] &&\ar@{-}[ddddll]\\
\\
\\
\\
&&\bullet\ar@{-}[dd]\\
\\
&&
}\raisebox{-20pt}{,\quad}
\xymatrix@=5pt@M=-1pt@W=-1pt{
1 &&2 &&3\\
\ar@{-}[ddddrr] &&\ar@{-}[dddd] &&\ar@{-}[ddddll]\\
\\
\\
\\
&&\circ\ar@{-}[dd]\\
\\
&&
}
\raisebox{-20pt}{$\left.\xymatrix{\\ \\}\right)$}\\
\raisebox{-20pt}{$(S_{\overline{\alpha}_{[1]}},T_{\overline{\alpha}_{[1]}})$}
&\raisebox{-20pt}{=\ }
&\raisebox{-20pt}{$\left(\xymatrix{\\ \\}\right.$}
\xymatrix@=5pt@M=-1pt@W=-1pt{
1 &&2 &&&3\\
\ar@{-}[ddr] &&\ar@{-}[ddl]&&&\ar@{-}[ddddll]\\
\\
&\circ\ar@{-}[ddrr]\\
\\
&&&\bullet\ar@{-}[dd]\\
\\
&&&
}
\raisebox{-20pt}{,\quad}
\xymatrix@=5pt@M=-1pt@W=-1pt{
1 &&2 &&&3\\
\ar@{-}[ddr] &&\ar@{-}[ddl]&&&\ar@{-}[ddddll]\\
\\
&\bullet\ar@{-}[ddrr]\\
\\
&&&\circ\ar@{-}[dd]\\
\\
&&&
}
\raisebox{-20pt}{$\left.\xymatrix{\\ \\}\right)$}\\
\raisebox{-20pt}{$(S_{\overline{\alpha}_{[2]}},T_{\overline{\alpha}_{[2]}})$}
&\raisebox{-20pt}{=\ }
&\raisebox{-20pt}{$\left(\xymatrix{\\ \\}\right.$}
\xymatrix@=5pt@M=-1pt@W=-1pt{
2 &&3 &&&1\\
\ar@{-}[ddr] &&\ar@{-}[ddl]&&&\ar@{-}[ddddll]\\
\\
&\bullet\ar@{-}[ddrr]\\
\\
&&&\circ\ar@{-}[dd]\\
\\
&&&
}
\raisebox{-20pt}{,\quad}
\xymatrix@=5pt@M=-1pt@W=-1pt{
2 &&3 &&&1\\
\ar@{-}[ddr] &&\ar@{-}[ddl]&&&\ar@{-}[ddddll]\\
\\
&\circ\ar@{-}[ddrr]\\
\\
&&&\bullet\ar@{-}[dd]\\
\\
&&&
}
\raisebox{-20pt}{$\left.\xymatrix{\\ \\}\right)$}\\
\raisebox{-20pt}{$(S_{\overline{\alpha}_{[3]}},T_{\overline{\alpha}_{[3]}})$}
&\raisebox{-20pt}{=\ }
&\raisebox{-20pt}{$\left(\xymatrix{\\ \\}\right.$}
\xymatrix@=5pt@M=-1pt@W=-1pt{
1 &&2 &&3\\
\ar@{-}[ddddrr] &&\ar@{-}[dddd] &&\ar@{-}[ddddll]\\
\\
\\
\\
&&\circ\ar@{-}[dd]\\
\\
&&
}\raisebox{-20pt}{,\quad}
\xymatrix@=5pt@M=-1pt@W=-1pt{
1 &&2 &&3\\
\ar@{-}[ddddrr] &&\ar@{-}[dddd] &&\ar@{-}[ddddll]\\
\\
\\
\\
&&\bullet\ar@{-}[dd]\\
\\
&&
}
\raisebox{-20pt}{$\left.\xymatrix{\\ \\}\right)$}\\
\raisebox{-20pt}{$(S_{\overline{\alpha}_{[0,1]}},T_{\overline{\alpha}_{[0,1]}})$}
&\raisebox{-20pt}{=\ }
&\raisebox{-20pt}{$\left(\xymatrix{\\ \\}\right.$}
\xymatrix@=5pt@M=-1pt@W=-1pt{
1 &&2 &&&3\\
\ar@{-}[ddr] &&\ar@{-}[ddl]&&&\ar@{-}[ddddll]\\
\\
&\bullet\ar@{-}[ddrr]\\
\\
&&&\bullet\ar@{-}[dd]\\
\\
&&&
}
\raisebox{-20pt}{,\quad}
\xymatrix@=5pt@M=-1pt@W=-1pt{
1 &&2 &&&3\\
\ar@{-}[ddr] &&\ar@{-}[ddl]&&&\ar@{-}[ddddll]\\
\\
&\bullet\ar@{-}[ddrr]\\
\\
&&&\circ\ar@{-}[dd]\\
\\
&&&
}
\raisebox{-20pt}{$\left.\xymatrix{\\ \\}\right)$}\\
\end{eqnarray*}
\begin{eqnarray*}
\raisebox{-20pt}{$(S_{\overline{\alpha}_{[0,2]}},T_{\overline{\alpha}_{[0,2]}})$}
&\raisebox{-20pt}{=\ }
&\raisebox{-20pt}{$\left(\xymatrix{\\ \\}\right.$}
\xymatrix@=5pt@M=-1pt@W=-1pt{
1 &&2 &&3\\
\ar@{-}[ddddrr] &&\ar@{-}[dddd] &&\ar@{-}[ddddll]\\
\\
\\
\\
&&\bullet\ar@{-}[dd]\\
\\
&&
}
\raisebox{-20pt}{,\quad}
\xymatrix@=5pt@M=-1pt@W=-1pt{
2 &&3 &&&1\\
\ar@{-}[ddr] &&\ar@{-}[ddl]&&&\ar@{-}[ddddll]\\
\\
&\circ\ar@{-}[ddrr]\\
\\
&&&\bullet\ar@{-}[dd]\\
\\
&&&
}
\raisebox{-20pt}{$\left.\xymatrix{\\ \\}\right)$}\\
\raisebox{-20pt}{$(S_{\overline{\alpha}_{[0,3]}},T_{\overline{\alpha}_{[0,3]}})$}
&\raisebox{-20pt}{=\ }
&\raisebox{-20pt}{$\left(\xymatrix{\\ \\}\right.$}
\xymatrix@=5pt@M=-1pt@W=-1pt{
1 &&2 &&3\\
\ar@{-}[ddddrr] &&\ar@{-}[dddd] &&\ar@{-}[ddddll]\\
\\
\\
\\
&&\bullet\ar@{-}[dd]\\
\\
&&
}
\raisebox{-20pt}{,\quad}
\xymatrix@=5pt@M=-1pt@W=-1pt{
1 &&2 &&3\\
\ar@{-}[ddddrr] &&\ar@{-}[dddd] &&\ar@{-}[ddddll]\\
\\
\\
\\
&&\bullet\ar@{-}[dd]\\
\\
&&
}
\raisebox{-20pt}{$\left.\xymatrix{\\ \\}\right)$}\\
\raisebox{-20pt}{$(S_{\overline{\alpha}_{[1,2]}},T_{\overline{\alpha}_{[1,2]}})$}
&\raisebox{-20pt}{=\ }
&\raisebox{-20pt}{$\left(\xymatrix{\\ \\}\right.$}
\xymatrix@=5pt@M=-1pt@W=-1pt{
1 &&2 &&&3\\
\ar@{-}[ddr] &&\ar@{-}[ddl]&&&\ar@{-}[ddddll]\\
\\
&\circ\ar@{-}[ddrr]\\
\\
&&&\bullet\ar@{-}[dd]\\
\\
&&&
}
\raisebox{-20pt}{,\quad}
\xymatrix@=5pt@M=-1pt@W=-1pt{
2 &&3 &&&1\\
\ar@{-}[ddr] &&\ar@{-}[ddl]&&&\ar@{-}[ddddll]\\
\\
&\circ\ar@{-}[ddrr]\\
\\
&&&\bullet\ar@{-}[dd]\\
\\
&&&
}
\raisebox{-20pt}{$\left.\xymatrix{\\ \\}\right)$}\\
\raisebox{-20pt}{$(S_{\overline{\alpha}_{[1,3]}},T_{\overline{\alpha}_{[1,3]}})$}
&\raisebox{-20pt}{=\ }
&\raisebox{-20pt}{$\left(\xymatrix{\\ \\}\right.$}
\xymatrix@=5pt@M=-1pt@W=-1pt{
1 &&2 &&&3\\
\ar@{-}[ddr] &&\ar@{-}[ddl]&&&\ar@{-}[ddddll]\\
\\
&\circ\ar@{-}[ddrr]\\
\\
&&&\bullet\ar@{-}[dd]\\
\\
&&&
}
\raisebox{-20pt}{,\quad}
\xymatrix@=5pt@M=-1pt@W=-1pt{
1 &&2 &&3\\
\ar@{-}[ddddrr] &&\ar@{-}[dddd] &&\ar@{-}[ddddll]\\
\\
\\
\\
&&\bullet\ar@{-}[dd]\\
\\
&&
}
\raisebox{-20pt}{$\left.\xymatrix{\\ \\}\right)$}\\
\raisebox{-20pt}{$(S_{\overline{\alpha}_{[2,3]}},T_{\overline{\alpha}_{[2,3]}})$}
&\raisebox{-20pt}{=\ }
&\raisebox{-20pt}{$\left(\xymatrix{\\ \\}\right.$}
\xymatrix@=5pt@M=-1pt@W=-1pt{
2 &&3 &&&1\\
\ar@{-}[ddr] &&\ar@{-}[ddl]&&&\ar@{-}[ddddll]\\
\\
&\bullet\ar@{-}[ddrr]\\
\\
&&&\circ\ar@{-}[dd]\\
\\
&&&
}
\raisebox{-20pt}{,\quad}
\xymatrix@=5pt@M=-1pt@W=-1pt{
2 &&3 &&&1\\
\ar@{-}[ddr] &&\ar@{-}[ddl]&&&\ar@{-}[ddddll]\\
\\
&\bullet\ar@{-}[ddrr]\\
\\
&&&\bullet\ar@{-}[dd]\\
\\
&&&
}
\raisebox{-20pt}{$\left.\xymatrix{\\ \\}\right)$}\\
\raisebox{-20pt}{$(S_{\overline{\alpha}_{[0,1,2]}},T_{\overline{\alpha}_{[0,1,2]}})$}
&\raisebox{-20pt}{=\ }
&\raisebox{-20pt}{$\left(\xymatrix{\\ \\}\right.$}
\xymatrix@=5pt@M=-1pt@W=-1pt{
1 &&2 &&&3\\
\ar@{-}[ddr] &&\ar@{-}[ddl]&&&\ar@{-}[ddddll]\\
\\
&\bullet\ar@{-}[ddrr]\\
\\
&&&\bullet\ar@{-}[dd]\\
\\
&&&
}
\raisebox{-20pt}{,\quad}
\xymatrix@=5pt@M=-1pt@W=-1pt{
2 &&3 &&&1\\
\ar@{-}[ddr] &&\ar@{-}[ddl]&&&\ar@{-}[ddddll]\\
\\
&\circ\ar@{-}[ddrr]\\
\\
&&&\bullet\ar@{-}[dd]\\
\\
&&&
}
\raisebox{-20pt}{$\left.\xymatrix{\\ \\}\right)$}\\
\raisebox{-20pt}{$(S_{\overline{\alpha}_{[0,1,3]}},T_{\overline{\alpha}_{[0,1,3]}})$}
&\raisebox{-20pt}{=\ }
&\raisebox{-20pt}{$\left(\xymatrix{\\ \\}\right.$}
\xymatrix@=5pt@M=-1pt@W=-1pt{
1 &&2 &&&3\\
\ar@{-}[ddr] &&\ar@{-}[ddl]&&&\ar@{-}[ddddll]\\
\\
&\bullet\ar@{-}[ddrr]\\
\\
&&&\bullet\ar@{-}[dd]\\
\\
&&&
}
\raisebox{-20pt}{,\quad}
\xymatrix@=5pt@M=-1pt@W=-1pt{
1 &&2 &&3\\
\ar@{-}[ddddrr] &&\ar@{-}[dddd] &&\ar@{-}[ddddll]\\
\\
\\
\\
&&\bullet\ar@{-}[dd]\\
\\
&&
}
\raisebox{-20pt}{$\left.\xymatrix{\\ \\}\right)$}\\
\raisebox{-20pt}{$(S_{\overline{\alpha}_{[0,2,3]}},T_{\overline{\alpha}_{[0,2,3]}})$}
&\raisebox{-20pt}{=\ }
&\raisebox{-20pt}{$\left(\xymatrix{\\ \\}\right.$}
\xymatrix@=5pt@M=-1pt@W=-1pt{
1 &&2 &&3\\
\ar@{-}[ddddrr] &&\ar@{-}[dddd] &&\ar@{-}[ddddll]\\
\\
\\
\\
&&\bullet\ar@{-}[dd]\\
\\
&&
}
\raisebox{-20pt}{,\quad}
\xymatrix@=5pt@M=-1pt@W=-1pt{
2 &&3 &&&1\\
\ar@{-}[ddr] &&\ar@{-}[ddl]&&&\ar@{-}[ddddll]\\
\\
&\bullet\ar@{-}[ddrr]\\
\\
&&&\bullet\ar@{-}[dd]\\
\\
&&&
}
\raisebox{-20pt}{$\left.\xymatrix{\\ \\}\right)$}\\
\end{eqnarray*}
\begin{eqnarray*}
\raisebox{-20pt}{$(S_{\overline{\alpha}_{[1,2,3]}},T_{\overline{\alpha}_{[1,2,3]}})$}
&\raisebox{-20pt}{=\ }
&\raisebox{-20pt}{$\left(\xymatrix{\\ \\}\right.$}
\xymatrix@=5pt@M=-1pt@W=-1pt{
1 &&2 &&&3\\
\ar@{-}[ddr] &&\ar@{-}[ddl]&&&\ar@{-}[ddddll]\\
\\
&\circ\ar@{-}[ddrr]\\
\\
&&&\bullet\ar@{-}[dd]\\
\\
&&&
}
\raisebox{-20pt}{,\quad}
\xymatrix@=5pt@M=-1pt@W=-1pt{
2 &&3 &&&1\\
\ar@{-}[ddr] &&\ar@{-}[ddl]&&&\ar@{-}[ddddll]\\
\\
&\bullet\ar@{-}[ddrr]\\
\\
&&&\bullet\ar@{-}[dd]\\
\\
&&&
}
\raisebox{-20pt}{$\left.\xymatrix{\\ \\}\right)$}\\
\raisebox{-20pt}{$(S_{\overline{\alpha}_{[0,1,2,3]}},T_{\overline{\alpha}_{[0,1,2,3]}})$}
&\raisebox{-20pt}{=\ }
&\raisebox{-20pt}{$\left(\xymatrix{\\ \\}\right.$}
\xymatrix@=5pt@M=-1pt@W=-1pt{
1 &&2 &&&3\\
\ar@{-}[ddr] &&\ar@{-}[ddl]&&&\ar@{-}[ddddll]\\
\\
&\bullet\ar@{-}[ddrr]\\
\\
&&&\bullet\ar@{-}[dd]\\
\\
&&&
}
\raisebox{-20pt}{,\quad}
\xymatrix@=5pt@M=-1pt@W=-1pt{
2 &&3 &&&1\\
\ar@{-}[ddr] &&\ar@{-}[ddl]&&&\ar@{-}[ddddll]\\
\\
&\bullet\ar@{-}[ddrr]\\
\\
&&&\bullet\ar@{-}[dd]\\
\\
&&&
}
\raisebox{-20pt}{$\left.\xymatrix{\\ \\}\right)$}\\
\end{eqnarray*}
We can check that
$$\cap_{i=0}^r G_m^{ax}(\alpha_i)
=(W|\mathcal{NM}_k|)(S_{\overline{\alpha}})\times(W|\mathcal{NM}_l|)(T_{\overline{\alpha}})$$
holds for each of these simplices by repeatedly applying Proposition \ref{8_5}.

For the second simplex
$$\{1\boxtimes_1 2\boxtimes_1 3,(1\boxtimes_2 2)\boxtimes_1 3,(1\boxtimes_1 2)\boxtimes_2 3, 1\boxtimes_2 2\boxtimes_2 3\}$$
any subsimplex not containing the edge
$\{(1\boxtimes_2 2)\boxtimes_1 3,(1\boxtimes_1 2)\boxtimes_2 3\}$
is, up to permutation, a subsimplex of the first maximal simplex and thus has already been analyzed. For any subsimplex
containing this edge, the recursive algorithm assigns
\begin{eqnarray*}
\raisebox{-20pt}{$(S_{\overline{\alpha}},T_{\overline{\alpha}})$}
&\raisebox{-20pt}{=\ }
&\raisebox{-20pt}{$\left(\xymatrix{\\ \\}\right.$}
\xymatrix@=5pt@M=-1pt@W=-1pt{
1 &&2 &&&3\\
\ar@{-}[ddr] &&\ar@{-}[ddl]&&&\ar@{-}[ddddll]\\
\\
&\bullet\ar@{-}[ddrr]\\
\\
&&&\bullet\ar@{-}[dd]\\
\\
&&&
}
\raisebox{-20pt}{,\quad}
\xymatrix@=5pt@M=-1pt@W=-1pt{
1 &&2 &&&3\\
\ar@{-}[ddr] &&\ar@{-}[ddl]&&&\ar@{-}[ddddll]\\
\\
&\bullet\ar@{-}[ddrr]\\
\\
&&&\bullet\ar@{-}[dd]\\
\\
&&&
}
\raisebox{-20pt}{$\left.\xymatrix{\\ \\}\right)$}\\
\end{eqnarray*}
Again we can check that
$$\cap_{i=0}^r G_m^{ax}(\alpha_i)
=(W|\mathcal{NM}_k|)(S_{\overline{\alpha}})\times(W|\mathcal{NM}_l|)(T_{\overline{\alpha}})$$
holds for each of these simplices by repeatedly applying Proposition \ref{8_5}.
\end{proof}

\begin{prop}\label{label_5}Let $\overline{\alpha}=\{\alpha_0,\alpha_1,\dots,\alpha_r\}$ be an $r$-simplex
in $\mathfrak{K}_\bullet(m)$ and let
$$\cap_{i=0}^r G_m^{ax}(\alpha_i)=(W|\mathcal{NM}_k|)(S_{\overline{\alpha}})\times(W|\mathcal{NM}_l|)(T_{\overline{\alpha}}).$$
Then
\begin{enumerate}
\item[(1)] $\cap_{i=0}^r G_m^{ax}(\alpha_i)$ has a cellular decomposition over the poset
$\scM_k(S_{\overline{\alpha}})\times\scM_k(T_{\overline{\alpha}})$
\item[(2)] If $\overline{\alpha}$ is a subsimplex of $\overline{\beta}$, then the inclusion
$\cap_j G_m^{ax}(\beta_j)\subset\cap_i G_m^{ax}(\alpha_i)$
is compatible with the cellular decomposition.
\item[(3)] The axial map sends $\cap_{i=0}^r G_m(\alpha_i)$ homeomorphically onto $\cap_{i=0}^r G_m^{ax}(\alpha_i)$.
\end{enumerate}
\end{prop}

\begin{proof}
Part (1) is an immediate consequence of Corollary \ref{8_10}. To establish the second part, we proceed by induction
on $m$  and refer back to the proof of Proposition \ref{label_4}. If $m=1$, the diagram is trivial and there is nothing
to prove. So let us assume we have established this for all $a<m$. It suffices to check the compatibility with the cellular
decompositions for inclusions $G'_m(\overline{\beta})\subset G'_m(\overline{\alpha})$, where the simplex $\overline{\beta}$
is obtained by adding a single vertex $\lambda$ to the simplex $\overline{\alpha}$. 

We refer back to the recursive algorithm of Proposition \ref{label_4} and analyze
several different cases. For convenience sake we will call vertices of our simplices as
$\boxtimes_1$-vertices if their outermost operation is $\boxtimes_1$, and $\boxtimes_2$-vertices
if their outermost operation is $\boxtimes_2$.

{\it Case (a):\/} Suppose $\overline{\alpha}$ contains both $\boxtimes_1$-vertices and
$\boxtimes_2$-vertices, the new vertex $\lambda$ is a $\boxtimes_1$-vertex and has a
$\boxtimes_1$ decomposition compatible with the maximal common $\boxtimes_1$ decomposition
of the $\boxtimes_1$-vertices of $\overline{\alpha}$.

In this case both simplices $\overline{\alpha}$ and $\overline{\beta}$ are in Case 1 of the proof of Proposition
\ref{label_4}. Thus the inclusion $(W\mathcal{NM}_k)(S_{\overline{\beta}})\subset (W\mathcal{NM}_k)(S_{\overline{\alpha}})$
takes the form
\begin{eqnarray*}
\lefteqn{W|\mathcal{NM}_k|(p)\times_{W|\mathcal{NM}_k|(1)^p}\prod_{j=1}^p(W|\mathcal{NM}_k|)(S_{\overline{\beta}\cap U_j})}
\qquad\qquad\qquad\qquad\\
 &\subset
&W|\mathcal{NM}_k|(p)\times_{W|\mathcal{NM}_k|(1)^p}\prod_{j=1}^p (W|\mathcal{NM}_k|)(S_{\overline{\alpha}\cap U_j})
\end{eqnarray*}
and is thus compatible with the cellular decomposition by the induction hypothesis. Similarly the inclusion
$(W|\mathcal{NM}_l|)(T_{\overline{\beta}})\subset(W|\mathcal{NM}_l|)(T_{\overline{\alpha}})$ takes the form
\begin{eqnarray*}
\lefteqn{W\mathcal{NM}_l(q)\times_{W|\mathcal{NM}_l|(1)^q}\prod_{j=1}^q(W|\mathcal{NM}_l|)(T_{\overline{\beta}\cap V_j})}
\qquad\qquad\qquad\qquad\\
 &\subset
&W|\mathcal{NM}_l|(q)\times_{W|\mathcal{NM}_l|(1)^q}\prod_{j=1}^q(W|\mathcal{NM}_l|)(T_{\overline{\alpha}\cap V_j})
\end{eqnarray*}
which again is compatible with the cellular decompositions for the same reason.

A similar analysis, reducing the proof to the induction hypothesis, applies in Cases (b), (c) and (d) below.

{\it Case (b):\/}  Suppose $\overline{\alpha}$ contains both $\boxtimes_1$-vertices and
$\boxtimes_2$-vertices, the new vertex $\lambda$ is a $\boxtimes_2$-vertex and has a
$\boxtimes_2$ decomposition compatible with the maximal common $\boxtimes_2$ decomposition
of the $\boxtimes_2$-vertices of $\overline{\alpha}$.

{\it Case (c):\/} Suppose $\overline{\beta}$ contains only  $\boxtimes_1$-vertices  and the
new vertex $\lambda$ has a $\boxtimes_1$ decomposition compatible with the maximal common $\boxtimes_1$ decomposition of the vertices of $\overline{\alpha}$.

{\it Case (d):\/} Suppose $\overline{\beta}$ contains only  $\boxtimes_2$-vertices  and the
new vertex $\lambda$ has a $\boxtimes_2$ decomposition compatible with the maximal common $\boxtimes_2$ decomposition of the vertices of $\overline{\alpha}$.

{\it Case (e):\/} Suppose $\overline{\alpha}$ contains both $\boxtimes_1$-vertices and
$\boxtimes_2$-vertices and the new vertex $\lambda$ is a $\boxtimes_1$-vertex which
does not have a $\boxtimes_1$ decomposition compatible  with the maximal common $\boxtimes_1$ decomposition of the $\boxtimes_1$-vertices of $\overline{\alpha}$.

Again both simplicies $\overline{\alpha}$ and $\overline{\beta}$ are in Case 1 of the proof of Proposition
\ref{label_4} but now the maximal common $\boxtimes_1$-decomposition of the $\boxtimes_1$-vertices
of $\overline{\alpha}$ is finer (i.e. has more summands) than the maximal common 
$\boxtimes_1$-decomposition of the $\boxtimes_1$-vertices of $\overline{\beta}$.

In this case the inclusion 
$W|\mathcal{NM}_k|(S_{\overline{\beta}})\subset W|\mathcal{NM}_k|(S_{\overline{\alpha}})$
corresponds to a sequence of intersections of the type\newline
\centerline{
\raisebox{-30pt}{$W|\mathcal{NM}_k|\left(\begin{array}{c}\\ \\ \\ \end{array}\right.$}
$\xymatrix@=5pt@M=-1pt@W=-1pt{
\ &\  &&\dots &\  &&\  &\  &&\dots &\ \\
\ar@{-}[dddrr] &\ar@{-}[dddr] && &\ar@{-}[dddll] &&\ar@{-}[ddddddll] &\ar@{-}[ddddddlll] &&&\ar@{-}[ddddddllllll]\\
\\
\\
&&\circ\ar@{-}[dddrr]&&&&\\
\\
\\
&&&&\bullet\ar@{-}[ddd]&&&\\
&&&&&&\\
\\
&&&&
}$
\raisebox{-30pt}{$\left.\begin{array}{c}\\ \\ \\ \end{array}\right)\bigcap W|\mathcal{NM}_k|
\left(\begin{array}{c}\\ \\ \\ \end{array}\right.$}
$\xymatrix@=5pt@M=-1pt@W=-1pt{
\  &\ &&\dots &\ &&\ &\ &&\dots &\ \\
\ar@{-}[ddddddrrrr] &\ar@{-}[ddddddrrr] && &\ar@{-}[dddddd] &&\ar@{-}[ddddddll] &\ar@{-}[ddddddlll] &&&\ar@{-}[ddddddllllll]\\
\\
\\
&&&&&&\\
\\
\\
&&&&\bullet\ar@{-}[ddd]&&&\\
&&&&&&\\
\\
&&&&
}$
\raisebox{-30pt}{$\left.\begin{array}{c}\\ \\ \\ \end{array}\right)$}
}\newline
\centerline{
\raisebox{-30pt}{$\qquad\qquad\qquad\qquad=
W|\mathcal{NM}_k|\left(\begin{array}{c}\\ \\ \\ \end{array}\right.$}
$\xymatrix@=5pt@M=-1pt@W=-1pt{
\ &\  &&\dots &\  &&\  &\  &&\dots &\ \\
\ar@{-}[dddrr] &\ar@{-}[dddr] && &\ar@{-}[dddll] &&\ar@{-}[ddddddll] &\ar@{-}[ddddddlll] &&&\ar@{-}[ddddddllllll]\\
\\
\\
&&\bullet\ar@{-}[dddrr]&&&&\\
\\
\\
&&&&\bullet\ar@{-}[ddd]&&&\\
&&&&&&\\
\\
&&&&
}$
\raisebox{-30pt}{$\left.\begin{array}{c}\\ \\ \\ \end{array}\right)$}}\newline
Hence $S_{\overline{\alpha}}$ can be obtained from $S_{\overline{\beta}}$ by a sequence
of the following types of moves:\newline
\centerline{$\xymatrix@=5pt@M=-1pt@W=-1pt{
W_1 &W_2 &&\dots &W_{u-1} &&W_u &W_{u+1} &&\dots &W_s\\
\ar@{-}[dddrr] &\ar@{-}[dddr] && &\ar@{-}[dddll] &&\ar@{-}[ddddddll] &\ar@{-}[ddddddlll] &&&\ar@{-}[ddddddllllll]\\
\\
\\
&&\bullet\ar@{-}[dddrr]&&&&\\
\\
\\
&&&&\bullet\ar@{-}[ddd]&&&\\
&&&&&&\\
\\
&&&&
}$
\raisebox{-30pt}{$\leadsto$}
$\xymatrix@=5pt@M=-1pt@W=-1pt{
W'_1 &W'_2 &&\dots &W'_{u-1} &&W_u &W_{u+1} &&\dots &W_s\\
\ar@{-}[ddddddrrrr] &\ar@{-}[ddddddrrr] && &\ar@{-}[dddddd] &&\ar@{-}[ddddddll] &\ar@{-}[ddddddlll] &&&\ar@{-}[ddddddllllll]\\
\\
\\
&&&&&&\\
\\
\\
&&&&\bullet\ar@{-}[ddd]&&&\\
&&&&&&\\
\\
&&&&
}$}
where $W_j$ and $W'_j$ are obtained by restricting $S_{\overline{\beta}}$, respectively $S_{\overline{\alpha}}$, to the same
set of inputs. The induced map on $W|\mathcal{NM}_k|(-)$ is the product of the composition map:
\begin{eqnarray*}
W|\mathcal{NM}_k|(s-u+1)\times W|\mathcal{NM}_k|(u-1)&\longrightarrow &W|\mathcal{NM}_k|(s)\\
(a,b) &\mapsto &a\circ(b\oplus id\oplus\dots\oplus id),
\end{eqnarray*}
which is obviously compatible with the cellular decomposition, and the inclusion map:
$$\prod_{j=1}^s W|\mathcal{NM}_k|(W_j)\subset \prod_{j=1}^s W|\mathcal{NM}_k|(W'_j)$$
which is compatible with the cellular decomposition by inductions hypothesis. The induced map
$(W|\mathcal{NM}_l|)(T_{\overline{\beta}})\subset(W|\mathcal{NM}_l|)(T_{\overline{\alpha}})$ takes the form
\begin{eqnarray*}
\lefteqn{W|\mathcal{NM}_l|(q)\times_{W|\mathcal{NM}_l|(1)^q}\prod_{j=1}^q(W|\mathcal{NM}_l|)(T_{\overline{\beta}\cap V_j})}
\qquad\qquad\qquad\qquad\\
 &\subset
&W|\mathcal{NM}_l|(q)\times_{W|\mathcal{NM}_l|(1)^q}\prod_{j=1}^q(W|\mathcal{NM}_l|)(T_{\overline{\alpha}\cap V_j}),
\end{eqnarray*}
for which compatibility with the cellular decomposition follows by induction hypothesis.

{\it Case (f):\/} Suppose $\overline{\alpha}$ contains both $\boxtimes_1$-vertices and
$\boxtimes_2$-vertices and the new vertex $\lambda$ is a $\boxtimes_2$-vertex which
does not have a $\boxtimes_2$ decomposition compatible  with the maximal common $\boxtimes_2$ decomposition of the $\boxtimes_2$-vertices of $\overline{\alpha}$.

In this case, a similar analysis as in Case (e) applies, with the roles of the inclusions
$(W|\mathcal{NM}_k|)(S_{\overline{\beta}})\subset(W|\mathcal{NM}_k|)(S_{\overline{\alpha}})$ and
$(W|\mathcal{NM}_l|)(T_{\overline{\beta}})\subset(W|\mathcal{NM}_l|)(T_{\overline{\alpha}})$ reversed.

{\it Case (g):\/} Suppose $\overline{\alpha}$ contains only $\boxtimes_1$-vertices and
and the new vertex $\lambda$ is a $\boxtimes_1$-vertex which
does not have a $\boxtimes_1$ decomposition compatible  with the maximal common $\boxtimes_1$ decomposition of the vertices of $\overline{\alpha}$.

In this case the argument that the resulting inclusion is compatible with the cellular decomposition is similar
to that in Case (e).

{\it Case (h):\/} Suppose $\overline{\alpha}$ contains only $\boxtimes_2$-vertices and
and the new vertex $\lambda$ is a $\boxtimes_2$-vertex which
does not have a $\boxtimes_2$ decomposition compatible  with the maximal common $\boxtimes_2$ decomposition of the vertices of $\overline{\alpha}$.

In this case the argument that the resulting inclusion is compatible with the cellular decomposition is similar
to that in Case (f).

{\it Case (i):\/} Suppose $\overline{\alpha}$ contains only $\boxtimes_1$-vertices and the new vertex $\lambda$ is $\boxtimes_2$-vertex.

In this case the induced inclusion 
$(W|\mathcal{NM}_k|)(S_{\overline{\beta}})\subset(W|\mathcal{NM}_k|)(S_{\overline{\alpha}})$ takes the form
\begin{eqnarray*}
\lefteqn{W|\mathcal{NM}_k|(p)\times_{W|\mathcal{NM}_k|(1)^p}\prod_{j=1}^p(W|\mathcal{NM}_k|)(S_{\overline{\beta}\cap U_j})}
\qquad\qquad\qquad\qquad\\
 &\subset
&W|\mathcal{NM}_k|(p)\times_{W|\mathcal{NM}_k|(1)^p}\prod_{j=1}^p(W|\mathcal{NM}_k|)(S_{\overline{\alpha}\cap U_j})
\end{eqnarray*}
and is thus compatible with the cellular decomposition by the induction hypothesis. The inclusion
$(W|\mathcal{NM}_l|)(T_{\overline{\beta}})\subset(W|\mathcal{NM}_l|)(T_{\overline{\alpha}})$ is more complicated.
The trees $T_{\overline{\beta}}$ and $T_{\overline{\alpha}}$ are related as shown below:
$$\mbox{
\setlength{\unitlength}{0.75\unitlength}
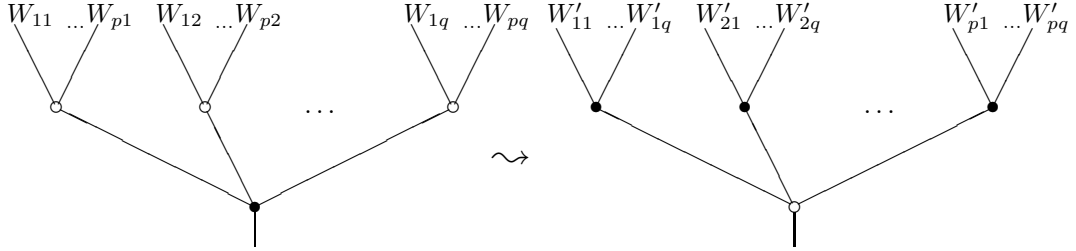
\begin{picture}(240,120)(-110,0)
\put(0,0){\line(0,1){20}}
\put(0,20){\circle*{5}}
\put(0,20){\line(-2,1){98}}
\put(0,20){\line(-1,2){24}}
\put(0,20){\line(2,1){98}}
\put(25,67){$\dots$}
\put(-125,113){$W_{11}$}
\put(-85,113){$W_{p1}$}
\put(4,12){}
\put(-100,70){\circle{5}}
\put(-101,72){\line(-1,2){20}}
\put(-99,72){\line(1,2){20}}
\put(-95,70){}
\put(-95,110){$\scriptstyle\dots$}
\put(-50,113){$W_{12}$}
\put(-10,113){$W_{p2}$}
\put(-25,70){\circle{5}}
\put(-26,72){\line(-1,2){20}}
\put(-24,72){\line(1,2){20}}
\put(-20,70){}
\put(-20,110){$\scriptstyle\dots$}
\put(100,70){\circle{5}}
\put(99,72){\line(-1,2){20}}
\put(101,72){\line(1,2){20}}
\put(105,70){}
\put(105,110){$\scriptstyle\dots$}
\put(75,113){$W_{1q}$}
\put(115,113){$W_{pq}$}
\put(120,40){\Large$\leadsto$}
\end{picture}
}\mbox{\qquad\qquad
\setlength{\unitlength}{0.75\unitlength}
\begin{picture}(220,120)(-80,0)
\put(0,0){\line(0,1){17}}
\put(0,20){\circle{5}}
\put(-2,21){\line(-2,1){100}}
\put(-1,22){\line(-1,2){25}}
\put(2,21){\line(2,1){100}}
\put(35,67){$\dots$}
\put(4,12){}
\put(-100,70){\circle*{5}}
\put(-100,70){\line(-1,2){20}}
\put(-100,70){\line(1,2){20}}
\put(-95,70){}
\put(-95,110){$\scriptstyle\dots$}
\put(-125,113){$W'_{11}$}
\put(-85,113){$W'_{1q}$}
\put(-25,70){\circle*{5}}
\put(-25,70){\line(-1,2){20}}
\put(-25,70){\line(1,2){20}}
\put(-20,70){}
\put(-20,110){$\scriptstyle\dots$}
\put(-50,113){$W'_{21}$}
\put(-10,113){$W'_{2q}$}
\put(100,70){\circle*{5}}
\put(100,70){\line(-1,2){20}}
\put(100,70){\line(1,2){20}}
\put(105,70){}
\put(105,110){$\scriptstyle\dots$}
\put(75,113){$W'_{p1}$}
\put(115,113){$W'_{pq}$}
\end{picture}
}$$
Here $W_{ij}$ and $W'_{ij}$ have the same sets of inputs, some of which may be empty, in which case the trees are understood to
be empty as well. The induced inclusion $(W|\mathcal{NM}_l|)(T_{\overline{\beta}})\subset(W|\mathcal{NM}_l|)(T_{\overline{\alpha}})$
is then the product of the maps
$$\mathcal{NM}_l(q)\longrightarrow \prod_{j=1}^p\mathcal{NM}_l(a_i)$$
and
$$\prod_{i,j}\mathcal{NM}_l(W_{ij})\subset \prod_{i,j}\mathcal{NM}_l(W'_{ij}).$$
Here $a_i$ is the cardinality of $\displaystyle\bigcup_{j=1}^q |W'_{ij}|$ and the first map is given by
$$\phi\mapsto\left(\phi\cdot(\epsilon_{11}\oplus\dots\oplus\epsilon_{1q}),
\phi\cdot(\epsilon_{21}\oplus\dots\oplus\epsilon_{2q}),\dots,
\phi\cdot(\epsilon_{p1}\oplus\dots\oplus\epsilon_{pq})\right),$$
where $\epsilon_{ij}=id\in\mathcal{NM}_l(1)$ if $|W'_{ij}|$ is nonempty or $\epsilon_{ij}=0\in\mathcal{NM}_l(0)$ otherwise.
Clearly this is compatible with the cellular decompositions while the same is true for the second map by the induction hypothesis.

{\it Case (j):\/} Suppose $\overline{\alpha}$ contains only $\boxtimes_2$-vertices and the new vertex $\lambda$ is $\boxtimes_1$-vertex.

In this case we argue as in Case (i), except with the roles of the inclusions
$(W|\mathcal{NM}_k|)(S_{\overline{\beta}})\subset(W|\mathcal{NM}_k|)(S_{\overline{\alpha}})$ and
$(W|\mathcal{NM}_l|)(T_{\overline{\beta}})\subset(W|\mathcal{NM}_l|)(T_{\overline{\alpha}})$ interchanged.

This concludes the proof of part (2).

To prove part (3) we proceed by a double induction on $m$ and $r$. If $m=1$ or $r=1$, there is nothing to prove.
So assume we have established the result for $m'<m$ and all $r'$ as well as for $m'=m$ and $r'\le r$. Now consider
an $(r+1)$-simplex $\overline{\beta}$ in $\mathfrak{K}_\bullet$ obtained by adding a new vertex $\lambda$ to an $r$-simplex
$\overline{\alpha}$. Suppose we are given an element $x\in\cap_{i=0}^r G_m^{ax}(\alpha_i)\cap G_m^{ax}(\lambda)$,
and let $y\in G_m(\lambda)$ be the unique preimage of $x$. It suffices to find a single vertex $\alpha_j$ such
that the preimage $z_j\in G_m(\alpha_j)$ of $x$ is the same element of $(W|\mathcal{NM}_k|\otimes W|\mathcal{NM}_l|)(m)$
as $y$. For then we would have $y=z_j\in G_m(\alpha_j)\cap G_m(\lambda)$ and by the second induction hypothesis
$z_j\in\cap_{i=0}^r G_m^{ax}(\alpha_i)$. Thus we would have $y=z_j\in\cap_{i=0}^r G_m(\alpha_i)\cap G_m(\lambda)$
and we would be done.

To establish this fact, we do the same case by case analysis as in the proof of part (2). In cases (a)-(d), that
analysis establishes that we can find a vertex $\alpha_j$ such that $y$ and $z_j$ have tree representatives
whose root nodes are identical, and all the axial images of the corresponding tree branches of $y$ and $z_j$
above the common root node are the same. By the primary induction hypothesis, these tree branches of $y$ and $z_j$
are also the same in $W|\mathcal{NM}_k|\otimes W|\mathcal{NM}_l|$. Hence $y=z_j$.

In cases (e)-(h), the analysis of the proof of part (2) shows that $y$ has a tree representative which can be
converted by series of compositions, either in $W|\mathcal{NM}_k|$ or in $W|\mathcal{NM}_l|$, into a tree which is
related to some $z_j\in G_m(\alpha_j)$ in the same way as in cases (a)-(d). By the same argument as above $y=z_j$.
In cases (i) and (j), that analysis shows $y$ has a tree representative which can be converted by an interchange
into a tree which is related to some $z_j\in G_m(\alpha_j)$ in the same way as in cases (a)-(d). Again we can
conclude that $y=z_j$.
\end{proof}

By combining the above results, we obtain the following more precise 
version of Proposition \ref{5_5}.

\begin{coro}\label{label_6}\begin{enumerate}
\item[(1)] $(W|\mathcal{NM}_k|\otimes W|\mathcal{NM}_l|)(m)$ is
the colimit of the diagram $G_m:\scI(m)\longrightarrow Top$.
\item[(2)] The maps in this diagram are all cofibrations.
\item[(3)] Each space $G_m(\overline{\alpha})$ has
a cellular decomposition over the poset 
${\mathcal{M}_k(S_{\overline{\alpha}})\times\mathcal{M}_k(T_{\overline{\alpha}})}$ and the maps in the diagram are
compatible with these cellular decompositions.
\item[(4)] $(W|\mathcal{NM}_k|\otimes W|\mathcal{NM}_l|)$ is axial and left factorial.
\end{enumerate}
\end{coro}

\begin{proof}
Part (1) is tautological. In view of Proposition \ref{label_5}, we have for any object $\overline{\alpha}$ in $\scI(m)$,
or equivalently any simplex $\overline{\alpha}=\{\alpha_0,\alpha_1,\dots,\alpha_r\}$ in $\mathfrak{K}_\bullet(m)$
$$G_m(\overline{\alpha})=\cap_{i=0}^r G_m(\alpha_i)\cong\cap_{i=0}^r G_m^{ax}(\alpha_i)
=(W|\mathcal{NM}_k|)(S_{\overline{\alpha}})\times(W|\mathcal{NM}_l|)(T_{\overline{\alpha}})$$
Thus the maps in the diagram are realizations of simplicial injections and thus are cofibrations. Part (3)
follows from parts (1) and (2) of Proposition \ref{label_5}. Axiality follows from part (3) of Proposition \ref{label_5},
whereas left factoriality immediately follows from the obvious observation that the product of two left factorial
monoids is left factorial.
\end{proof}

\section{The Grothendieck construction for interchanges}\label{sec11}

According to Corollary \ref{label_6}, $(W|\mathcal{NM}_k|\otimes W|\mathcal{NM}_l|)(m)$
is homeomorphic to an iterated colimit of diagrams of cofibrations. We can now
apply Proposition 5.2 of \cite{BFV} which identifies such an iterated colimit
as a simple colimit of the Grothendieck construction of the diagram of categories
which to each node $\overline{\alpha}$ of the outer diagram assigns the category
indexing the inner colimit at that node.

\begin{defi}\label{10_1}
Let $F:\scA\to\scC at$ be any functor. Then the Grothendieck
construction $\scA\int F$ is the category whose objects are pairs
$(A,B)$ with $A\in{\rm Ob}\scA$ and with $B\in{\rm Ob}F(A)$. A morphism
$(A_1,B_1)\longrightarrow(A_2,B_2)$ is a pair $(\alpha,\beta)$ where
$\alpha:A_1\to A_2$ and $\beta:F(\alpha)(B_1)\to B_2$.
\end{defi}

\begin{defi}\label{10_2}
In Corollary \ref{label_6} the appropriate diagram for the Gro\-then\-dieck
construction is the diagram of posets $\hat{G}'_m:\scI(m)\longrightarrow Cat$
given by 
$\hat{G}'_m(\overline{\alpha})= \scM_k(S_{\overline{\alpha}})\times\scM_l(T_{\overline{\alpha}})$.
We denote the Grothendieck construction $\scI(m)\int \hat{G}'_m$ by
$\scI(k,l)(m)$
\end{defi}

Thus we have established that

\begin{prop}\label{10_3} $(W|\mathcal{NM}_k|\otimes W|\mathcal{NM}_l|)(m)$ has a
cellular decomposition indexed by $\scI(k,l)(m)$.
This decomposition assigns to an object $(\overline{\alpha},\beta)$ the cell $G'_m(\overline{\alpha},\beta)$
in $G_m(\overline{\alpha})\cong(W|\mathcal{NM}_k|)(S_{\overline{\alpha}})\times(W|\mathcal{NM}_k|)(T_{\overline{\alpha}})$
indexed by $\beta\in\scM_k(S_{\overline{\alpha}})\times\scM_l(T_{\overline{\alpha}})$
\end{prop}

We now prove Proposition \ref{5_7}. We first need the following result.

\begin{lem}\label{lem_11.4}
Let $\overline{\alpha}$ be a simplex in $\mathfrak{K}_\bullet(m)$. Consider the set $\overline{\gamma}$
of all objects $\lambda\in\mathfrak{M}_2^{ab}(m)=\mbox{Ob}\scM_2^{ab}(m)$ such that
$G_m(\overline{\alpha})\subseteq G_m(\lambda)$. Then $\overline{\gamma}$ contains a unique minimal element
$L'(\overline{\alpha})$ with respect to the order relation
on the poset $\scM_2^{ab}(m)$. The mapping $\overline{\alpha}\mapsto L'(\overline{\alpha})$ defines a poset map
$L':\scI(m)\longrightarrow\scM_2^{ab}(m)$.
\end{lem}

\begin{proof}
By construction and Proposition \ref{empty}, $\overline{\gamma}$ is 
also a simplex in $\mathfrak{K}_\bullet(m)$ and we have $G_m(\overline{\alpha})=G_m(\overline{\gamma})$.
Hence wolog we may take $\overline{\alpha}=\overline{\gamma}$, i.e. that $G_m(\overline{\alpha})\subseteq G_m(\lambda)$
implies $\lambda\in\overline{\alpha}$.

Let $(S_{\overline{\alpha}},T_{\overline{\alpha}})$ be the pair of binodal trees corresponding to
$\overline{\alpha}$ in Proposition \ref{label_4}.  We define $L'(\overline{\alpha})$ as follows. In the
tree $S_{\overline{\alpha}}$ replace each node label $\bullet$ by node label $\boxtimes_1$ and
each node label $\circ$ by node label $\boxtimes_2$.  Let $L'(\overline{\alpha})$ be the element
of $\scM_2^{ab}(m)$ corresponding to this relabeled tree.

To check that $G_m(\overline{\alpha})\subseteq G_m(L'(\overline{\alpha}))$, it suffices to check that
$G_m(\overline{\alpha}\cap\{a,b,c\})\subseteq G_m(L'(\overline{\alpha})\cap\{a,b,c\})$ for all
$\{a<b<c\}\subseteq\{1,2,\dots,m\}$, which is easily verified. To check that $L'(\overline{\alpha})\le\alpha_j$
for all vertices $\alpha_j$ in $\overline{\alpha}=\overline{\gamma}$, it suffices to check that
$L'(\overline{\alpha})\cap\{a,b\}\le\alpha_j\cap\{a,b\}$ for all $\{a<b\}\subseteq\{1,2,\dots,m\}$, which is
even easier to verify.

Finally if $\overline{\beta}\subset\overline{\alpha}$ is an inclusion of simplices in $\mathfrak{K}_\bullet(m)$,
then it easily follows that $L'(\overline{\alpha})\cap\{a,b\}\le L'(\overline{\beta})\cap\{a,b\}$
for all $\{a<b\}\subseteq\{1,2,\dots,m\}$.  This implies $L'(\overline{\beta})\le L'(\overline{\alpha})$,
and so $L':\scI(m)\longrightarrow\scM_2^{ab}$ is a map of posets. (Recall that the poset structure
on $\scI(m)$ is opposite to inclusion of simplices, cf. Definition \ref{5_41}.)
\end{proof}

\begin{exam}\label{exam_11_4}
Consider the following 1-simplex in $\mathfrak{K}_\bullet(5)$
$$\overline{\alpha}=\{(1\boxtimes_1 2\boxtimes_1 5)\boxtimes_2 3\boxtimes_2 4, (1\boxtimes_1 5)\boxtimes_2 2\boxtimes_2 (3\boxtimes_1 4)\}$$
The pair of binodal trees associated to this simplex is
$$
(S_{\overline{\alpha}}, T_{\overline{\alpha}})=\left(
\begin{array}{cc}
\begin{tikzpicture}
[level distance=10mm,
every node/.style={fill=white,circle,inner sep=2pt},
level 1/.style={sibling distance=50mm,nodes={fill=white}},
level 2/.style={sibling distance=20mm,nodes={fill=black}},
level 3/.style={sibling distance=10mm,nodes={fill=white}}]
\node[white] {}[grow=north]
child {node[draw] (above node) {}
child {node[draw] {}
child {node {4}}
child {node{3}}
}
child {node[draw] {}
child {node {5}
}
child {node {2}}
child {node {1}}
}
};
\end{tikzpicture},
&
\begin{tikzpicture}
[level distance=8mm,
every node/.style={fill=white,circle,inner sep=2pt},
level 1/.style={sibling distance=50mm,nodes={fill=white}},
level 2/.style={sibling distance=20mm,nodes={fill=black}},
level 3/.style={sibling distance=10mm,nodes={fill=white}},
level 4/.style={sibling distance=10mm,nodes={fill=white}}]
\node[white] {}[grow=north]
child {node[draw] (above node) {}
child {node[draw] {}
child {node {4}}
child {node{3}}
}
child {node[draw] {}
child {node {2}}
child [missing] {}
child {node[draw] {}
child {node {5}}
child [missing] {}
child {node {1}}
}
}
};
\end{tikzpicture}
\end{array}
\right)
$$
We obtain $L'(\overline{\alpha})$ by taking the first component of this pair and replacing
$\bullet$ by $\boxtimes_1$ and $\circ$ by $\boxtimes_2$. This gives the tree
$$
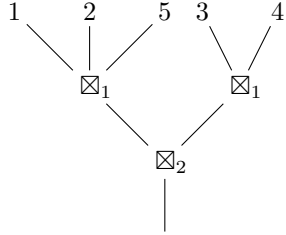
\begin{tikzpicture}
[level distance=10mm,
every node/.style={fill=white,rectangle,inner sep=2pt},
level 1/.style={sibling distance=50mm,nodes={fill=white}},
level 2/.style={sibling distance=20mm,nodes={fill=white}},
level 3/.style={sibling distance=10mm,nodes={fill=white}}]
\node[white] {}[grow=north]
child {node (above node) {$\phantom{1}\boxtimes_2$}
child {node {$\phantom{1}\boxtimes_1$}
child {node {4}}
child {node{3}}
}
child {node {$\phantom{1}\boxtimes_1$}
child {node {5}
}
child {node {2}}
child {node {1}}
}
};
\end{tikzpicture}
$$
which corresponds to the element $L'(\overline{\alpha})=(1\boxtimes_1 2\boxtimes_1 5)\boxtimes_2(3\boxtimes_1 4)\in\scM_2^{ab}(5 )$.
The set $\overline{\gamma}$ is the 3-simplex
$$\{(1\boxtimes_1 2\boxtimes_1 5)\boxtimes_2(3\boxtimes_1 4),(1\boxtimes_1 2\boxtimes_1 5)\boxtimes_2 3\boxtimes_2 4,
(1\boxtimes_1 5)\boxtimes_2 2\boxtimes_2 (3\boxtimes_1 4),(1\boxtimes_1 5)\boxtimes_2 2\boxtimes_2 3\boxtimes_2 4\}$$
\end{exam}

{\bf Proof of Proposition \ref{5_7}} We construct
$$L:\scI(k,l)(m)\longrightarrow\mathcal{M}_{k+l}(m)$$
so that the following diagram commutes:
$$\xymatrix{
\scI(k,l)(m)\ar[r]^L\ar[d]^P &\mathcal{M}_{k+l}(m)\ar[d]^{P'}\\
\scI(m)\ar[r]^{L'} &\scM_2^{ab}(m)}$$
Here $P:\scI(k,l)(m)=\scI(m)\int \hat{G}'_m\longrightarrow \scI(m)$ is the natural map in the
Grothendieck construction and $P':\mathcal{M}_{k+l}(m)\longrightarrow\scM_2^{ab}(m)$
is the poset map which replaces $\Box_i$ by $\boxtimes_1$ if $i\le k$ and by $\boxtimes_2$
if $i>k$.

The construction of $L$ goes as follows.  An object in $\scI(k,l)(m)$
is a pair $(\overline{\alpha},\beta)$, where $\overline{\alpha}$ is an object
in $\scI(m)=\mbox{Sd}\mathfrak{K}_\bullet(m)$, or equivalently a simplex
$\{\alpha_0,\alpha_1,\dots,\alpha_r\}$
in $\mathfrak{K}_\bullet(m)$, and $\beta$ is an object in the poset
$\scM_k(S_{\overline{\alpha}})\times\scM_l(T_{\overline{\alpha}})$.
As in the proof of Lemma \ref{lem_11.4} wolog we may assume that
$\overline{\alpha}=\overline{\gamma}$

There is then a unique morphism $\overline{\alpha}\longrightarrow L'(\overline{\alpha})$
in $\scI(m)$. Let $\beta'$ denote the image of $\beta$ under
the corresponding functor $\hat{G}'_m(\overline{\alpha})\longrightarrow \hat{G}'_m(L'(\overline{\alpha}))$.
Then $\beta'$ is an object in the poset
$\scM_k(S_{L'(\overline{\alpha})})\times\scM_l(T_{L'(\overline{\alpha})})$.
 But $S_{L'(\overline{\alpha})}$ and $T_{L'(\overline{\alpha})}$
are trees of the same shape with black nodes in $S_{L'(\overline{\alpha})}$ corresponding to
white nodes in $T_{L'(\overline{\alpha})}$ and vice versa. Now an element $\beta'$ in
$\scM_k(S_{L'(\overline{\alpha})})\times\scM_l(T_{L'(\overline{\alpha})})$ can be identified with
a pair of planar trees $(S',T')$ with $S'$ looking exactly like $S_{L'(\overline{\alpha})}$,
except that every black node is decorated with an object of $\scM_k$ of the
appropriate arity and with $T'$ looking exactly like $T_{L'(\overline{\alpha})}$,
except that every black node is decorated with an object of $\scM_l$ of the
appropriate arity. Putting together the decorations on $S'$ and $T'$, we obtain
a tree all of whose nodes are decorated with objects either of $\scM_k$
or of $\scM_l$. Regarding $\scM_k$ and $\scM_l$ as suboperads of $\scM_{k+l}$
as in paragraph \ref{4_7a}, we can interpret the resulting tree as an object
$\beta''$ in $\scM_{k+l}$, by considering the edges of the tree to be compositions
in the operad $\scM_{k+l}$. We define $L(\overline{\alpha},\beta)=\beta''$. This
is easily checked to be a map of posets.

To check Quillen's Theorem A for the functor $L$, we show that the over categories
$L/\alpha$ have contractible nerves for any object $\alpha$ in $\scM_{k+l}(m)$.
We note that we can obtain an object $\alpha'$ in $\overline{M}_2(m)$ by replacing
each $\square_i$ in $\alpha$ by $\boxtimes_1$ if $i\le k$ and by $\boxtimes_2$ if
$i>k$. This resulting object $\alpha'$ can also be regarded as an object in $\scI(m)$.
It is easy to see that there is a unique object $\beta'$ in
$\scM_k(S_{\alpha'})\times\scM_l(T_{\alpha'})$ such that $L(\alpha',\beta')=\alpha$
and that $\mbox{id}:L(\alpha',\beta')\longrightarrow\alpha$ is the terminal object
in $L/\alpha$.
\hfill\ensuremath{\Box}

\begin{exam}\label{exam_11_5}
We illustrate the construction $L$ with an example based on Example \ref{exam_11_4}.
Let us take $k=2$ and $l=3$. Now consider an object $(\overline{\alpha},\beta)\in\scI(2,3)(5)$.
The second component $\beta$ is an element in
$$\scM_2(S_{\overline{\alpha}})\times\scM_3(T_{\overline{\alpha}})\cong\scM_2(3)\times\scM_2(2)\times\scM_3(2)^3$$
The object  $(\overline{\alpha},\beta)$ can be represented by a pair of trees $(S_{\overline{\alpha},\beta},T_{\overline{\alpha},\beta})$,
obtained from the pair of binodal trees $(S_{\overline{\alpha}},T_{\overline{\alpha}})$
in Example \ref{exam_11_4} as follows.  We replace each black node $v$ of $S_{\overline{\alpha}}$
by the tree representing the object in $\scM_2(\In(v))$, specified by the component $\beta$, to 
obtain $S_{\overline{\alpha},\beta}$. Likewise we replace each black node $w$ of $T_{\overline{\alpha}}$
by the tree representing the object in $\scM_3(\In(w))$, specified by the component $\beta$, to obtain 
$T_{\overline{\alpha},\beta}$. Thus a typical object $(\overline{\alpha},\beta)$ might be
represented by the pair of trees
$$
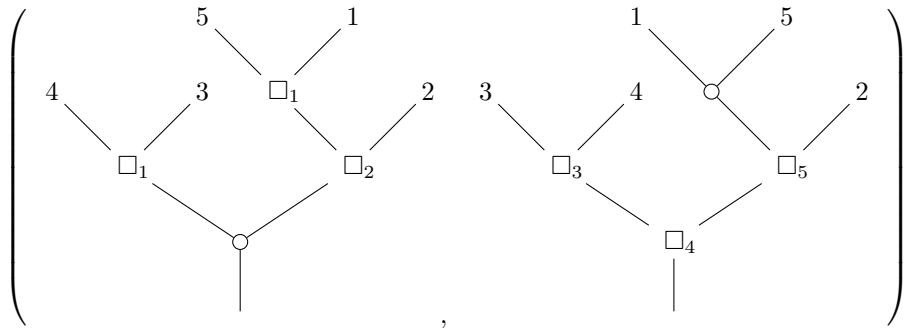
\left(
\begin{array}{cc}
\begin{tikzpicture}
[level distance=10mm,
every node/.style={fill=white,inner sep=2pt},
level 1/.style={sibling distance=50mm,nodes={fill=white}},
level 2/.style={sibling distance=30mm,nodes={fill=white}},
level 3/.style={sibling distance=10mm,nodes={fill=white}}]
\node[white] {}[grow=north]
child {node[draw,circle] (above node) {}
child {node{$\phantom{1}\square_2$}
child {node {2}}
child [missing] {}
child {node{$\phantom{1}\square_1$}
child {node {1}}
child [missing] {}
child {node {5}}
}
}
child {node {$\phantom{1}\square_1$}
child {node {3}}
child [missing] {}
child {node{4}}
}};
\end{tikzpicture},
&
\begin{tikzpicture}
[level distance=10mm,
every node/.style={fill=white,inner sep=2pt},
level 1/.style={sibling distance=50mm,nodes={fill=white}},
level 2/.style={sibling distance=30mm,nodes={fill=white}},
level 3/.style={sibling distance=10mm,nodes={fill=white}}]
\node[white] {}[grow=north]
child {node (above node) {$\phantom{1}\square_4$}
child {node  {$\phantom{1}\square_5$}
child {node {2}}
child [missing] {}
child {node [draw,circle] {}
child {node {5}}
child [missing] {}
child {node {1}}
}
}
child {node {$\phantom{1}\square_3$}
child {node {4}}
child [missing] {}
child {node{3}}
}
};
\end{tikzpicture}
\end{array}
\right)
$$

Then $(L'(\overline{\alpha}),\beta')$ is represented by

$$
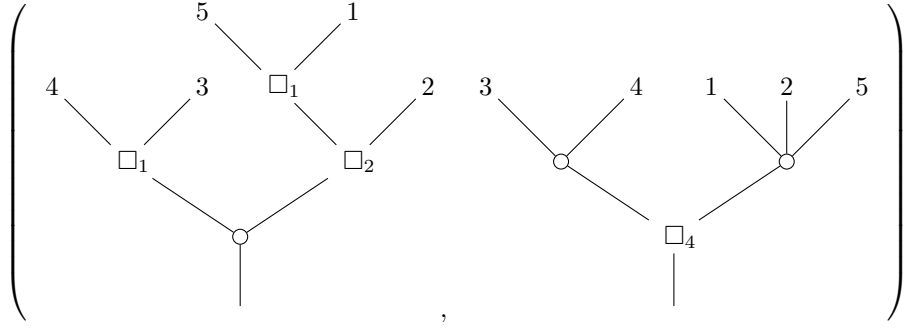
\left(
\begin{array}{cc}
\begin{tikzpicture}
[level distance=10mm,
every node/.style={fill=white,inner sep=2pt},
level 1/.style={sibling distance=50mm,nodes={fill=white}},
level 2/.style={sibling distance=30mm,nodes={fill=white}},
level 3/.style={sibling distance=10mm,nodes={fill=white}}]
\node[white] {}[grow=north]
child {node[draw,circle] (above node) {}
child {node{$\phantom{1}\square_2$}
child {node {2}}
child [missing] {}
child {node{$\phantom{1}\square_1$}
child {node {1}}
child [missing] {}
child {node {5}}
}
}
child {node {$\phantom{1}\square_1$}
child {node {3}}
child [missing] {}
child {node{4}}
}};
\end{tikzpicture},
&
\begin{tikzpicture}
[level distance=10mm,
every node/.style={fill=white,inner sep=2pt},
level 1/.style={sibling distance=50mm,nodes={fill=white}},
level 2/.style={sibling distance=30mm,nodes={fill=white}},
level 3/.style={sibling distance=10mm,nodes={fill=white}}]
\node[white] {}[grow=north]
child {node (above node) {$\phantom{1}\square_4$}
child {node  [draw,circle] {}
child {node {5}}
child {node {2}}
child {node {1}}
}
child {node [draw,circle] {}
child {node {4}}
child [missing] {}
child {node{3}}
}
};
\end{tikzpicture}
\end{array}
\right)
$$
and $L(\overline{\alpha},\beta)$ is represented by the tree
$$
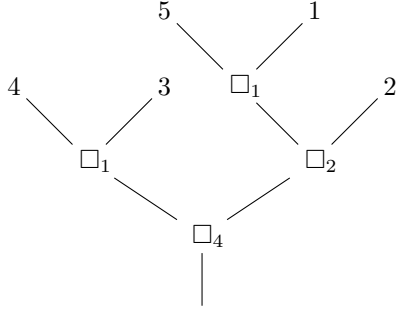
\begin{tikzpicture}
[level distance=10mm,
every node/.style={fill=white,inner sep=2pt},
level 1/.style={sibling distance=50mm,nodes={fill=white}},
level 2/.style={sibling distance=30mm,nodes={fill=white}},
level 3/.style={sibling distance=10mm,nodes={fill=white}}]
\node[white] {}[grow=north]
child {node (above node) {$\phantom{1}\square_4$}
child {node{$\phantom{1}\square_2$}
child {node {2}}
child [missing] {}
child {node{$\phantom{1}\square_1$}
child {node {1}}
child [missing] {}
child {node {5}}
}
}
child {node {$\phantom{1}\square_1$}
child {node {3}}
child [missing] {}
child {node{4}}
}};
\end{tikzpicture}
$$
Thus $L(\overline{\alpha},\beta)=(4\Box_1 3)\Box_4((5\Box_1 1)\Box_2 2)$.
\end{exam}

Thus we have established a chain of equivalences
$$
\begin{array}{rl}
(W|\mathcal{NM}_k|  \otimes W|\mathcal{NM}_l|)(m) &
\cong\mbox{colim}_{\scI(k,l)(m)}G'_m
 \stackrel{\simeq}{\longleftarrow}\mbox{hocolim}_{\scI(k,l)(m)}G'_m\\
& \stackrel{\simeq}{\longrightarrow} 
\mbox{hocolim}_{\scI(k,l)(m)}*
=|\mathcal{N}\scI(k,l)(m)|\\
& \stackrel{\simeq}{\longrightarrow}|\mathcal{NM}_{k+l}|(m)
\end{array}
$$
We now need to show that this chain of maps is compatible with the operad structures on the two
ends of the chain.  Thus we need to construct operad structures on all the intervening spaces and to
show that the maps are homomorphisms between these structures.

We begin with a simple construction.

\begin{lem}\label{lem_11_7}
 Let $\scC$ be an operad in $\Sets$ such that $\scC(m)$ is finite for each $m\ge0$. Let $\mathfrak{D}(m)$
be the simplex whose vertex set is  $\scC(m)$, regarded as an abstract simplicial complex in the classical sense.
Then the barycentric subdivisions $\mbox{Sd}\mathfrak{D}=\{\mbox{Sd}\mathfrak{D}(m)\}_{m\ge0}$
have a natural structure of an operad in $\Cat$.
\end{lem}

\begin{proof} Define the operad composition
 $$\mbox{Sd}\mathfrak{D}(k)\times\prod_{i=1}^k\mbox{Sd}\mathfrak{D}(m_i)
 \longrightarrow\mbox{Sd}\mathfrak{D}(m_1+m_2+\dots+m_k)$$
by
 $$\{a_i\}\circ(\oplus\{b_{ij}\}) = \{a_i\circ(b_{1j_1}\oplus b_{2j_2}\oplus\dots\oplus b_{kj_k})\}$$
In other words, take the barycenter of the simplex whose vertices are all possible
compositions of the vertices of the inputs.  The symmetric group actions are induced by the
symmetric group actions on the vertices.
It is straightforward to check that this defines an operad structure in $\Cat$.
\end{proof}

\begin{prop}\label{prop_11_8}
\begin{itemize}
\item[{\rm (i)}] $\scI=\{\scI(m)\}_{m\ge0}$ has  a $\Cat$-operad structure.
\item[{\rm (ii)}] The colimit decomposition of the operad $W|\mathcal{NM}_k|\otimes W|\mathcal{NM}_l|$
over $\scI$, given in Proposition \ref{5_5} is compatible with this operad structure.
\item[{\rm (iii)}] The functor $L':\scI\longrightarrow\scM_2^{ab}$ of Lemma \ref{lem_11.4}
is a map of $\Cat$-operads.
\end{itemize}
\end{prop}

\begin{proof} In Lemma \ref{lem_11_7} take the set operad $\scC$ to be $\mathfrak{M}_2^{ab}$.
Then by definition $\scI(m)$ is a subposet of $\mbox{Sd}\mathfrak{D}(m)$ (where we take the order
relation in both cases to be opposite to inclusion of faces).  Thus to prove (i), we have to show that
the operad structure on  $\mbox{Sd}\mathfrak{D}$ restricts to $\scI$.

Suppose that $\overline{\alpha}\in\scI(m)$ and that $\overline{\beta}_i\in\scI(n_i)$, for $i=1,2,\dots,m$.
Then it is clear from Definition \ref{5_41} that the operad composition
\begin{eqnarray*}
\lefteqn{(W|\mathcal{NM}_k|\otimes W|\mathcal{NM}_l|)(m)
\times\prod_{i=1}^m(W|\mathcal{NM}_k|\otimes W|\mathcal{NM}_l|)(n_i)}\\
&\longrightarrow &(W|\mathcal{NM}_k|\otimes W|\mathcal{NM}_l|)(n_1+n_2+\dots+n_m)
\end{eqnarray*}
restricts to
$$G(\overline{\alpha})\times\prod_{i=1}^m G(\overline{\beta}_i)\longrightarrow
\cap_{\lambda\in\overline{\alpha}\circ(\oplus_{i=1}^m\overline{\beta}_i)}G(\lambda),$$
where $\overline{\alpha}\circ(\oplus_{i=1}^m\overline{\beta}_i)$ denotes operad composition in
 $\mbox{Sd}\mathfrak{D}$.  It follows that 
$\cap_{\lambda\in\overline{\alpha}\circ(\oplus_{i=1}^m\overline{\beta}_i)}G(\lambda)$
is nonempty.  By Proposition \ref{5_5} it follows that
$\overline{\alpha}\circ(\oplus_{i=1}^m\overline{\beta}_i)\in\scI(n_1+n_2+\dots+n_k)$ and
$\cap_{\lambda\in\overline{\alpha}\circ(\oplus_{i=1}^m\overline{\beta}_i)}G(\lambda)
=G(\overline{\alpha}\circ(\oplus_{i=1}^m\overline{\beta}_i)).$
This proves (i) and also (ii), since the additional condition that the colimit decomposition
behaves appropriately under the symmetric group actions is obviously true.

Finally to prove (iii) we first note that since both $\scI$ and $\scM_2^{ab}$ are operads in the
category of posets, it suffices that $L'$  induces a map of operads on objects, which are operads
in $\Sets$.  To check this, we observe that this map decomposes as a composite
$$\mbox{Ob}\,\scI\stackrel{S}{\longrightarrow}\scB\scT
\stackrel{V}{\longrightarrow}\mbox{Ob}\,\scM_2^{ab}=\mathfrak{M}_2^{ab},$$
where $\scB\scT$ is the operad of binodal trees (cf. Definition \ref{scbt}),
$S(\overline{\alpha})=S_{\overline{\alpha}}$, and
$V$ replaces each node label $\bullet$ by node label $\boxtimes_1$ and each node label $\circ$
by node label $\boxtimes_2$.  Both $S$ and $V$ are easily seen to be operad maps.
\end{proof}

The following result establishes the desired operad equivalence between
$W|\mathcal{NM}_k|\otimes W|\mathcal{NM}_l|$ and $\mathcal{M}_{k+l}$ and shows that the
former is an $E_{k+l}$-operad.

\begin{theo}\label{thm_11_9}
\begin{itemize}
\item[{\rm (i)}] $\scI(k,l)=\{\scI(k,l)(m)\}_{m\ge0}$ has a $\Cat$-operad structure.
\item[{\rm (ii)}] The cellular decomposition of the operad $W|\mathcal{NM}_k|\otimes W|\mathcal{NM}_l|$
over $\scI(k,l)$, given in Proposition \ref{10_3} is compatible with this operad structure.
\item[{\rm (iii)}] The equivalence $L:\scI(k,l)\longrightarrow\mathcal{M}_{k+l}$ of Proposition \ref{5_7}
is a map of $\Cat$-operads.
\end{itemize}
\end{theo}

\begin{proof} We define the operad structure on $\mbox{Ob}\,\scI(k,l)$ by observing that
$$\mbox{Ob}\,\scI(k,l)\subset\mbox{Ob}\,\scI\times\mbox{Ob}\,(\scM_k)(\scB\scT)
\times\mbox{Ob}\,(\scM_l)(\scB\scT)$$
and noting that the operad structure on
$\mbox{Ob}\,\scI\times\mbox{Ob}\,(\scM_k)(\scB\scT)\times\mbox{Ob}\,(\scM_l)(\scB\scT)$
restricts to $\mbox{Ob}\,\scI(k,l)$.  This follows from the evident fact that 
$$S:\mbox{Ob}\,\scI\longrightarrow\scB\scT,\qquad T:\mbox{Ob}\,\scI\longrightarrow\scB\scT$$
given by $S(\overline{\alpha})=S_{\overline{\alpha}}$, $T(\overline{\alpha})=T_{\overline{\alpha}}$
respectively, are operad maps.  It is straightforward to check that this operad
structure on objects extends to a $\Cat$-operad structure on the poset operad $\scI(k,l)$.  This
proves part (i).

To prove part (ii), we note that according to Proposition \ref{prop_11_8} part (ii), operad composition
in $W|\mathcal{NM}_k|\otimes W|\mathcal{NM}_l|$ restricts to maps of the form
$$G(\overline{\alpha})\times\prod_{i=1}^m G(\overline{\beta}_i)\longrightarrow
G(\overline{\alpha}\circ(\oplus_{i=1}^m\overline{\beta}_i)).$$
We must show that these restrictions are compatible with the cellular decompositions on
both sides.  This follows from the commutativity of the diagram
$$\xymatrix{G(\overline{\alpha})\times\prod_{i=1}^m G(\overline{\beta}_i)
\ar[r]^(0.4){\cong}\ar[dd]
&\mbox{$\begin{array}{c}
W|\mathcal{NM}_k(S_{\overline{\alpha}})\times W|\mathcal{NM}_l(T_{\overline{\alpha}})\\
\times\prod_{i=1}^m W|\mathcal{NM}_k(S_{\overline{\beta_i}})\times W|\mathcal{NM}_l(T_{\overline{\beta_i}})\end{array}$}\ar[d]\\
&W|\mathcal{NM}_k(S_{\overline{\alpha}}\circ(\oplus_{i=1}^m S_{\overline{\beta_i}}))
\times W|\mathcal{NM}_l(T_{\overline{\alpha}}\circ(\oplus_{i=1}^m T_{\overline{\beta_i}}))
\ar@{=}[d]\\
G(\overline{\alpha}\circ(\oplus_{i=1}^m\overline{\beta}_i))\ar[r]^(0.3){\cong}
&W|\mathcal{NM}_k(S_{\overline{\alpha}\circ(\oplus_{i=1}^m\overline{\beta}_i)})
\times W|\mathcal{NM}_l(T_{\overline{\alpha}\circ(\oplus_{i=1}^m\overline{\beta}_i)})
}$$
Here the horizontal homeomorphisms are those given in Proposition \ref{label_5} and the right
hand vertical map is operad composition in the operad 
$(W|\mathcal{NM}_k|)(\scB\scT)\times (W|\mathcal{NM}_l|)(\scB\scT)$, which as noted in Remark
\ref{scPscA}, has a cellular decomposition over the poset operad 
$(\mathcal{M}_k)(\scB\scT)\times (\mathcal{M}_l)(\scB\scT)$.

Part (iii) is an immediate consequence of the definition of the operad structure on
$\scI(k,l)$ and the fact that $L'$ is a map of operads, as shown in Proposition \ref{prop_11_8}
part (iii).

\end{proof}

\newpage
\section{Appendix: Intersection table for Proposition \ref{8_5}}

\begin{center}
\begin{tabular}{|c|c|c|c|}
\hline
&$T_1$ &$T_2$ &$T_3$\\
\hline
\raisebox{-10pt}{1}
&$\xymatrix@=5pt@M=-1pt@W=-1pt{
1 &&2 &&3\\
\ar@{-}[ddddrr] &&\ar@{-}[dddd] &&\ar@{-}[ddddll]\\
\\
\\
\\
&&\circ\ar@{-}[dd]\\
\\
&&
}$
&\raisebox{-10pt}{$T_2$ arbitrary} &\raisebox{-10pt}{$T_2$}\\
\hline
\raisebox{-10pt}{2}
&$\xymatrix@=5pt@M=-1pt@W=-1pt{
1 &&2 &&3\\
\ar@{-}[ddddrr] &&\ar@{-}[dddd] &&\ar@{-}[ddddll]\\
\\
\\
\\
&&\bullet\ar@{-}[dd]\\
\\
&&
}$
&$\xymatrix@=5pt@M=-1pt@W=-1pt{
i &&j &&&k\\
\ar@{-}[ddr] &&\ar@{-}[ddl]&&&\ar@{-}[ddddll]\\
\\
&\bullet\ar@{-}[ddrr]\\
\\
&&&\bullet\ar@{-}[dd]\\
\\
&&&
}$
&$\xymatrix@=5pt@M=-1pt@W=-1pt{
i &&j &&&k\\
\ar@{-}[ddr] &&\ar@{-}[ddl]&&&\ar@{-}[ddddll]\\
\\
&\bullet\ar@{-}[ddrr]\\
\\
&&&\bullet\ar@{-}[dd]\\
\\
&&&
}$\\
\hline
\raisebox{-10pt}{3}
&$\xymatrix@=5pt@M=-1pt@W=-1pt{
1 &&2 &&3\\
\ar@{-}[ddddrr] &&\ar@{-}[dddd] &&\ar@{-}[ddddll]\\
\\
\\
\\
&&\bullet\ar@{-}[dd]\\
\\
&&
}$
&$\xymatrix@=5pt@M=-1pt@W=-1pt{
i &&j &&&k\\
\ar@{-}[ddr] &&\ar@{-}[ddl]&&&\ar@{-}[ddddll]\\
\\
&\bullet\ar@{-}[ddrr]\\
\\
&&&\circ\ar@{-}[dd]\\
\\
&&&
}$
&$\xymatrix@=5pt@M=-1pt@W=-1pt{
1 &&2 &&3\\
\ar@{-}[ddddrr] &&\ar@{-}[dddd] &&\ar@{-}[ddddll]\\
\\
\\
\\
&&\bullet\ar@{-}[dd]\\
\\
&&
}$\\
\hline
\raisebox{-10pt}{4}
&$\xymatrix@=5pt@M=-1pt@W=-1pt{
1 &&2 &&3\\
\ar@{-}[ddddrr] &&\ar@{-}[dddd] &&\ar@{-}[ddddll]\\
\\
\\
\\
&&\bullet\ar@{-}[dd]\\
\\
&&
}$
&$\xymatrix@=5pt@M=-1pt@W=-1pt{
i &&j &&&k\\
\ar@{-}[ddr] &&\ar@{-}[ddl]&&&\ar@{-}[ddddll]\\
\\
&\circ\ar@{-}[ddrr]\\
\\
&&&\bullet\ar@{-}[dd]\\
\\
&&&
}$
&$\xymatrix@=5pt@M=-1pt@W=-1pt{
i &&j &&&k\\
\ar@{-}[ddr] &&\ar@{-}[ddl]&&&\ar@{-}[ddddll]\\
\\
&\bullet\ar@{-}[ddrr]\\
\\
&&&\bullet\ar@{-}[dd]\\
\\
&&&
}$\\
\hline
\raisebox{-10pt}{5}
&$\xymatrix@=5pt@M=-1pt@W=-1pt{
i &&j &&&k\\
\ar@{-}[ddr] &&\ar@{-}[ddl]&&&\ar@{-}[ddddll]\\
\\
&\bullet\ar@{-}[ddrr]\\
\\
&&&\bullet\ar@{-}[dd]\\
\\
&&&
}$
&$\xymatrix@=5pt@M=-1pt@W=-1pt{
p &&q &&&r\\
\ar@{-}[ddr] &&\ar@{-}[ddl]&&&\ar@{-}[ddddll]\\
\\
&\bullet\ar@{-}[ddrr]\\
\\
&&&\circ\ar@{-}[dd]\\
\\
&&&
}$
&$\xymatrix@=5pt@M=-1pt@W=-1pt{
i &&j &&&k\\
\ar@{-}[ddr] &&\ar@{-}[ddl]&&&\ar@{-}[ddddll]\\
\\
&\bullet\ar@{-}[ddrr]\\
\\
&&&\bullet\ar@{-}[dd]\\
\\
&&&
}$\\
\hline
\raisebox{-10pt}{6}
&$\xymatrix@=5pt@M=-1pt@W=-1pt{
i &&j &&&k\\
\ar@{-}[ddr] &&\ar@{-}[ddl]&&&\ar@{-}[ddddll]\\
\\
&\bullet\ar@{-}[ddrr]\\
\\
&&&\bullet\ar@{-}[dd]\\
\\
&&&
}$
&$\xymatrix@=5pt@M=-1pt@W=-1pt{
i &&j &&&k\\
\ar@{-}[ddr] &&\ar@{-}[ddl]&&&\ar@{-}[ddddll]\\
\\
&\circ\ar@{-}[ddrr]\\
\\
&&&\bullet\ar@{-}[dd]\\
\\
&&&
}$
&$\xymatrix@=5pt@M=-1pt@W=-1pt{
i &&j &&&k\\
\ar@{-}[ddr] &&\ar@{-}[ddl]&&&\ar@{-}[ddddll]\\
\\
&\bullet\ar@{-}[ddrr]\\
\\
&&&\bullet\ar@{-}[dd]\\
\\
&&&
}$\\
\hline
\raisebox{-10pt}{7}
&$\xymatrix@=5pt@M=-1pt@W=-1pt{
i &&j &&&k\\
\ar@{-}[ddr] &&\ar@{-}[ddl]&&&\ar@{-}[ddddll]\\
\\
&\circ\ar@{-}[ddrr]\\
\\
&&&\bullet\ar@{-}[dd]\\
\\
&&&
}$
&$\xymatrix@=5pt@M=-1pt@W=-1pt{
i &&j &&&k\\
\ar@{-}[ddr] &&\ar@{-}[ddl]&&&\ar@{-}[ddddll]\\
\\
&\bullet\ar@{-}[ddrr]\\
\\
&&&\circ\ar@{-}[dd]\\
\\
&&&
}$
&$\xymatrix@=5pt@M=-1pt@W=-1pt{
i &&j &&&k\\
\ar@{-}[ddr] &&\ar@{-}[ddl]&&&\ar@{-}[ddddll]\\
\\
&\bullet\ar@{-}[ddrr]\\
\\
&&&\bullet\ar@{-}[dd]\\
\\
&&&
}$\\
\hline
\raisebox{-10pt}{8}
&$\xymatrix@=5pt@M=-1pt@W=-1pt{
i &&j &&&k\\
\ar@{-}[ddr] &&\ar@{-}[ddl]&&&\ar@{-}[ddddll]\\
\\
&\circ\ar@{-}[ddrr]\\
\\
&&&\bullet\ar@{-}[dd]\\
\\
&&&
}$
&$\xymatrix@=5pt@M=-1pt@W=-1pt{
i &&k &&&j\\
\ar@{-}[ddr] &&\ar@{-}[ddl]&&&\ar@{-}[ddddll]\\
\\
&\bullet\ar@{-}[ddrr]\\
\\
&&&\circ\ar@{-}[dd]\\
\\
&&&
}$
&$\xymatrix@=5pt@M=-1pt@W=-1pt{
i &&j &&&k\\
\ar@{-}[ddr] &&\ar@{-}[ddl]&&&\ar@{-}[ddddll]\\
\\
&\circ\ar@{-}[ddrr]\\
\\
&&&\bullet\ar@{-}[dd]\\
\\
&&&
}$\\
\hline
\raisebox{-10pt}{9}
&$\xymatrix@=5pt@M=-1pt@W=-1pt{
i &&j &&&k\\
\ar@{-}[ddr] &&\ar@{-}[ddl]&&&\ar@{-}[ddddll]\\
\\
&\bullet\ar@{-}[ddrr]\\
\\
&&&\circ\ar@{-}[dd]\\
\\
&&&
}$
&$\xymatrix@=5pt@M=-1pt@W=-1pt{
i &&k &&&j\\
\ar@{-}[ddr] &&\ar@{-}[ddl]&&&\ar@{-}[ddddll]\\
\\
&\bullet\ar@{-}[ddrr]\\
\\
&&&\circ\ar@{-}[dd]\\
\\
&&&
}$
&\raisebox{-8pt}{\parbox{140 pt}{nonempty but not of the form $(W\scB_\bullet)(T_3)$}} 
\\
\hline
\end{tabular}
\end{center}

\end{document}